\documentclass[reqno,12pt]{amsart}

\usepackage{amssymb} \usepackage{amsfonts} \usepackage{amsmath}
\usepackage{amsthm} \usepackage{epsfig} 
\usepackage{color}
\usepackage{amscd}
\usepackage[all]{xy}

\newtheorem{lemma}{Lemma}[section]
\newtheorem{teo}[lemma]{Theorem}
\newtheorem{prop}[lemma]{Proposition}
\newtheorem{cor}[lemma]{Corollary} 

\newtheorem{ex}[lemma]{Exercise} 

\theoremstyle{definition}
\newtheorem{defn}[lemma]{Definition}

\newtheorem{example}[lemma]{Example}

\theoremstyle{remark}
\newtheorem{rem}[lemma]{Remark}

\newcommand{\matQ} {\ensuremath {\mathbb{Q}}}
\newcommand{\matZ} {\ensuremath {\mathbb{Z}}}

\newcommand{\matRP} {\ensuremath {\mathbb{RP}}}
\newcommand{\matCP} {\ensuremath {\mathbb{CP}}}

\newcommand{\calS} {\ensuremath {\mathcal{S}}}

\newcommand{\calC} {\ensuremath {\mathcal{C}}}

\newcommand{\calB} {\ensuremath {\mathcal{B}}}
\newcommand{\calD} {\ensuremath {\mathcal{D}}}

\newcommand{\calP} {\ensuremath {\mathcal{P}}}

\newcommand{\nota} [1] {\caption{\footnotesize{#1}}}
\newcommand{\matr} [4] {\left(\begin{array}{@{}c@{\ }c@{}} #1 & #2 \\ #3 & #4 \\ \end{array} \right)}

\newcommand{\interior}[1]{{\rm int}(#1)}
\newfont{\Got}{eufm10 scaled 1200}

\newcommand{\PSL}{{\rm PSL}}

\newcommand{\cref}{c^{\rm ref}}
\newcommand{\isom}{\cong}
\newcommand{\gl}{{\rm gl}}
\newcommand{\rot}{{\rm rot}}

\newcommand{\timtil}{\begin{picture}(12,12)
\put(2,0){$\times$}\put(2,4.5){$\sim$}\end{picture}}

\author{Yuya Koda}
\address{Department of Mathematics
Hiroshima University, 1-3-1 Kagamiyama, Higashi - Hiroshima, 739-8526, Japan}
\email{ykoda at hiroshima-u dot ac dot jp}
\author{Bruno Martelli}
\address{Dipartimento di Matematica, Universit\`a di Pisa, Largo Pontecorvo 5, 56127 Pisa, Italy}
\email{martelli at dm dot unipi dot it}
\author{Hironobu Naoe}
\address{Mathematical Institute
Tohoku University, 6-3 Aramakiaza-Aoba, Aobaku, Sendai, 980-8578, Japan}
\email{hironobu dot naoe dot b2 at tohoku dot ac dot jp}

\thanks{The first author is supported in part by JSPS KAKENHI Grant Numbers
15H03620, 17K05254, 17H06463, and JST CREST Grant Number JPMJCR17J4. The third author is supported by JSPS KAKENHI Grant Number 17J02915.}

\title{Four-manifolds with shadow-complexity one}

\begin{document}

\begin{abstract}
We study the set of all closed oriented smooth 4-manifolds experimentally, according to a suitable complexity defined using Turaev's shadows. This complexity roughly measures how complicated the 2-skeleton of the 4-manifold is. 

We characterise here all the closed oriented 4-manifolds that have complexity at most one. They are generated by a certain set of 20 \emph{blocks}, that are some basic 4-manifolds with boundary consisting of copies of $S^2 \times S^1$, plus connected sums with some copies of $\matCP^2$ with either orientation. 

All the manifolds generated by these blocks are doubles. Many of these are doubles of 2-handlebodies and are hence efficiently encoded using finite presentations of groups. In contrast to the complexity zero case, in complexity one there are also plenty of doubles that are \emph{not} doubles of 2-handlebodies, like for instance $\matRP^3 \times S^1$.
\end{abstract}

\maketitle

\section{Introduction}
As pointed out by Donaldson in 2008:
\begin{quote}
\emph{For 4-dimensional manifolds a great deal is now known in the way of ``examples of phenomena that can occur'', but there is at present no kind of systematic picture, even at the most conjectural level}. \cite{Do}
\end{quote}
Ten years later, a systematic picture is still missing. In the present situation, it might be interesting to study the 4-dimensional smooth manifolds from an experimental viewpoint. An experimental approach usually consists in choosing a reasonable combinatorial description of the objects that we want to study (here, all closed smooth 4-manifolds) and then trying to classify them according to an increasing \emph{complexity}, that is some natural number that measures how complicated the combinatorial description is.

Historically, the experimental approach has been fundamental in the evolution of low-dimensional topology. Knots in the 3-sphere have been listed according to their crossing number since the very beginning of their study \cite{HoThWe}, and analogously 3-manifolds have been tabulated according to the number of tetrahedra in an (ideal) triangulation \cite{CaHiWe, Mat}. These tables have been used extensively to test conjectures and more generally to get an experimental grasp on the subject, often via beautiful and sophisticated computer programs like \emph{SnapPy} \cite{Sna}  or \emph{Regina} \cite{Reg}. Despite its success in dimension 3, the experimental approach is almost absent from the literature in dimension 4. The reason for that is, of course, that smooth 4-manifolds are combinatorially much more complicated than 3-manifolds and knots. Some recent interesting studies use crystallisations \cite{CC}, triangulations \cite{BS}, and trisections \cite{Me, MZ, ST}.

In this paper we pursue the experimentation started in \cite{Co, Ma:zero} via Turaev's \emph{shadows} \cite{Tu}. We stress the fact that we work in the smooth (equivalently, piecewise linear) category. Every 4-manifold is tacitly assumed to be smooth and oriented. Our main result is Theorem \ref{1:teo}, that characterises all the closed 4-manifolds with complexity at most 1.

\subsection{Outline}
We describe informally the main results of this paper. Let $M$ denote a closed smooth oriented 4-manifold. The complexity that we choose here is a natural number $c^*(M)$, called the \emph{connected shadow complexity} of $M$, that measures how complicated the 2-skeleton of $M$ is. The notion of 2-skeleton that we use is that of Turaev's \emph{shadow} \cite{Tu} and the theory is heavily inspired from Matveev's complexity of 3-manifolds \cite{Mat}. We postpone the rigorous definition to Section \ref{results:section}.

The complexity $c^*(M)$ has two nice features that are of fundamental importance. The first is that there are plenty of closed smooth four-manifolds $M$ with low complexity, in particular with $c^*(M)=0$ or 1. This holds because there are many 2-skeleta with low complexity: for instance, any plumbing of surfaces has complexity 0 after a little modification. We do not have to wait long to find interesting manifolds: they are already there from the very beginning, at complexity $c^*=0$ or $1$.

The second nice feature is that there are many manifolds with $c^*(M)=0$ or $1$, but \emph{not too many}: we discover a posteriori that a finite number of \emph{blocks} with boundary diffeomorphic to copies of $S^2 \times S^1$ is enough to generate precisely all of them. This allows us to study and to classify these manifolds at least in some cases, for instance those with finite fundamental group.



We now describe the 4-manifolds that we have found. A summary of our discoveries is drawn very schematically in Figure \ref{summary:fig}. We start by recalling a simple and very productive technique to build many closed four-manifolds.

\begin{figure}
\begin{center}
\includegraphics[width = 16 cm]{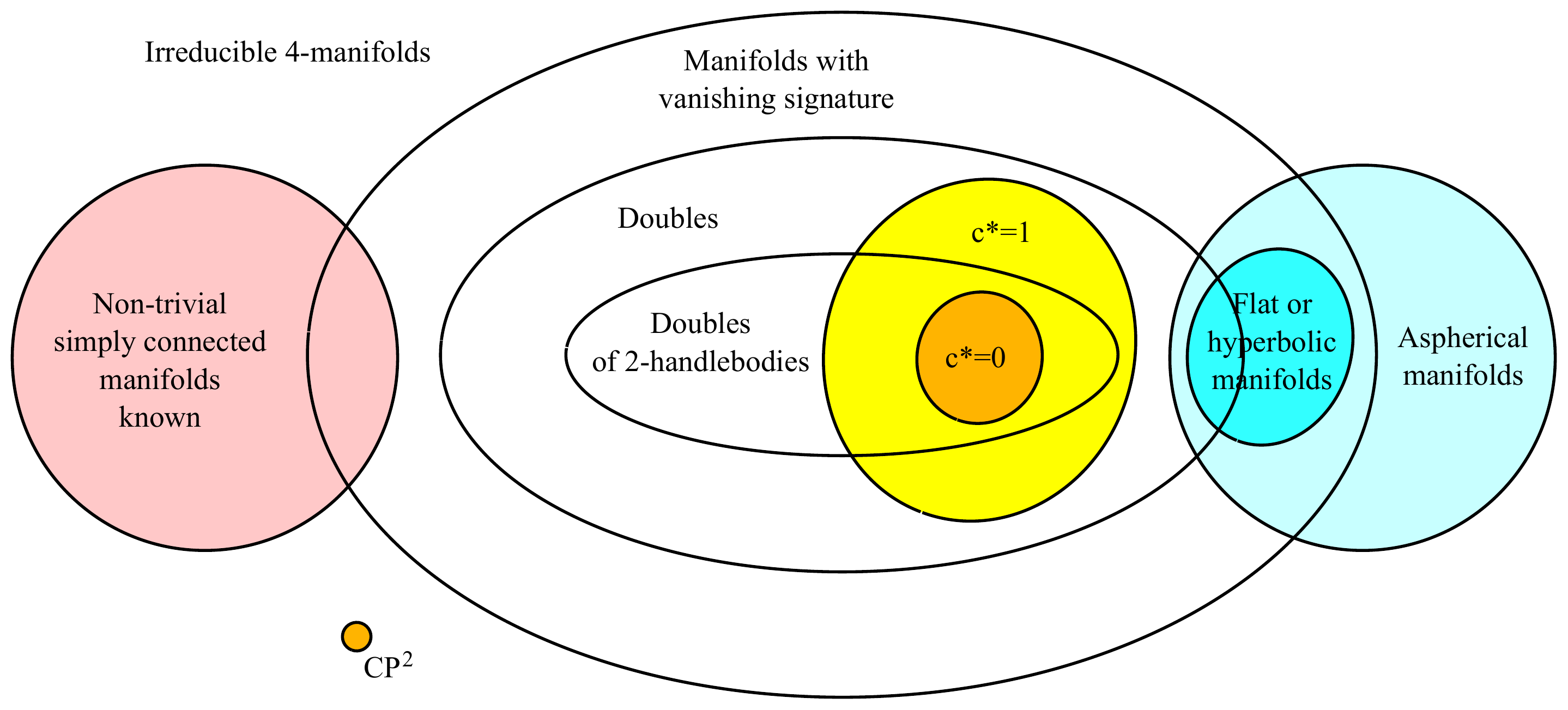}
\nota{A schematic picture that describes the irreducible orientable 4-manifolds $M$ with $c^*(M)=0$ and $1$ among all irreducible orientable 4-manifolds. Here the ``trivial'' simply connected manifolds are $S^4$, $S^2 \times S^2$, and $\matCP^2$, and they all have $c^*=0$. The complex projective plane $\matCP^2$ is the only irreducible manifold with $c^*\leq 1$ that is not a double. A \emph{2-handlebody} is a 4-manifold that decomposes into 0-, 1-, and 2-handles.}
\label{summary:fig}
\end{center}
\end{figure}

\subsection{From presentations to 4-manifolds} \label{presentations:subsection}
One of the simplest way to construct a closed four-manifold $M$ is the following: pick a finite presentation $\calP$ of some group $G$, construct a CW-complex $X$ from it in the standard way, thicken it to a smooth 5-manifold $W$, and take $M=\partial W$. We denote by $\calS(\calP)$ the set of all the four-manifolds $M$ constructed from $\calP$ in this way, considered up to diffeomorphism.

It is easy to see that $G=\pi_1(W) = \pi_1(M)$ and that there are finitely many ways to thicken $X$ to $W$, naturally parametrised by the set $H^2(X,\matZ/_{2\matZ})$, see \cite{Ma:zero, HaKrTe}. We get a spin manifold $M$ precisely in correspondence with the trivial element. In particular the set $\calS(\calP)$ is finite and it contains exactly one spin manifold. All the 4-manifolds in $\calS(\calP)$ share the same 3-skeleton, so their $\pi_1$, $\pi_2$, and all homology groups depend only on $\calP$. Moreover, if $\calP$ and $\calP'$ are related by Andrews-Curtis moves, then $\calS(\calP) = \calS(\calP')$. See \cite[Proposition 1.5]{Ma:zero} for a proof.

The manifolds in $\calS(\calP)$ are also precisely the doubles of the 4-dimensional thickenings of $X$. The 4-dimensional thickenings of $X$ are infinite in number and much more complicated to classify, so it is easier to adopt a 5-dimensional perspective here. See \cite[Lemma 2.7]{Ma:zero}. All these manifolds are mirrorable, so an orientation for them may be fixed arbitrarily.

A \emph{stabilisation} of $\calP$ is a move that consists of adding a new generator $g$ and two new relators $g,g$. We get a new presentation $\calP'$ that is not Andrew-Curtis equivalent to $\calP$ since they have different deficiency. The new CW-complex is of course $X' = X\vee S^2$. The 4-manifolds in the new set $\calS(\calP')$ are those in $\calS(\calP)$ plus  a connected sum with either $S^2\times S^2$ or $S^2 \timtil S^2$ to each.

We can assign to any presentation $\calP$ a \emph{connected complexity} $c^*(\calP)$ much in the same way as we do for the 4-manifolds, by estimating the minimum complexity of a 2-dimensional CW complex that represents $\calP$.

\subsection{Complexity zero.} \label{zero:subsection}
The closed smooth 4-manifolds $M$ with $c^*(M)=0$ were characterised in \cite{Ma:zero}. These are precisely those of the form
$$M = M'\#_h\matCP^n$$
where:
\begin{itemize}
\item $M'$ is any manifold constructed from some presentation $\calP$ with $c^*(\calP)=0$;
\item $h\in \matZ$ is any integer. 
\end{itemize}
When $h$ is negative, the symbol $M'\#_h\matCP^2$ indicates a connected sum with $|h|$ copies of $\overline{\matCP}^2$. In other words $M$ is obtained from $M'$ via some topological blow-ups. 

To characterise $M'$ we need to understand the presentations $\calP$ with $c^*(\calP)=0$. Recall that these are actually important only up to Andrews-Curtis moves. Consider the standard presentations of the cyclic and dihedral groups:
$$\calC_n = \langle a\ |\ a^n \rangle, \quad \calD_{2n} = \langle a, b\ |\ a^2, b^2, (ab)^n \rangle.$$
The unique spin 4-manifold in $\calS(\calC_n)$ is sometimes called the \emph{spun lens space} and is the boundary of the 5-manifold $(L(n,1)  \setminus B^3) \times D^2$. The sets $\calS(\calC_n)$ and $\calS(\calD_{2n})$ contain between 1 and 6 distinct manifolds depending on $n$, see \cite[Proposition 1.6]{Ma:zero}.
 
It is shown in \cite{Ma:zero} that the presentations $\calP$ of a finite group $G$ with $c^*(\calP)=0$ are precisely those that belong to the following three families:
$$\calC_{2^k}, \quad \calC_{3\cdot 2^k}, \quad \calD_{2^k}$$
and those obtained from them by stabilisation. Here $k\geq 0$ is any integer. From this we deduce that the four-manifolds $M$ with $c^*(M)=0$ and finite $\pi_1(M)$ are precisely those of the following type:
$$M=M'\#_k(S^2 \times S^2) \#_h(\matCP^2)\#_k(\overline{\matCP}^2)$$ 
where $M'$ is obtained by one of the three types of presentations just listed. In particular, the simply connected ones are just those that we would expect, namely:
$$S^4, \quad \#_h(S^2 \times S^2), \quad \#_h \matCP^2 \#_k \overline{\matCP}^2.$$
There are also many presentations $\calP$ of infinite groups $G$ with $c^*(\calP)=0$. For instance, we find all the standard presentations of the free groups $F_n$, of the surface groups $\pi_1(S_g)$, and of the products $F_n \times \matZ$ and $\matZ \times \matZ/_{2^n\matZ}$. The first three presentations give rise in particular to the following spin 4-manifolds:
$$\#_n(S^3 \times S^1), \qquad S_g \times S^2, \qquad \big(\#_n (S^2 \times S^1)\big) \times S^1.$$ 
All these 4-manifolds $M$ have $c^*(M)=0$. See Section \ref{examples:subsection}.

\subsection{Complexity one.}
The main result of this paper is a characterisation of all the 4-manifolds $M$ with $c^*(M)=1$. This is stated below as Theorem \ref{1:teo}. 

We describe this result here informally: roughly speaking, the set of all 4-manifolds $M$ with $c^*(M)=1$ may be subdivided into \emph{expected} and \emph{unexpected} manifolds. The expected manifolds are constructed with the presentation technique like in the $c^*=0$ case, from presentations $\calP$ with $c^*(\calP)=1$. These manifolds are all doubles of 2-handlebodies (plus possibly some topological blow-ups). The unexpected manifolds are of some new type: they are still doubles of some 4-manifolds with boundary (plus possibly some topological blow-ups), but they are sometimes \emph{not} doubles of 2-handlebodies. 
Recall that a 2-handlebody is a 4-manifold that decomposes into 0-, 1-, and 2-handles.
See the sketch in Figure \ref{summary:fig}.

We now describe these two sets of manifolds with more details.

\subsection{The expected manifolds.} \label{expected:subsection}
With the same techniques adopted in complexity zero, we can easily prove that 
among the manifolds $M$ with $c^*(M)=1$ we find all those that may be written as  
$$M = M'\#_h\matCP^n$$
where $M'\in \calS(\calP)$ for some presentation $\calP$ with $c^*(\calP)=1$. 
We show some examples. Consider the Von Dyck groups:
$$D(l,m,n) = \langle a,b\ |\ a^l, b^m, (ab)^n \rangle$$
and also the following groups defined by Coxeter in \cite{Co}:
$$ (l,m\ |\ n,k) = \langle a,b\ |\ a^l, b^m, (ab)^n, (ab^{-1})^k \rangle.$$
We easily prove in Section \ref{examples:subsection} that the following presentations $\calP$ have $c^*(\calP) \leq 1$:
$$\calC_n, \quad \calC_{5n}, \quad \calD_{2n}, \quad D(l,m,n), \quad (l,m\ |\ n,k) $$
as soon as the numbers $l,m,n,k$ are all of the type $2^a3^b$, that is if they are divisible only by 2 or 3. All the manifolds $M\in \calS(\calP)$ for these presentations $\calP$ have $c^*(M) \leq 1$.

In particular, the following presentations describe some finite groups:
$$D(2,3,3), \quad D(2,3,4), \quad (3, 3\ |\ 4,4), \quad (4, 9\ |\ 2, 3).$$
Coxeter showed in \cite{Co} that the last two groups are isomorphic to $\PSL(2,7)$ and $\PSL(2,17)$, two simple groups of order 168 and 2448 respectively. 

Among the 4-manifolds $M$ with $c^*(M)=1$ there are some having these two simple finite groups as fundamental groups. Their presence suggests that a complete classification of all manifolds $M$ with finite $\pi_1(M)$ and $c^*(M)=1$ is a more difficult goal to achieve than in complexity $c^*=0$. We do not attempt to attack this problem here and leave this for future work. 

\subsection{The unexpected manifolds}
Quite unexpectedly, there is much more than that. Among the manifolds $M$ with $c^*(M)=1$, we also find many boundaries of 5-dimensional thickenings of 3-dimensional CW complexes, that are allowed to have some 3-dimensional strata of a very controlled nature. 

Said with other words, we find some 4-manifolds $M$ that are doubles of compact 4-manifolds with boundary, but that are \emph{not} doubles of any 2-handlebody. The simplest example is the manifold
$$M = \matRP^3 \times S^1$$
that has $c^*(M)=1$. It is clearly the double of $\matRP^3 \times [0,1]$, but it cannot be the double of a 2-handlebody, see Proposition \ref{not:2:prop}. 

Note the interesting fact that we already encountered some manifolds with the same fundamental group as $\matRP^3 \times S^1$ in complexity zero, since we mentioned above that the presentation 
$$\calP = \langle a,b\ |\ a^2, [a,b] \rangle$$
has $c^*(\calP)=0$. Since $\calP$ is balanced, a 5-dimensional thickening $W$ of $\calP$ has $\chi(W) = 1$ and hence its boundary $M=\partial W$ has $\chi(M)=2$. By stabilising $k$ times we can construct manifolds $M$ with $\pi_1(M) = \matZ \times \matZ/_{2\matZ}$, $c^*(M)=0$, $\sigma (M)=0$, and $\chi(M) = 2+2k$ for every $k\geq 0$. However the manifold $\matRP^3 \times S^1$ has $\chi=0$ and $c^*=1$ and clearly cannot be constructed in this way. 

\subsection{Finitely many blocks}
There are plenty of manifolds in complexity one: there are more than we expected. Luckily, however, these manifolds still satisfy the same type of finiteness property that holds in complexity zero: all the manifolds $M$ with $c^*(M)\leq 1$ are obtained by gluing altogether some copies of finitely many \emph{blocks} along their boundaries, each diffeomorphic to $S^2 \times S^1$. There are 8 blocks with $c^*=0$ and 12 with $c^*=1$. Among the latter, eleven were expected and one was not: this last one is responsible for all the unexpected manifolds mentioned above. See Theorem \ref{1:teo}. 

\begin{figure}
\begin{center}
\includegraphics[width = 12 cm]{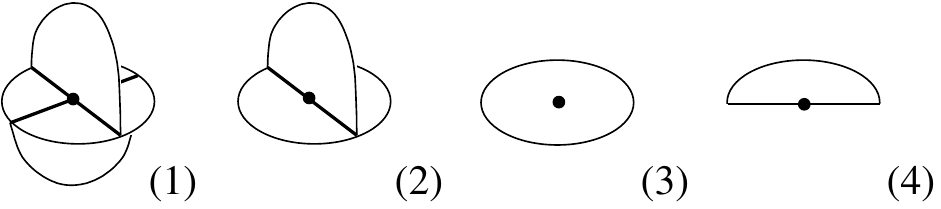}
\nota{Neighborhoods of points in a simple polyhedron with boundary.}
\label{shadow:fig}
\end{center}
\end{figure}

\section{Main results} \label{results:section}
We now expose more formally all the results proved in this paper. We start with the definitions of simple polyhedron and shadow complexity. We work in the piecewise-linear category.

\subsection{Simple polyhedron} 
A 2-dimensional compact polyhedron $X$ is \emph{simple} if every point has a star neighbourhood of one of the types shown in Figure \ref{shadow:fig}. 
The \emph{boundary} $\partial X$ is the union of all points of type (4), and is a union of circles. In this section we consider implicitly only simple polyhedra without boundary, except when mentioned otherwise. 

The points of type (1) are called \emph{vertices}. The points of type (2) and (3) form respectively some manifolds of dimension 1 and 2: their connected components are called respectively \emph{edges} and \emph{regions}. An edge is either an open interval or a circle, and a region is the interior of a compact surface with boundary. The \emph{singular part} $SX$ of $X$ is the union of all points of type (1) and (2). It is a 4-valent graph (possibly disconnected and/or with circular components).

\subsection{Shadow complexity}
A \emph{shadow} for a smooth closed orientable 4-manifold $M$ is a locally flat simple 2-dimensional polyhedron $X\subset M$ such that $M$ is obtained from a regular neighbourhood of $X$ by adding 3- and 4-handles. The polyhedron $X$ should be thought of as a 2-skeleton for $M$. This notion was introduced by Turaev in \cite{Tu} and is exposed with more details in Section \ref{definitions:section}.

As defined by Costantino \cite{Co}, the \emph{complexity} $c(X)$ of a shadow $X$ is its number of vertices, and the \emph{shadow complexity} $c(M)$ of a closed oriented smooth 4-manifold $M$ is the minimum complexity of a shadow $X$ for $M$. This definition is inspired by Matveev's complexity of 3-manifolds \cite{Mat88}.

The closed oriented 4-manifolds with shadow complexity zero were studied by the second author in \cite{Ma:zero}.
The original goal of the present research was to study those with complexity one. During the investigation we realised that it is quite natural in this setting to study a larger set of 4-manifolds, related to a more relaxed notion of complexity that we now introduce.

\begin{defn}
The \emph{connected complexity} $c^*(X)$ of a simple polyhedron $X$ is the maximum number of vertices that are contained in some connected component of $SX$. The \emph{connected shadow complexity} $c^*(M)$ of $M$ is the minimum connected complexity of a shadow $X$ for $M$.
\end{defn}

In particular, for a simple polyhedron $X$ we have:
\begin{itemize}
\item $c^*(X)=0$ $\Longleftrightarrow$ $SX$ consists of disjoint circles;
\item $c^*(X)=1$ $\Longleftrightarrow$ $SX$ consists of disjoint circles and 8-shaped graphs.
\end{itemize}

We clearly have $c^*(X) \leq c(X)$ and $c^*(X) = 0 \Longleftrightarrow c(X)=0$, which implies $c^*(M) \leq c(M)$ and $c^*(M) = 0 \Longleftrightarrow c(M)=0$ for every smooth closed oriented 4-manifold $M$.
In the rest of this introduction we work only with $c^*$ and disregard $c$. In this paper we consider only oriented 4-manifolds, although $c^*(M)$ clearly does not depend on the orientation for $M$.

Roughly speaking, the complexity $c^*(M)$ is a measure of how complicated the 2-skeleton of $M$ is. 

It is easy to prove that the set of manifolds $M$ with $c^*(M) \leq n$ is closed under connected sum, for any natural number $n$, see Proposition \ref{closed:prop}. This is probably not the case for the manifolds with $c(M) \leq n$, and it is one reason for preferring $c^*$ to $c$.

The main contribution of this paper is a characterisation of all the closed orientable smooth 4-manifolds $M$ with $c^*(M) = 1$. To introduce this result we first recall with some detail what is known about 4-manifolds with $c^*(M) = 0$, previously studied in \cite{Ma:zero}.

\subsection{Complexity zero} \label{zero:subsection}
Let a \emph{block} be a compact oriented smooth 4-manifold, possibly with boundary consisting of copies of $S^2 \times S^1$. Let $\calS$ be a finite set of blocks. A \emph{graph manifold generated by $\calS$} is any closed oriented 4-manifold obtained by taking some copies of elements in $\calS$ and glueing their boundaries in pairs, via any pairing and any orientation-reversing diffeomorphisms.
The name is of course inspired from Waldhausen's three-dimensional graph manifolds that are defined similarly as the set of all 3-manifolds generated by $D^2\times S^1$ and $P\times S^1$, where $P$ is a pair-of-pants.

The following theorem was proved in \cite{Ma:zero}.

\begin{figure}
\begin{center}
\includegraphics[width = 15 cm]{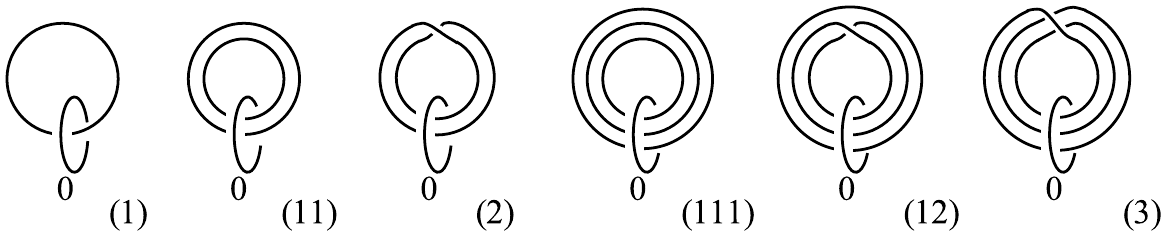}
\nota{Some links in $S^2 \times S^1 \subset S^3 \times S^1$, drawn as Kirby diagrams.}
\label{blocchi:fig}
\end{center}
\end{figure}

\begin{teo} \label{0:teo}
A closed oriented 4-manifold $M$ has $c^*(M)=0$ if and only if 
$$M=M' \#_h \matCP^2$$ 
where $h\in \matZ$ and $M'$ is is a graph manifold generated by the set 
$$\calS_0 = \big\{M_1, M_{11}, M_2, M_{111}, M_{12}, M_3, N_1, N_2, N_3\big\}.$$
If $M'\neq \#_k(S^3 \times S^1)$, this holds if and only if $M'$ is:
\begin{enumerate}
\item the double of a 4-dimensional thickening of a $X$ with $c^*(X)=0$, or equivalenty
\item the boundary of a 5-dimensional thickening of a $X$ with $c^*(X)=0$.
\end{enumerate} 
The symbol $X$ indicates a 2-dimensional simple polyhedron without boundary.
\end{teo}
As we said above, when $h$ is negative the symbol $M'\#_h\matCP^2$ indicates a connected sum with $|h|$ copies of $\overline{\matCP}^2$. The 9 manifolds in $\calS_0$ are in some vague sense among the simplest blocks one could reasonably construct: the manifolds $M_1, M_{11}, M_2, M_{111}, M_{12}, M_3$ are obtained from $S^3 \times S^1$ by drilling the corresponding links in Figure \ref{blocchi:fig}, that are contained in $S^2 \times S^1 \subset S^3 \times S^1$. In particular we get $M_1 = D^3 \times S^1$ and $M_{11} = S^2 \times S^1 \times [-1,1]$. The manifolds $N_1, N_2, N_3$ are:
$$N_1 = S^2 \times D^2, \quad N_2 = S^2 \times A, \quad N_3 = S^2 \times P.$$

Here $A$ and $P$ are an annulus and a pair-of-pants. Actually $M_{11}=N_2$, so there are in fact 8 blocks in $\calS_0$. Moreover $N_2 = N_1 \cup N_3$ so we could actually remove $N_2$ from the list and 7 blocks would suffice: we keep $N_2$ in $\calS_0$ only for aesthetic reasons.

All the blocks in $\calS_0$ are mirrorable, so we can fix their orientations arbitrarily. In particular all the blocks in $\calS_0$ have vanishing signature $\sigma = 0$ and therefore $\sigma(M')=0$, which gives $\sigma(M) = h$ in Theorem \ref{0:teo}.

The most relevant information contained in Theorem \ref{0:teo} is that finitely many manifolds are enough to generate precisely all the manifolds $M$ with $c^*(M) = 0$. 
Given how wild 4-manifolds can be, this is a very satisfactory picture: there are infinitely many manifolds with $c^*(M)=0$, but they are generated by finitely many ones.

Another important information is that every such $M$ is the boundary of a five-dimensional thickening of some $X$ with $c^*(X)=0$. The five-dimensional thickenings are better treated via presentations, as described in Section \ref{presentations:subsection}. We can switch from polyhedra to presentations and back, thanks to the nice 1-1 correspondence:
$$\left\{\!
\begin{array}{c}{\rm compact\ 2-dimensional\ polyhedra} \\ {\rm up\ to\ 3-deformation} \end{array}\!\right\} \longleftrightarrow
\left\{\!
\begin{array}{c}{\rm finite\ presentations} \\ {\rm up\ to\ Andrews-Curtis\ moves} \end{array}\!\right\}.$$
A \emph{3-deformation} on a 2-dimensional polyhedron $X$ is a simple homotopy that is a composition of expansions and collapses that involve only simplexes of dimension at most 3. See \cite{HoMe} for an introduction to this fascinating subject.

Following \cite{Ma:zero}, we define the \emph{connected complexity} of a presentation $\calP$ as the minimum connected complexity of a simple polyhedron $X$, possibly with boundary, that represents $\calP$. All the presentations $\calP$ of some finite group $G$ with $c^*(\calP)=0$ were classified in \cite{Ma:zero} and the results were described in Section \ref{zero:subsection}.

We make some other observations concerning the signature and the Euler characteristic of the manifolds with $c^*=0$. We note that the connected sum with $\matCP^2$ with either orientation is the only available tool for producing 4-manifolds with non-zero signature when $c^*=0$.
The Euler characteristic of the blocks in $\calS_0$ is always zero, except
$$\chi(N_1) = 2, \qquad \chi(N_3) = -2.$$

In Theorem \ref{0:teo}, obviously $\chi(M')$ is the sum of the Euler characteristics of the blocks involved to construct $M'$, and $\chi(M) = \chi(M')+|h|$. The block $N_3$ contributes negatively to the Euler characteristic, while $N_1$ and $\matCP^2$ contribute positively.

The set of all closed 4-manifolds with $c^*(M)=0$ is closed under connected sum and finite covering. A manifold with $c^*(M)=0$ is never aspherical; see \cite{Ma:zero}.

\subsection{Complexity one}
We are now ready to expose the main result of the paper, that is a theorem analogous to Theorem \ref{0:teo} for the case $c^*(M) \leq 1$. 

\begin{teo} \label{1:teo}
A closed oriented 4-manifold $M$ has $c^*(M)\leq 1$ if and only if 
$$M=M' \#_h \matCP^2$$ 
where $h\in \matZ$ and $M'$ is equivalently of one of these types:
\begin{enumerate}
\item any graph manifold generated by the set 
$$\calS_1 = \calS_0 \cup \big\{M_1^1, \ \ldots,\ M_{12}^1 \big\};$$
\item the double of a 4-dimensional thickening of a $\bar X$ with $c^*(\bar X) \leq 1$.
\end{enumerate} 
The symbol $\bar X$ denotes a polyhedron of dimension $2$ or $3$, that consists of a simple 2-dimensional one $X$ with $c^* \leq 1$ that may have non-empty boundary, plus some copies of $\matRP^3$, each attached to a component of $\partial X$ along any projective line $l\subset \matRP^3$.
\end{teo}

\begin{figure}
\begin{center}
\includegraphics[width = 12.5 cm]{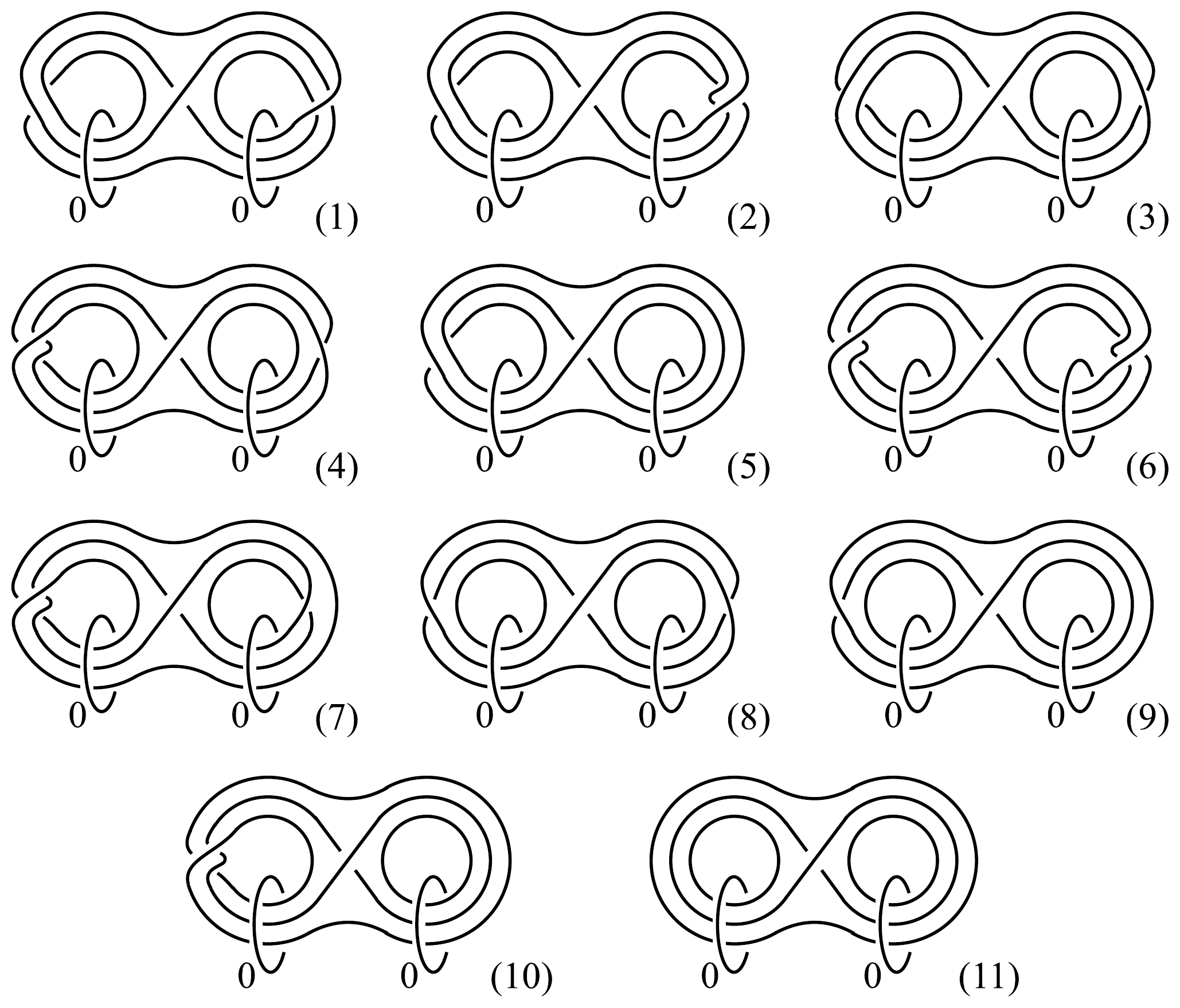}
\nota{Eleven links in $\#_2(S^2 \times S^1) \subset \#_2(S^3 \times S^1)$.}
\label{M:fig}
\end{center}
\end{figure}

We now describe the 12 additional blocks in $\calS_1$. The blocks 
$$M_1^1,\ \ldots,\ M_{11}^1$$ 
are obtained from $\#_2(S^3 \times S^1)$ by drilling the corresponding links in $\#_2(S^2 \times S^1) \subset \#_2(S^3 \times S^1)$ shown in Figure \ref{M:fig}. The last block 
$$M_{12}^1$$ 
is the result of drilling $\matRP^3 \times S^1$ along the curve $l \times \{{\rm pt} \}$ where $l\subset \matRP^3$ is any projective line. The manifolds $M_{i}^1$ are mirrorable for all $i=1,\ldots,12$, so we choose any orientation for them. All the blocks in $\calS_1$ have $\sigma =0$, so again we get $\sigma(M')=0$ and $\sigma(M)=h$.

As in Theorem \ref{0:teo}, the most relevant information contained in Theorem \ref{1:teo}-(1) is that a finite set of blocks is enough to generate precisely all the manifolds $M$ with $c^*(M)\leq 1$. We have thus discovered that this satisfactory picture in complexity zero is also predominant in complexity one.

We remark that the blocks $M_1^1, \ldots, M_{11}^1$ were somehow expected from the beginning of our examination, but the last one $M_{12}^1$ was not, and the presence of this additional block has some important consequences. It produces a more involved statement in Theorem \ref{1:teo}-(2), where $\bar X$ is a polyhedron that may contain both 2- and 3-dimensional strata. The manifold $M'$ obtained as a double of a thickening of $\bar X$ is still the double of a compact 4-manifold with boundary as in complexity zero, but $M'$ is \emph{not} necessarily the double of a 2-handlebody. For instance, we discover that the manifold
$$M = \matRP^3 \times S^1$$
has $c^*(M)=1$, it is the double of $\matRP^3 \times [0,1]$, but it cannot be the double of a 2-handlebody, see Proposition \ref{not:2:prop}. This is the most important novelty in complexity one that we could notice.

The manifolds generated by all the blocks in $\calS_1 \setminus \{ M_1^{12}\}$ are precisely those that are constructed from a presentation $\calP$ with $c^*(\calP) \leq 1$. Among these we find the examples already described in Section \ref{expected:subsection}. If we use also $M_1^{12}$ we find more manifolds, like $\matRP^3 \times S^1$, that cannot be constructed from any presentation.

We now study the Euler characteristic of $M$. We have $\chi(M_i^1)=-2$ for $i=1,\ldots, 11$, and $\chi(M_{12}^1)=0$. As we pass from complexity zero to one we seem to find a greater predominance of manifolds $M$ with $\chi(M)<0$. Note that the 4-manifolds $M$ that have been studied most in the literature are either simply connected, or aspherical, or symplectic, and in all the known cases they have $\chi(M)\geq 0$ (except some ruled surfaces blown up at some points, that may be symplectic with $\chi <0$).

From Theorem \ref{1:teo}-(2) it is easy to deduce the following.

\begin{teo}
There are no aspherical closed 4-manifolds $M$ with $c^*(M) \leq 1$.
\end{teo}

\subsection{Conclusions}
The sketch in Figure \ref{summary:fig} summarises our present knowledge.
The ecosystem formed by all the 4-manifolds with $c^* \leq 1$ is still dominated by doubles, and the connected sum with some copies of $\matCP^2$ with either orientation is still the only available tool to construct manifolds with $\sigma \neq 0$. However, as opposite to the $c^*=0$ case, there are many doubles in $c^*=1$ that are not doubles of 2-handlebodies. The ecosystem has thus enlarged considerably, although many important species are still missing. There are certainly no manifolds of the following types among those with $c^* \leq 1$:
\begin{itemize}
\item manifolds with even intersection form and non-zero signature, in other words with intersection form $mE_8 \oplus n H$ with $m\neq 0$;
\item manifolds with $\sigma \neq 0$ that are not of the form $M'\#_h \matCP^2$ with $\sigma(M') =0$;
\item manifolds with $\sigma = 0$ that are not doubles;
\item aspherical manifolds.
\end{itemize}
The techniques introduced here are quite general and could be used in principle to attack the $c^*=2$ case; however it is impossible to predict how the computational complexity will grow as we pass from $c^*=1$ to $c^*=2$. As shown in this paper, to prove Theorem \ref{1:teo} we need to (1) understand the exceptional Dehn fillings of 11 hyperbolic 3-manifolds, and (2) produce by hand a considerable amount of combinatorial moves between shadows, that are strictly necessary to simplify about 100 possible local configurations at the very end of the paper. Unfortunately the second step was done entirely by hand and not by computer. A more computer-assisted strategy would be very much desirable, but we do not know how to implement it for the moment.

We also stress that we do not know if for every value of $n$ there are only finitely many blocks that generate all the closed manifolds $M$ with $c^*(M) \leq n$. For the moment we only know the following, proved in Section \ref{finiteness:subsection} using Freedman's recent notion of \emph{group width} \cite{Fr}.

\begin{teo}
There are closed 4-manifolds $M$ with arbitrarily large $c^*(M)$.
\end{teo}

Given the length of this paper and the high level of technicalities already present, we concentrate ourselves here in proving Theorem \ref{1:teo} and postpone the analysis of its consequences for future work. For instance, it would be interesting to classify all manifolds with $c^*=1$ having trivial, or maybe finite, fundamental group as it was done in \cite{Ma:zero}. This problem translates into the classification of all the presentations $\calP$ of some finite group $G$ with $c^*(\calP)=1$,
up to Andrews-Curtis moves. We do not do this here. In the simply connected case we do not expect any new manifold, since all the presentations involved should reasonably be Andrews-Curtis equivalent to a bouquet of spheres (and of course we would have no clue on how to prove the contrary).

\subsection{Structure of the paper}
We introduce shadows with more details in Section \ref{definitions:section}.
In Section \ref{one:section} we describe a combinatorial way to encode any shadow $X$ with $c^*(X) \leq 1$ via some decorated graph. Theorem \ref{1:teo} is then restated again in Section \ref{main:section}. The constructive part of the theorem consists of showing that every graph manifold $M$ generated by $\calS_1$ has $c^*(M) \leq 1$. This is the easy part of the proof and is proved in Section \ref{constructive:section}.

The hard part in the proof of Theorem \ref{1:teo} is to show that, conversely, every manifold $M$ with $c^*(M)\leq 1$ is generated by $\calS_1$ plus some copies of $\matCP^2$. To show this, we first study in Section \ref{tori:section} how this relates to the 3-dimensional problem of studying some decompositions of $\#_h(S^2 \times S^1)$ along tori. It is then crucial to examine the exceptional Dehn fillings of 11 hyperbolic manifolds in Section \ref{exceptional:section}. Finally, we introduce many moves on shadows in Section \ref{moves:section} and then use them to finally conclude the proof of Theorem \ref{1:teo} via a long case-by-case analysis in Section \ref{proof:section}.

\section{Shadows} \label{definitions:section}

We introduce here Turaev's shadows, following \cite{Tu}.

\subsection{Simple polyhedra}
As stated in Section \ref{results:section}, we let a \emph{simple polyhedron} be a compact polyhedron $X$ where every point has a star neighborhood of one of the types shown in Figure \ref{shadow:fig}. We will henceforth allow the presence of boundary $\partial X$ except when it is forbidden explicitly. The terms \emph{vertex}, \emph{edge}, and \emph{region} were defined in Section \ref{results:section}.

\subsection{Odd and even regions}
Let $f$ be a region of a simple polyhedron $X$. We denote by $\partial f$ the boundary of the abstract closure of $f$. The polyhedron $X$ induces an interval bundle on every component of $\partial f$ that is not contained in $\partial X$. The interval bundle may be either untwisted (an annulus) or twisted (a M\"obius strip). The region $f$ is \emph{even} or \emph{odd} depending on the parity of the number of twisted bundles on $\partial f$.

\begin{rem}
If $X$ is a spine of a 3-manifold, then all these bundles are necessarily untwisted, and hence in particular all the regions are even. The existence of a twisted bundle on some region is in fact a complete obstruction for $X$ to have a 3-dimensional thickening. \end{rem}

\subsection{Shadows} \label{shadows:subsection}
Following Turaev \cite{Tu}, a \emph{shadow} is a simple polyhedron $X$ without boundary decorated with \emph{gleams}. A gleam is a half-integer $\gl(f)$ attached to each region $f$ of $X$, with the requirement that $\gl(f)$ is an integer if and only if $f$ is even. 

As proved by Turaev, a shadow $X$ determines a 4-dimensional thickening $N(X)$, that is an oriented smooth 4-manifold $N(X)$ with boundary, that contains $X$ in its interior and that collapses onto $X$. Moreover, the gleams of $X$ are intrinsically determined by the embedding of $X$ in $N(X)$.

If $X$ is a surface, then $N(X)$ is a disc bundle over $X$ and the gleam of $X$ is just the self-intersection number of $X$, that equals the Euler number of the bundle. In general, the gleam $\gl(f)$ is the self-intersection of $f$ in some precise sense.

Let a \emph{k-handlebody} be an oriented 4-manifold made of handles of index $\leq k$. It turns out that $N(X)$ is always a 2-handlebody and conversely every 2-handlebody is obtained from a (non-unique) shadow $X$ in this way.

\subsection{Homology}
Homology computations with shadows are particularly simple. Let $X$ be a shadow and $N(X)$ its thickening. The inclusion $i\colon X \hookrightarrow N(X)$ is a homotopy equivalence and hence induces isomorphisms in homology. 

As in every simple polyehdron, each class $\alpha \in H_2\big(X, \matZ/_{2\matZ}\big)$ is naturally represented as a closed subsurface of $X$, and vice-versa.
Moreover, the intersection form $\langle \alpha, \beta \rangle$ is the parity of the sum of the gleams of the regions contained in the intersection of the two surfaces (this sum is always an integer). The second Stiefel-Whitney class $w_2(N(X))\in H^2\big(N(X), \matZ/_{2\matZ}\big)$ corresponds to the class $i^*(w_2(N(X)))\in H^2\big(X, \matZ/_{2\matZ}\big)$ that assigns to any closed subsurface the parity of the sum of its gleams. The 4-manifold $N(X)$ is spin if and only if this sum is always even.

The homology with integer coefficients is read from $X$ similarly. Every class in $H_2(X,\matZ)$ is uniquely determined as a sum $\sum a_if_i$ of oriented regions $f_i$ with integer weights $a_i$ which sum to zero at every edge $e$ (if we change the orientation of $f_i$ then $a_i$ changes its sign, and we sum the 3 regions adjacent to $e$ with matching orientations). The intersection form is easily calculated using the formula
$$\left\langle \sum a_if_i, \sum b_jf_j \right\rangle = \sum a_kb_k\gl(f_k)$$
where $\gl(f)$ is the gleam of $f$.

\subsection{Shadows of closed 4-manifolds}
We denote by $\#_h(S^2\times S^1)$ the connected sum of $h\geq 0$ copies of $S^2\times S^1$. When $h=0$ we mean $S^3$. 

Let $X$ be a shadow and $N(X)$ be its thickening. If $\partial N(X) \isom \#_h(S^2\times S^1)$ for some $h\geq 0$, we can add some 3- and 4-handles to $N(X)$ and obtain an oriented closed smooth 4-manifold $M$. By a famous theorem of Laudenbach and Poenaru \cite{LaPo} there is an essentially unique way to attach 3- and 4-handles, so the closed oriented manifold $M$ is fully determined by $N(X)$, and hence by the shadow $X$ alone. We say that \emph{$X$ is a shadow of $M$}.

For instance, the 2-sphere $X=S^2$ with gleam 0, 1, or $-1$ is a shadow of $S^4$, $\matCP^2$, or $\overline{\matCP}^2$ respectively. In the first case $\partial N(X) = S^2\times S^1$ and we attach a 3- and a 4-handle, in the two other cases $\partial N(X) = S^3$ and we only attach a 4-handle.

Every closed oriented 4-manifold $M$ has a shadow $X$, which is however not unique.

\begin{rem}
A shadow complexity can also be defined on 4-manifolds with boundary. This notion was investigated in \cite{Na0, Na, Na2}. We will not study this version here.
\end{rem}

\section{Simple polyhedra with connected complexity one} \label{one:section}
We describe a combinatorial notation for treating simple polyhedra with connected complexity one. This discussion is also useful to understand the main theorem of this paper and in particular where do the blocks $M_1^1, \ldots, M_{11}^1$ come from.

\subsection{Decomposition of a simple polyhedron into pieces}
Let $X$ be a simple polyhedron, possibly with non-empty boundary, with connected complexity one. Each component $C$ of $SX$ is either a circle or a 8-shaped graph. In the first case, the regular neighbourhood of $C$ in $X$ is a simple polyhedron with boundary, homeomorphic to one of the pieces
$$Y_{111}, \quad Y_{12}, \quad Y_3$$
shown in Figure \ref{Y:fig}. In particular $Y_{111} \isom Y\times S^1$, where $Y$ is the cone over 3 points, the polyhedron $Y_{12}$ is a M\"obius strip with an annulus attached to its core, and $Y_3$ is an annulus that winds 3 times around a circle. The boundary of the piece consists of 3, 2, or 1 circle, respectively.

\begin{figure}
\begin{center}
\includegraphics[width = 7 cm]{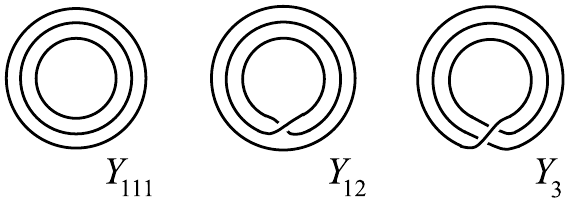}
\nota{The 3 possible regular neighbourhoods of circular components of $SX$.}
\label{Y:fig}
\end{center}
\end{figure}

If $C$ is a 8-shaped graph, that is a bouquet of two circles, a simple analysis shows that its regular neighbourhood is one of the types
$$X_1, \quad X_2, \quad \ldots, \quad X_{11}$$
drawn in Figure \ref{Xall:fig}. It is a simple polyhedron with boundary, and the boundary consists of a number of circles that ranges from 1 to 4.

\begin{figure}
\begin{center}
\includegraphics[width = 12.5 cm]{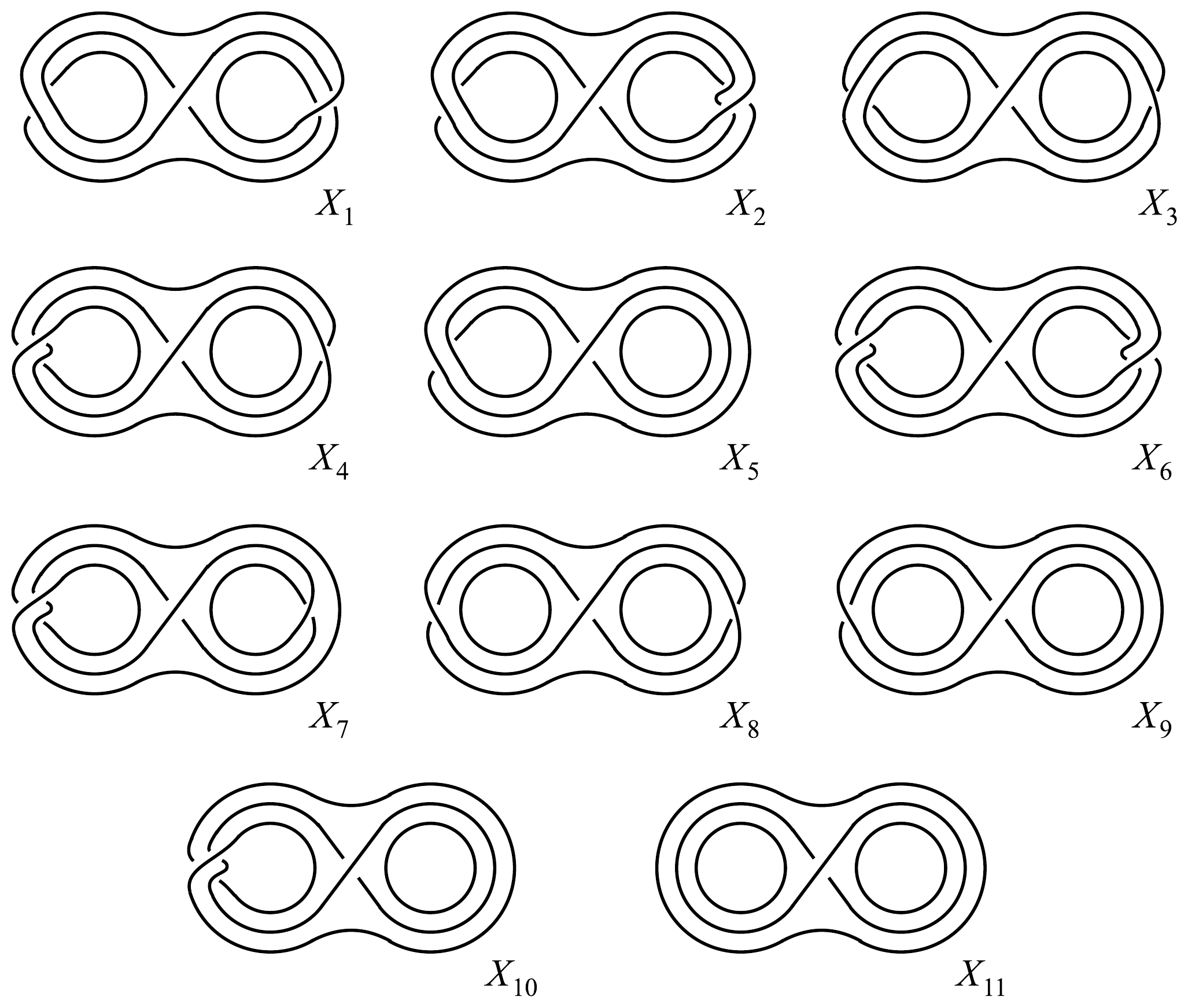}
\nota{The 11 possible regular neighbourhoods of a 8-shaped component of $SX$.}
\label{Xall:fig}
\end{center}
\end{figure}

Every region in $X$ is a surface and hence decomposes into discs, pairs-of-pants, and M\"obius strips. Summing up, we have discovered the following.

\begin{prop}
Every simple polyhedron $X$ with connected complexity one decomposes (along circles contained in some regions) into pieces homeomorphic to:
$$D,\ P,\ Y_2,\ Y_{111},\ Y_{12},\ Y_3,\ X_1, \ \ldots,\  {\rm or}\ X_{11}.$$ 
Here $D$ is a disc, $P$ is a pair of pants, and $Y_2$ a M\"obius strip. 
\end{prop}

Each piece is a simple polyhedron with boundary. We use the notation $Y_2$ for the M\"obius strip because it is somehow coherent with the symbols $Y_{111}, Y_{12}, Y_3$. Note that $D,P,Y_2$ are surfaces, while all the regions in the other pieces $Y_{111}, Y_{12}, Y_3, X_1, \ldots, X_{11}$ are annuli.

\subsection{Encoding graph} \label{encoding:graph:subsection}

\begin{figure}
\begin{center}
\includegraphics[width = 13 cm]{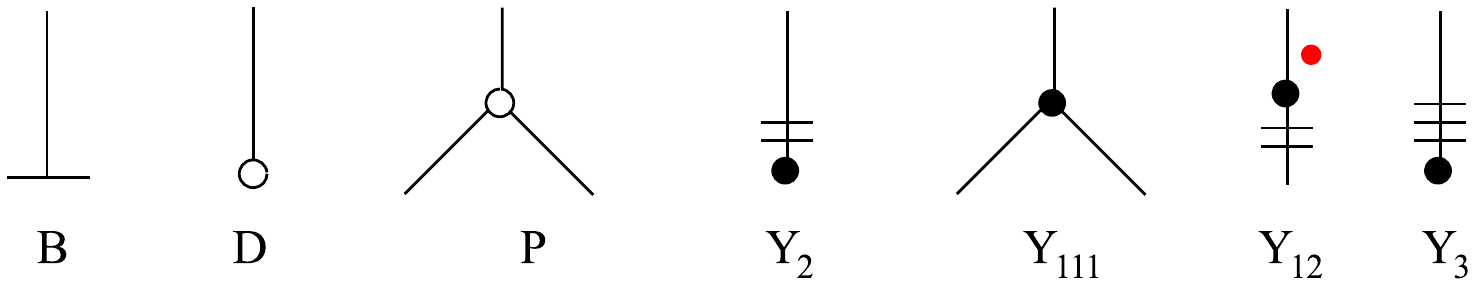}
\nota{These vertices encode respectively the boundary components of $X$, and the pieces $D, P, Y_2, Y_{111}, Y_{12}$, and $Y_3$.}
\label{vertices:fig}
\end{center}
\end{figure}

A decomposition of $X$ into pieces homeomorphic to
$$D,\ P,\ Y_2,\ Y_{111},\ Y_{12},\ Y_3,\ X_1, \ \ldots,\  X_{11}$$ 
induces a graph $G$ that has one vertex for every piece or boundary component of $X$ and one edge for every adjacency between pieces along common boundary circles, or between a piece and a boundary component. The notation chosen for all the vertices involved is illustrated in Figures \ref{vertices:fig} and \ref{Xi:fig}, where B stands for a boundary component. The notation in Figure \ref{vertices:fig} was introduced in \cite{Ma:zero} and then used also in \cite{Na} to show that every acyclic polyhedron $X$ with $c^*(X)=0$ collapses onto a disc. That in Figure \ref{Xi:fig} is new.

We now explain the various symbols present in the pictures. Let a vertex $v$ in $G$ represent some piece $W$ as in Figures \ref{vertices:fig} and \ref{Xi:fig}. The vertex $v$ has one incident edge $e$ for each boundary component $\gamma$ of $W$. Suppose that $W \neq D,P, Y_2$, so that all regions in $W$ are annuli. Let $A$ be the annular region of $W$ adjacent to $\gamma$. The edge $e$ is decorated near $v$ with two symbols:
\begin{itemize}
\item Some $k\geq 1$ dashes, where $k$ is the number of times that $A$ run on some edge of $SW$. The number $k$ is the \emph{length} of $\gamma$. When $k=1$ the dash is omitted. 
\item A red dot if $A$ is an odd region.
\end{itemize}

The notation for the M\"obius strip $Y_2$ does not follow strictly these rules since $SY_2 = \emptyset$, but is chosen to be somehow more coherent with the other pieces.

\begin{figure}
\begin{center}
\includegraphics[width = 12.5 cm]{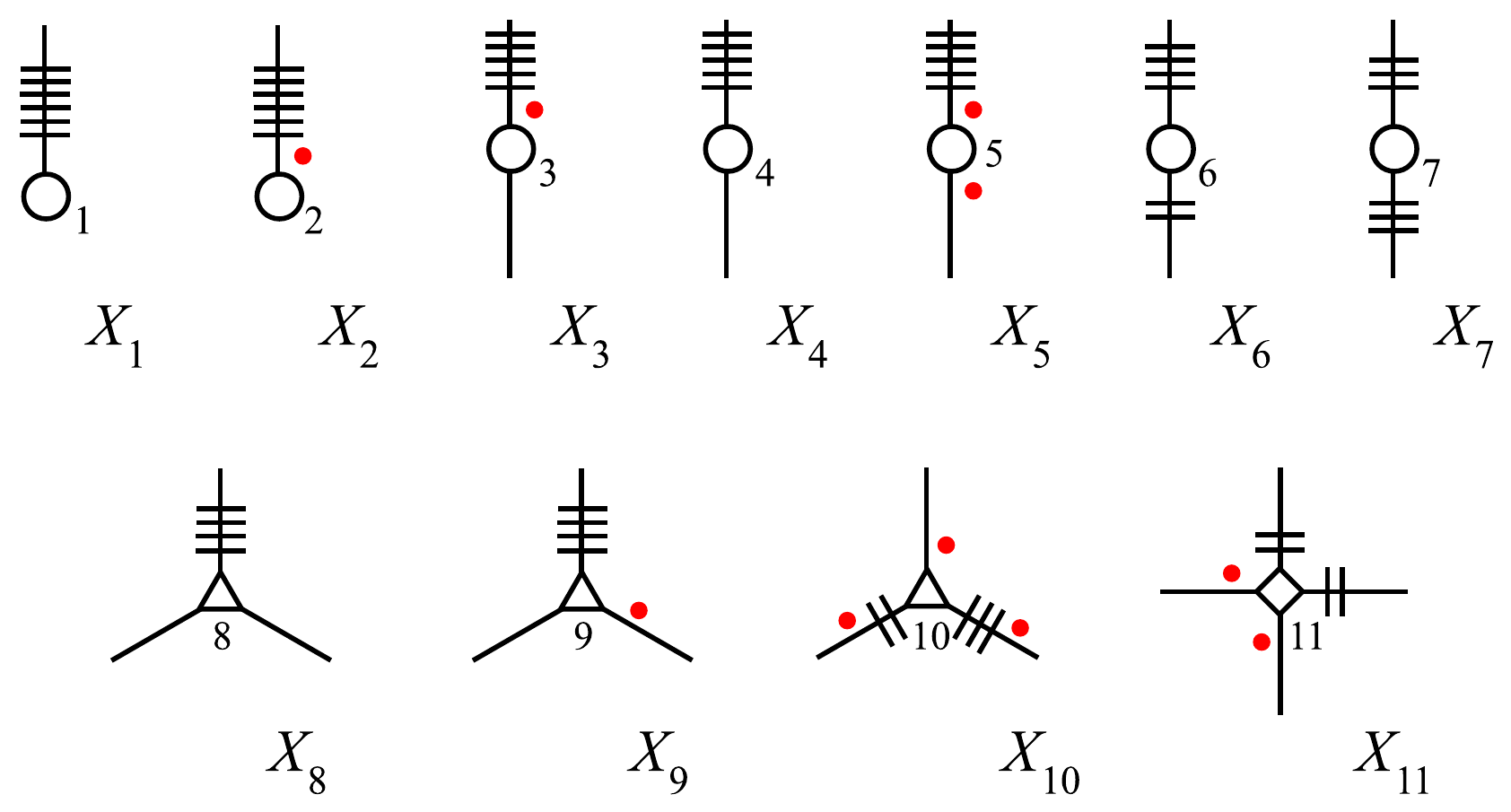}
\nota{Eleven vertices that encode the pieces $X_1, \ldots, X_{11}$.}
\label{Xi:fig}
\end{center}
\end{figure}

The decoration on the edges is useful because it carries some important information and is also enough to determine $\gamma$ unambiguously from $e$ in all cases, up to symmetries. We mean the following: two distinct edges $e,e'$ may have identical decorations only in the pieces $P, Y_{111}, X_8,$ and $X_{11}$, but in each of these cases there is a self-homeomorphism of $W$ that interchanges the two corresponding boundary components $\gamma,\gamma'$ while leaving all the other boundary components of $W$ fixed.

Does the graph $G$ so constructed determine $X$ unambiguously? Not quite, because at every edge there are two possible gluings (up to isotopy) between the adjacent pieces and we should indicate which one we use. We neglect this annoying issue because in all the cases we will be interested either there will be no ambiguity thanks to the symmetries of the pieces involved, or the choice will be clear from the context.

For simplicity, we sometimes drop the numbers $1,\ldots, 11$ and the red dots from the notation when they are not necessary (but we always keep the dashes).

\subsection{Examples} \label{examples:subsection}
Here are some examples.

\begin{ex}
The simple polyhedra $A$ and $B$ from Figure \ref{examples:fig} are homeomorphic to:
\begin{itemize}
\item[(A)] A torus with two discs added to a meridian and a longitude.
\item[(B)] A projective plane $\matRP^2$ with an annulus attached to two distinct lines $l_1,l_2 \subset \matRP^2$.
\end{itemize}
\end{ex}

\begin{figure}
\begin{center}
\includegraphics[width = 6 cm]{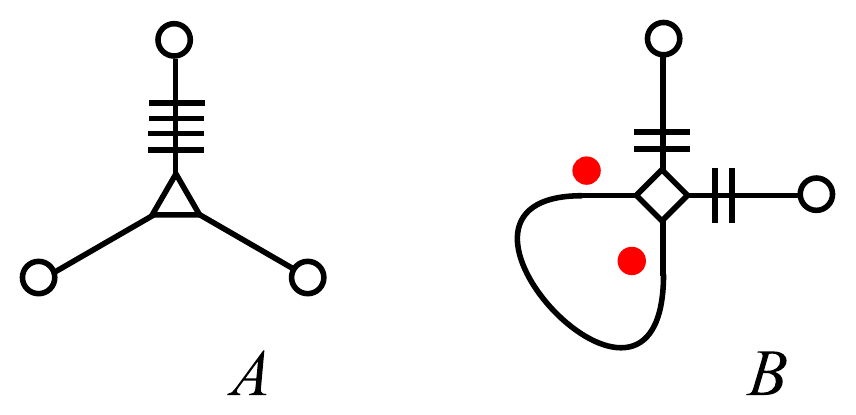}
\nota{Two simple polyhedra with complexity one.}
\label{examples:fig}
\end{center}
\end{figure}

The polyhedron (B) plays an important role in this paper. As we mentioned in Section \ref{zero:subsection}, there is a 1-1 correspondence between simple polyhedra up to 3-deformation and presentations up to Andrews-Curtis moves, see \cite{HoMe}. 

\begin{figure}
\begin{center}
\includegraphics[width = 16 cm]{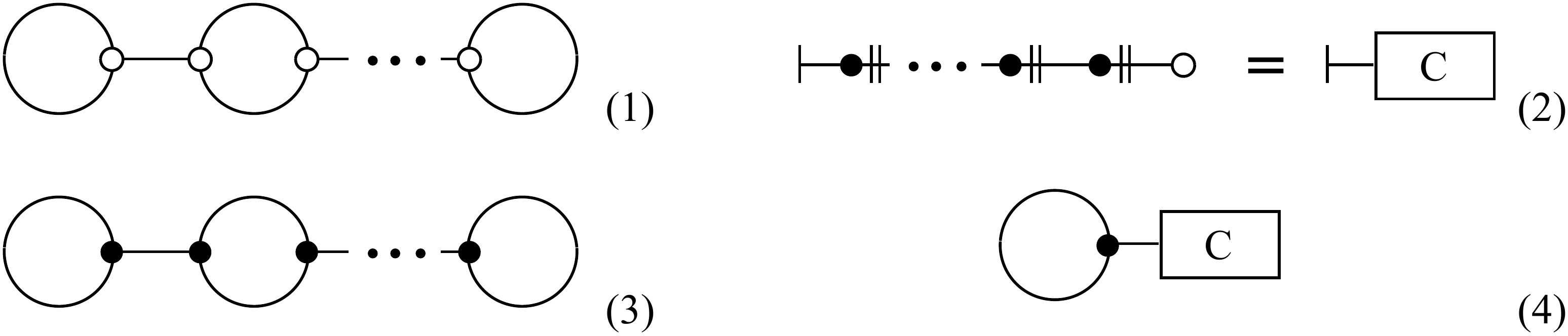}
\nota{Some simple polyhedra with complexity zero.}
\label{intro_examples0:fig}
\end{center}
\end{figure}

\begin{figure}
\begin{center}
\includegraphics[width = 16 cm]{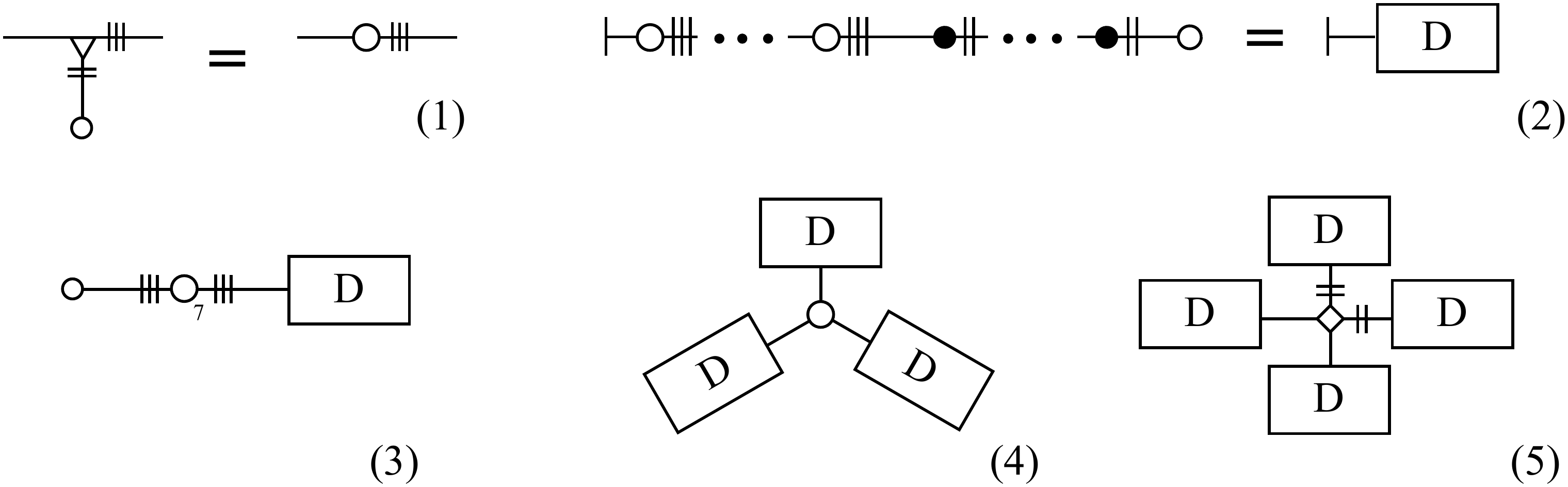}
\nota{Some simple polyhedra with complexity one.}
\label{intro_examples1:fig}
\end{center}
\end{figure}

\begin{example}
The simple polyhedra described in Figure \ref{intro_examples0:fig} have complexity 0 and determine respectively the following presentations:
$$\langle a_1,b_1, \ldots, a_g,b_g\ |\ [a_1,b_1] \cdots [a_g,b_g] \rangle,$$
$$\langle a \ |\ a^{2^n} \rangle, \quad \langle a_1,\ldots, a_n, b \ |\ [a_1,b], \ldots, [a_n,b] \rangle, \quad \langle a,b\ |\ b^{2^n}, [a,b] \rangle.$$
\end{example}

The polyhedron in Figure \ref{intro_examples0:fig}-(1) is simply a closed surface $S_g$ of some genus $g$. The presentations just listed have therefore complexity 0. By picking the spin 5-dimensional thickening $W$ of the polyhedra (1) and (3) we find the  4-manifolds $\partial W = S_g \times S^2$ and $\big(\#_n (S^2 \times S^1)\big) \times S^1$ listed in Section \ref{zero:subsection}.

In Figure \ref{intro_examples1:fig} we show some simple polyhedra that may have vertices, with connected complexity at most one. The polyhedron in Figure \ref{intro_examples1:fig}-(1) has fundamental group $\langle a\ |\ \rangle = \matZ$ and we easily deduce from Figure \ref{Xall:fig}-(10) that its two boundary components represent the elements $a$ (on the left) and $a^3$ (on the right). It is somehow similar to $Y_{12}$, that has $\pi_1(Y_{12}) = \langle a\ |\ \rangle = \matZ$ and whose two boundary components represent $a$ and $a^2$. We denote this useful polyhedron with a simpler notation as indicated in Figure \ref{intro_examples1:fig}-(1). Note that the singular part of the polyhedron D in Figure \ref{intro_examples1:fig}-(2) contains an arbitrary number of circles and 8-shaped graphs. 

\begin{ex}
The simple polyhedra described in Figure \ref{intro_examples1:fig}-(2, 3, 4, 5) determine the following presentations:
$$\calC_n, \quad \calC_{5n}, \quad D(l,m,n), \quad (l,m\ |\ n,k) $$
where $l, m, n, k$ are all of some type $2^a3^b$.
\end{ex}

We have proved that the presentations $\calP$ mentioned in Section \ref{expected:subsection} have $c^*(\calP)\leq 1$. Note that we may obtain the dihedral group as $\calD_{2n} = D(2,2,n)$.

\subsection{Encoding shadows}
Having defined a way to encode every simple polyhedron $X$ with connected complexity $\leq 1$, it is now straightforward to encode any shadow $X$ with connected complexity $\leq 1$. It suffices to add some decoration that determines the gleams of $X$.

We do this as follows. Let $G$ be any graph that describes a simple polyhedron $X$. Every edge $e$ of $G$ is \emph{even} or \emph{odd} depending on the parity of red dots colouring it from its sides (there can be 0, 1, or 2 red dots). The graph $G$ is \emph{decorated} if every edge $e$ is assigned a half-integer, with the requirement that this half-integer should be an integer if and only if $e$ is even.

The decoration on $G$ induces some gleams on $X$ in the obvious way. Every edge $e$ of $G$ determines a simple closed curve in some region $f$, and we assign the half-integer decorating $e$ to $f$. It may happen that distinct edges $e_1,\ldots, e_k$ determine curves that are contained in the same region $f$, and in this case we just add their contributions. The parity convention ensures that the resulting gleam of $f$ is an integer if and only if $f$ is an even region.

The following proposition displays some examples. The last example will be important in this paper.

\begin{figure}
\begin{center}
\includegraphics[width = 13 cm]{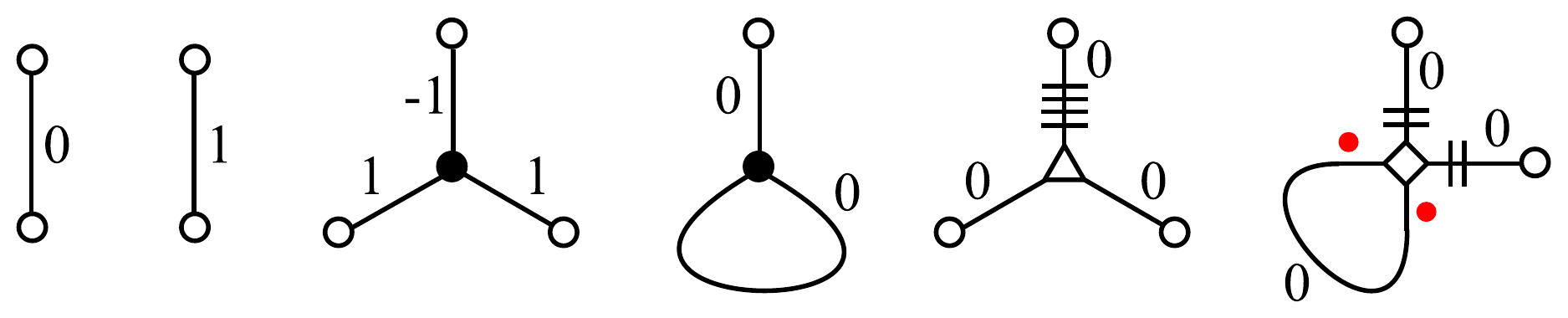}
\nota{These graphs describe some shadows of $S^4$, $\matCP^2$, $S^2\times S^2$, $S^3 \times S^1$, $S^4$, and $\matRP^3 \times S^1$ respectively.}
\label{examples_shadows:fig}
\end{center}
\end{figure}

\begin{prop}
The decorated graphs in Figure \ref{examples_shadows:fig} describe respectively some shadows of the 4-manifolds
$$S^4, \quad \matCP^2, \quad S^2\times S^2, \quad S^3 \times S^1, \quad S^4, \quad \matRP^3 \times S^1.$$
\end{prop}
\begin{proof}
The first 3 examples were already described in \cite{Ma:zero}. The fourth is a torus with a meridian attached, with gleams zero in both regions. Its 4-dimensional thickening is a punctured $D^2 \times S^1$ times an interval. By adding a 3-handle we get $D^3 \times S^1$. By adding one more 3-handle and one 4-handle we get $S^3 \times S^1$.

The fifth is a torus with a meridian and a longitude attached, everything with gleam zero. Its thickening is $S^2 \times D^2$, and by attaching a 3- and 4-handle we get $S^4$.

The last example $X$ is less obvious and is a bit similar to the fourth. It is a projective plane $\matRP^2$ with an annulus attached to two distinct lines. It has 3 regions, each with gleam zero. Its thickening $N(X)$ is diffeomorphic to the regular neighbourhood $N(X')$ of 
$$X' = \big(\matRP^2 \times \{1\}\big) \cup \big(l \times S^1\big) \subset \matRP^3 \times S^1$$
inside $\matRP^3 \times S^1$. Here $l\subset \matRP^2 \subset \matRP^3$ is any projective line. In fact $X$ is obtained from $X'$ by a small perturbation. 

Now it is easy to check that the complement of $N(X')$ in $\matRP^3 \times S^1$ is a 1-handlebody, with one 0-handle and two 1-handles. Therefore $\matRP^3 \times S^1$ is obtained from $N(X') = N(X)$ by attaching two 3-handles and one 4-handle. 
\end{proof}

\subsection{Manifolds with arbitrarily large complexity} \label{finiteness:subsection}
We end this section by proving the following general fact.

\begin{teo}
There are closed 4-manifolds $M$ with arbitrarily large $c^*(M)$.
\end{teo}
\begin{proof}
Every simple polyhedron $X$ with bounded $c^*(X)\leq n$ is constructed by attaching along their boundaries arbitrarily many simple polyhedra with boundary, that however belong only to finitely many topological types since they each have at most $n$ vertices. Since there are only finitely many topological types, we deduce easily that the \emph{width} of $\pi_1(X)$, as recently defined by Freedman \cite{Fr}, is bounded by a number that depends only by $n$. 

It is shown in \cite{Fr} that there are groups with arbitrarily large width, for instance $\matZ^n$ has width $n-1$. Therefore if we bound $c^*(X)$ we cannot get all possible finitely presented fundamental groups for $X$, hence nor for $M$.
\end{proof}

\section{The main theorem} \label{main:section}
We introduce here the main result proved in this paper, that is Theorem \ref{1:teo}-(1). We prove here its equivalence with Theorem \ref{1:teo}-(2).

\subsection{The complexity zero case} \label{zero:case:subsection}
As we stated in Section \ref{results:section}, the following theorem was proved in \cite{Ma:zero}:

\begin{teo} \label{zero:teo}
A closed oriented 4-manifold $M$ has connected complexity zero if and only if $M=M' \#_h \matCP^2$ for some integer $h$ and some graph manifold $M'$ generated by $\calS_0$.
\end{teo}

\begin{rem}
There are 4 possible ways to glue orientation-reversingly two copies of $S^2 \times S^1$. The group of orientation-preserving diffeomorphisms of $S^2\times S^1$ up to isotopy is isomorphic to $\matZ/_{2\matZ} \times \matZ/_{2\matZ}$ and is generated by the map $(x,\theta) \mapsto (-x, \bar \theta)$ and the Gluck twist $(x,\theta) \mapsto (\rot_\theta(x), \theta)$ where $\rot_\theta$ is the rotation of angle $\theta$ around the $z$ axis. See \cite{Gl}.
\end{rem}

\subsection{The complexity one case} \label{one:case:subsection}
Here is the main result of this paper:

\begin{teo} \label{main:teo}
A closed oriented 4-manifold $M$ has connected complexity $\leq 1$ if and only if 
$$M=M' \#_h \matCP^2$$
for some integer $h$ and some graph manifold $M'$ generated by $\calS_1$.
\end{teo}

We start by noting the following.

\begin{prop} \label{closed:sum:prop}
The sets of all the closed 4-manifolds generated by $\calS_0$ or $\calS_1$ are both closed under connected sum.
\end{prop}
\begin{proof}
The proof is very similar to the one that shows that Waldhausen's 3-dimensional graph manifolds are closed under connected sum. By attaching $N_3 = S^2 \times P$ and $M_1 = D^3 \times S^1$ we get the manifold $\#_2(D^3 \times S^1)$. By attaching two copies of $M_{111}$ to it we get $\#_2\big(S^2 \times S^1 \times [0,1]\big)$ that can be inserted between any gluing of two blocks to perform connected sums.
\end{proof}

In Theorem \ref{1:teo}-(2) we also stated an alternative version of the theorem that uses the thickenings of some particular polyhedra $\bar X$. We now introduce these polyhedra more formally: we will prove the equivalence of the two versions of the theorem at the end of this section.

\subsection{Simple polyhedra enriched with projective spaces}

Given a simple polyhedron $X$ with some $k \geq 0$ boundary components, we denote by $\bar X$ the polyhedron obtained by attaching a projective space $\matRP^3$ to each boundary component $\gamma\subset \partial X$ via a homeomorphism that identifies $\gamma$ with a projective line $l$ in $\matRP^3$.

We call $\bar X$ an \emph{enriched simple polyhedron}. The polyhedron $\bar X$ has dimension 2 if $k=0$ and 3 if $k>0$. If we assign some gleams to $X$, we get a 4-dimensional thickening $N(\bar X)$ of $\bar X$, where every $\matRP^3$ thickens to a $\matRP^3 \times [-1,1]$ and $X$ thickens as prescribed by the gleams. (To be precise, to interpret the gleams on the regions incident to $\partial X$ we need to fix a line bundle above every boundary component $\gamma=l$ of $X$, and we choose the one induced by any projective plane $\matRP^2\subset \matRP^3$ containing $l$.)  The 4-manifold $N(\bar X)$ is oriented and with boundary. The boundary $\partial N(\bar X)$ has $k+1$ connected components, $k$ of which are copies of $\matRP^3$.

We also admit the degenerate case $X=S^1$ and $\bar X = \matRP^3$. In this case we get $N(\bar X) = \matRP^3 \times [-1,1]$. As another example, if $X$ is an annulus with gleam zero, then $\bar X$ consists of two copies of $\matRP^3$ connected by an annulus, and it thickens to a 4-manifold $N(\bar X)$ with boundary consisting of two $\matRP^3$ and one $S^2\times S^1$. One may verify quite easily that $N(\bar X)$ is diffeomorphic to $\matRP^3 \times [-1,1]$ with one line $l \times \{0\}$ drilled.

If $k=0$ of course we get $\bar X = X$.

\subsection{Alternative version with doubles}
Here is an alternative version of Theorem \ref{main:teo}, already stated as Theorem \ref{1:teo}-(2). Given an orientable manifold with boundary $W$ we denote by $DW$ its double, equipped with any orientation (a double is always mirrorable).

\begin{teo} \label{main2:teo}
A closed oriented 4-manifold $M$ has connected complexity $\leq 1$ if and only if 
$$M=M'\#_h \matCP^2$$ 
for some $h\in \matZ$ and with $M'=D(N(\bar X))$ for some simple polyhedron $X$ with $c^*(X) \leq 1$.
\end{teo}

The polyhedron $X$ has some $k \geq 0$ boundary components; the polyhedron $\bar X$ has dimension 2 if $k=0$ and dimension 3 otherwise. This alternative version of our main theorem furnishes immediately a relevant information: the manifold $M'$ is the double of some manifold with boundary. This implies immediately that its signature vanishes, that is $\sigma (M')=0$ and hence $\sigma(M)=h$. 

\begin{example}
In the degenerate case $X = S^1$ we get $\bar X = \matRP^3$ and $D(N(\bar X)) = \matRP^3 \times S^1$.
\end{example}

Note that $\matRP^3 \times S^1$ can also be obtained by glueing the blocks $M_{12}$ and $N_1$.

\subsection{Doubles of 2-handlebodies}
Recall that a 2-handlebody is any $4$-manifold $W$ that decomposes with 0-, 1- and 2-handles only. Being a 2-handlebody is a quite restrictive condition: for instance, the map $\pi_1(\partial W) \to \pi_1(W)$ must be surjective. We now make an important observation.

\begin{prop}
If $X$ is a simple polyhedron without boundary, then $D(N(X))$ is the double of a 2-handlebody. 
\end{prop}
\begin{proof}
The thickening $N(X)$ is a 2-handlebody.
\end{proof}

It is clear that many closed 4-manifolds with vanishing signature are not doubles of 2-handlebodies. The following is a relevant example for us.

\begin{prop} \label{not:2:prop}
The manifold $\matRP^3 \times S^1$ is not the double of a 2-handlebody.
\end{prop}
\begin{proof}
Suppose that $\matRP^3 \times S^1 = DW$ for some 2-handlebody $W$. We get $\chi(W)=0$, so $W$ has a handle decomposition with $h+1$ one-handles and $h$ two-handles. This leads to a contradiction because the group $\pi_1(DW) = \pi_1(W) = \matZ \times \matZ/_{2\matZ}$ has deficiency zero, see \cite[Chapter 5]{CZ}.
\end{proof}

The manifold $\matRP^3 \times S^1$ has connected complexity one, it is the double of $\matRP^3 \times [-1,1]$, but it is not the double of a 2-handlebody. 

\subsection{Asphericity}
Here is another important topological information derived from Theorem \ref{main2:teo}.

\begin{teo} \label{aspherical:teo}
No closed oriented 4-manifold $M$ with $c^*(M) \leq 1$ is aspherical.
\end{teo}
\begin{proof}
Every such manifold is diffeomorphic to $M=M' \#_h \matCP^2$ with $M' = D(N(\bar X))$ for some enriched simple polyhedron $\bar X$. Suppose that $M$ is aspherical. Since $\pi_2(M)$ vanishes, we get $h=0$ and $M=M'$. If $\bar X$ is 2-dimensional, the retraction $D(N(\bar X)) \to N(\bar X) \subset D(N(\bar X))$ induces an isomorphism on fundamental groups: since $M'$ is aspherical, the retraction is homotopic to the identity, a contradiction since it has degree zero.

If $\bar X$ contains $k>0$ projective spaces, the map $\pi_3(\bar X) \to H_3(\bar X, \matZ) = \matZ^k$ has non-trivial image, so in particular $\pi_3(\bar X) \neq \{e\}$. The retractions $D( N(\bar X) ) \to N(\bar X) \to \bar X$ imply that $\pi_3(D(N(\bar X) )) \neq \{e\}$.
\end{proof}

\subsection{How we can encode doubles}
As noted in \cite{Ma:zero}, the doubles of thickenings of simple 2-dimensional polyhedra are easily encoded by homological data. We extend this observation to enriched simple polyhedra.

Let $\bar X$ be an enriched simple polyhedron. By varying the gleams on $X$ we get many different thickenings $N(\bar X)$. However, the following proposition shows that we get only finitely many doubles $D(N(\bar X))$ up to diffeomorphisms, and these are easily classified by the elements in $H^2\big(X, \matZ/_{2\matZ}\big)$. For every thickening $N(\bar X)$, the natural inclusion $i\colon X \hookrightarrow D(N(\bar X)) = M'$ induces a map $i^*\colon H^2\big(M',\matZ/_{2 \matZ}\big) \to H^2\big(X, \matZ/_{2 \matZ}\big)$.

\begin{prop} \label{alpha:prop}
For every $\alpha \in H^2\big(X, \matZ/_{2 \matZ}\big)$ there is (up to diffeomorphism) precisely one double $M' = D(N(\bar X))$ whose second Stiefel-Whitney class $w_2$ satisfies $i^*(w_2) = \alpha$. The double is spin if and only if $\alpha=0$.
\end{prop}

We defer the proof of the proposition to Section \ref{alpha:proof:subsection}.
For the moment we content ourselves with the following simple examples:
\begin{itemize}
\item If $X=S^2$, then $M'$ equals $S^2 \times S^2$ or $S^2 \timtil S^2 = \matCP^2 \# \overline{\matCP}^2$ depending on the parity of the gleam on $X$.
\item If $X= \emptyset$, then $\bar X = \matRP^3$ and $M' = \matRP^3 \times S^1$.
\item If $X$ is an orientable surface with $k>0$ boundary components, then $M'$ is the unique oriented manifold obtained by gluing $X \times S^2$ to $k$ copies of $M_{12}^1$.
\item If $X = Y_{111}$, then $M'$ is the unique oriented manifold obtained by gluing $P^3 \times S^1$ to 3 copies of $M_{12}$. Here $P^3$ is $S^3$ minus 3 open balls.
\end{itemize}
All the manifolds listed are spin except $S^2 \timtil S^2$.

\section{The constructive part} \label{constructive:section}
We prove here the constructive part of Theorem \ref{main:teo}, namely that every manifold $M=M' \#_h\matCP^2$ as stated there has connected complexity $\leq 1$. The other half of the theorem, which says that all the manifolds with connected complexity $\leq 1$ are of this kind, is harder and will be proved in the next sections.

We also show the equivalence between Theorems \ref{main:teo} and \ref{main2:teo}, that is between Theorem \ref{1:teo}-(1) and (2).

\subsection{Shadows with boundary}
In the definition that we gave in Section \ref{shadows:subsection} a shadow is a simple polyhedron without boundary decorated with gleams. We now relax this definition by allowing the presence of some boundary component. We follow \cite{CoThu}.

From this point on, we let a \emph{shadow} be a simple polyhedron $X$, possibly with boundary, decorated with gleams; as usual, these are half-integers attached to regions, that are integers precisely on the even regions. 

A shadow $X$ thickens to a compact oriented 4-manifold $N(X)$ that fibres over $X$ via a map $\pi\colon N(X) \to X$. The fibre over a point in $\partial X$ or in some region of $X$ is a disc. 

The boundary $\partial N(X)$ decomposes into a \emph{vertical part} $\partial_vN(X) = \pi^{-1}(\partial X)$ and a \emph{horizontal part} $\partial_hN(X)$ that is the closure of $(\pi|_{\partial N(X)})^{-1}(X \setminus \partial X)$. The vertical part consists of solid tori $V_1, \ldots, V_h$ above the components $\gamma_1,\ldots, \gamma_h$ of $\partial X$. The core $\gamma_i$ of $V_i$ is equipped with a framing, induced by the gleam of the adjacent region. See \cite[Section 3]{CoThu} for more details.

For instance, a surface with boundary $X$ thickens to a disc bundle over $X$, with its obvious vertical and horizontal boundary.

\subsection{Blocks}
Recall that a \emph{block} is a compact oriented 4-manifold $M$ with (possibly empty) boundary made of some copies of $S^2\times S^1$. For instance, all the manifolds in $\calS_1$ are blocks. 

A \emph{framed block} is a pair $(M,L)$ where $M$ is a block and $L\subset \partial M$ is a framed link that consists of one framed fiber $\{pt\}\times S^1$ on each boundary component. The framing has only an auxiliary role, so we usually drop $L$ from the notation.

\subsection{Shadow of a block}
Let $X$ be a shadow with boundary and $N(X)$ its thickening. Suppose that $\partial N(X) \isom \#_h(S^2 \times S^1)$ for some $h\geq 0$. In this case we can perform the following construction, first defined in \cite{Ma:zero}, that produces a framed block $M$ from $X$.

\begin{figure}
\begin{center}
\includegraphics[width = 13 cm]{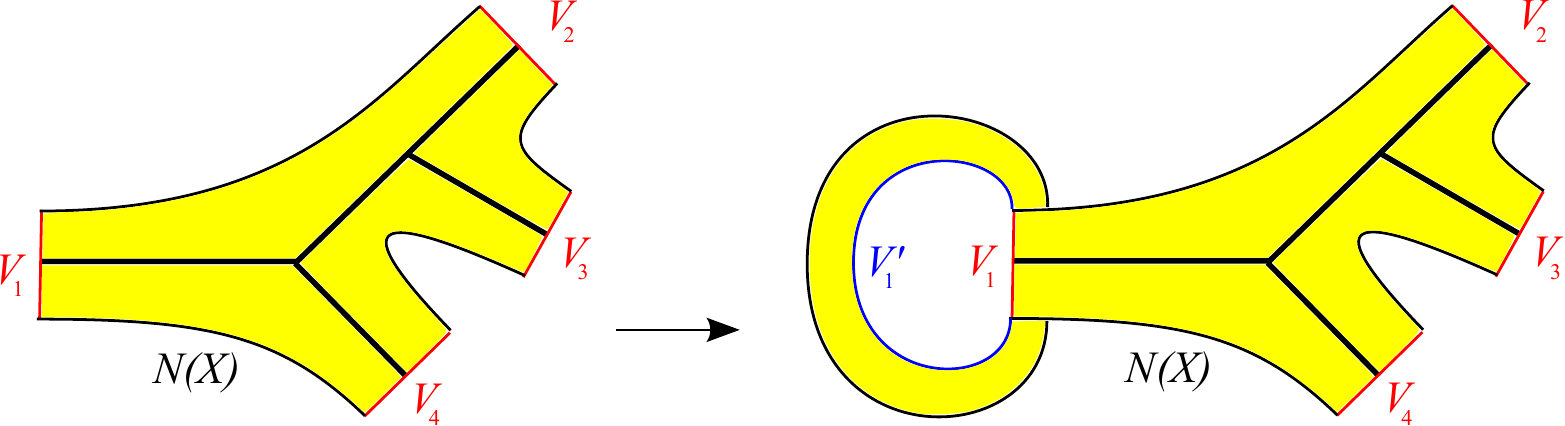}
\nota{How to construct a framed block $M$ from $X$. Here $X$ is represented in black (as a graph) and $N(X)$ is in yellow. Each vertical solid torus $V_i\subset\partial N(X)$ is doubled, so that $V_i\cup V_i'\isom S^2\times S^1$. (Here, this is shown for $i=1$ only.) 
}
\label{determina:fig}
\end{center}
\end{figure}

Let $X$ have $k$ boundary components $\gamma_1,\ldots, \gamma_k$, that are framed cores of the vertical solid tori $V_1,\ldots, V_k$ in $\partial_v N(X)$. As suggested by Figure \ref{determina:fig}, we pick $N(X)$ and we double each vertical solid torus $V_i$ along its boundary, thus adding another solid torus $V_i'$. Now $V_i \cup V_i' \isom S^2 \times S^1$. Moreover we thicken $V_i'$ as in the figure. 

We have thus enlarged $N(X)$ to a bigger compact 4-manifold $W$, that has $k+1$ boundary components. Of these, we have that $k$ are $V_i\cup V_i' \isom S^2 \times S^1$, and the last one is still diffeomorphic to $\partial N(X) \isom \#_h(S^2 \times S^1)$. We cap off the last boundary component by attaching $h$ 3-handles and one 4-handle, and call $M$ the resulting manifold.

We have constructed a block $M$ with $k$ boundary components. The block is framed as $(M,L)$ with $L= \gamma_1 \sqcup \cdots \sqcup \gamma_k$. Recall that each $\gamma_i$ has a framing induced by the gleam of the adjacent region in $X$. 

We say that $X$ is a shadow of the block $M$. When $\partial X = \emptyset$ then $\partial M = \emptyset$ and we recover here the original definition of shadow of a closed 4-manifold.

\begin{example}
Let $X$ be a surface with non-empty boundary equipped with some gleam. The thickening $N(X)$ is a disc bundle over $X$ and is also a 1-handlebody. Therefore $\partial N(X) \isom \#_h(S^2 \times S^1)$. We deduce that $X$ is a shadow of some framed block $M$, uniquely determined by $X$. We can see easily that $M$ is the unique oriented $S^2$-bundle over $X$.
\end{example}

\begin{rem}
Alternatively, we can say that a shadow for a block $(M, L)$ is a locally flat simple polyhedron $X \subset M$ such that $\partial X = L = X \cap \partial M$ and $M\setminus \interior{N(X\cup \partial M)}$ is a 1-handlebody. The embedding $X\subset M$ induces the appropriate gleams on $X$. See \cite{Ma:zero}. 
\end{rem}

\subsection{Important examples}
Here are some examples that are fundamental for us.

\begin{prop} \label{equipped:prop}
The simple polyhedra
$$D,\ A,\ P,\ Y_2,\ Y_{111},\ Y_{12},\ Y_3,\ X_1,\ \ldots,\ X_{11}$$
equipped with arbitrary gleams are shadows of the blocks
$$N_1,\ N_2,\ N_3,\ M_2,\ M_{111},\ M_{12},\ M_3,\ M_1^1,\ \ldots,\ M_{11}^1.$$
\end{prop}
\begin{proof}
Same proof as in \cite[Proposition 3.16]{Ma:zero}. Each polyhedron $X$ in the list, equipped with arbitrary gleams, thickens to a 4-manifold $N(X)$ which is in fact a 1-handlebody. Therefore $X$ is a shadow of some block $M$. To see that $M$ is as stated, note that the candidate $M$ is obtained by mirroring $N(X)$ along its horizontal boundary $\partial_hN(X)$, so $M \setminus \interior {N(X\cup \partial M)} \isom \interior{N(X)}$ is also an open 1-handlebody, that is it is made of 3- and 4-handles.
\end{proof}

The regions of all the simple polyhedra involved in the previous proposition are incident to the boundary, so by varying their gleams we only change the framing of the respective block.

\subsection{Complexity of blocks}
The \emph{complexity} $c(M)$ of a block $M$ is the minimum complexity of a shadow $X$ for $M$. The \emph{connected complexity} $c^*(M)$ is the minimum connected complexity of a shadow $X$ for $M$.

For instance, all the blocks listed in Proposition \ref{equipped:prop} have complexity zero or one.

\begin{rem}
The block $M_1= D^3 \times S^1$ is a bit peculiar. A natural shadow for it should be a 1-dimensional circle, since $D^3 \times S^1$ is obtained by adding a 3- and a 4-handle to its 4-dimensional thickening. We set $c(M_1) = c^*(M_1) = 0$ by convention.
\end{rem}

\subsection{Connected sum and assembling}
We now introduce some important manipulations on blocks and show how these can be easily translated into manipulations of shadows and decorated graphs.

We recall from \cite[Sections 4.1 and 4.3]{Ma:zero} the crucial operations of connected sum and assembling. Let $M$ be a (possibly disconnected) framed block. A connected sum consists as usual as the removal of the interior of two 4-discs from the interior of $M$ and the gluing of the two resulting boundary 3-spheres via an orientation-reversing diffeomorphism.
An assembling consists of gluing altogether two boundary components of $M$ via a framing-preserving orientation-reversing diffeomorphism. 

\begin{figure}
\begin{center}
\includegraphics[width = 12 cm]{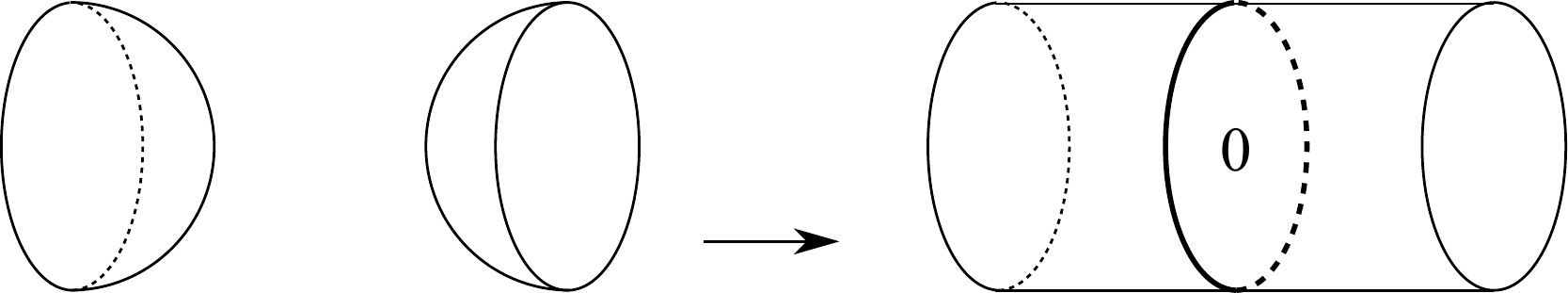}
\nota{This move on shadows corresponds to a connected sum of manifolds.}
\label{sum:fig}
\end{center}
\end{figure}

\begin{figure}
\begin{center}
\includegraphics[width = 12 cm]{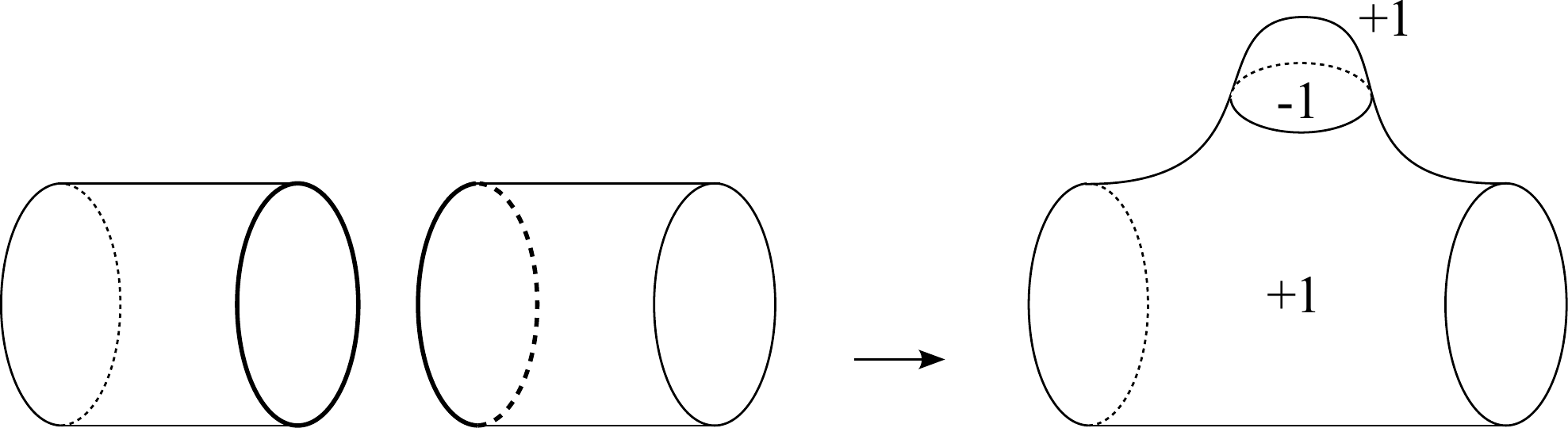}
\nota{This move on shadows represents an assembling of blocks. Two boundary components are glued, and a bubble is added.}
\label{assembling:fig}
\end{center}
\end{figure}

On shadows, connected sums and assemblings may be realized as in Figures \ref{sum:fig} and \ref{assembling:fig}, as proved in \cite[Sections 4.1 and 4.3]{Ma:zero}. We may encode these moves at the level of decorated graphs as in Figure \ref{sum_ass_new:fig}.

\begin{figure}
\begin{center}
\includegraphics[width = 14 cm]{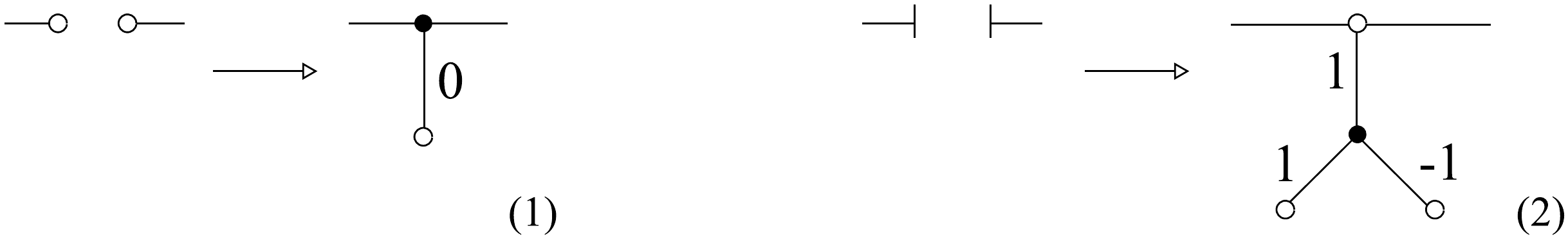}
\nota{Connected sum and assemblings of blocks via decorated graphs.}
\label{sum_ass_new:fig}
\end{center}
\end{figure}

A crucial observation is that both operations do not produce new vertices and hence do not increase the (connected) complexity of the shadows. Therefore the following holds.

\begin{prop} \label{leq:prop}
If a block $M'$ is obtained from $M$ by assembling or connected sum, then
$$c(M') \leq c(M), \qquad c^*(M') \leq c^*(M).$$
\end{prop}

Strictly speaking, the previous discussion does not apply if we assemble a block with $D^3 \times S^1$ since the complexity of $D^3 \times S^1$ has been set zero by assumption. We now consider this peculiar assembling separately, and show that Proposition \ref{leq:prop} holds also in this case.

\subsection{Filling a block}
Let $M$ be a block. The \emph{filling} of a boundary component of $M$ is the assembling of $M$ with $D^3 \times S^1$ along that component. Equivalently, this operation consists of adding a 3- and a 4-handle. The result is a new block $M'$ with one boundary component less than $M$.

\begin{prop}  \label{filling:prop}
If a block $M'$ is obtained by filling a boundary component of $M$, then
$$c(M') \leq c(M), \qquad c^*(M') \leq c^*(M).$$
\end{prop}
\begin{proof}
If $X$ is a shadow of $M$, a shadow $X'$ for $M'$ is constructed simply by collapsing the region of $X$ incident to that boundary component. Since by collapsing we do not create any new vertex, we get the inequality.

More precisely, we collapse $X$ starting from the region adjacent to the filled boundary as much as possible. We end up with a shadow $X''$ plus possibly a 1-dimensional part. The shadow $X''$ determines a block $M''$ and $M'$ is obtained from $M''$ via connected sums, possibly also with additional copies of $S^3\times S^1$. Since $c(S^3 \times S^1)=0$ we get $c(M') \leq c(M'') \leq c(M)$
and $c^*(M') \leq c^*(M'') \leq c^*(M)$.
\end{proof}

The proof of the proposition also shows how to construct a shadow $X'$ for $M'$ from one $X$ of $M$: we only have to collapse $X$ starting from that component of $\partial X$ contained in the boundary that we want to fill. We will use this move quite often.

\subsection{Drilling along a curve}
The inverse operation of filling is of course \emph{drilling} a block $M$ along a simple closed curve $\gamma\subset \interior M$. The result is a new block $M'$ with one additional boundary component.

If $X$ is a shadow of $M$, a shadow $X'$ for $M'$ is constructed by isotoping $\gamma$ to an immersed generic curve in $X$ and then attaching an annulus to $X$ along $\gamma$, and modifying the gleams of the regions near $\gamma$ in any way, provided that whenever a region $f$ of $X$ is subdivided into some regions $f_1', \ldots, f_k'$ of $X'$ the gleams of $f_1', \ldots, f_k'$ sum to the original gleam of $f$. There is of course some freedom here, but it has the only effect of modifying the framing of the new boundary component of $M'$.

This operation produces a shadow $X'$ for $M'$, that has more vertices than $X$ if $\gamma$ has self-intersections or intersects $SX$. So drilling along a curve may in principle increase arbitrarily the complexity of a manifold.

If $\gamma$ is embedded and does not intersect $SX$, then $X'$ has the same vertices as $X$. For instance if $X$ is described by a decorated graph $G$ and $\gamma$ corresponds to an edge of the graph, this operation is easily encoded as in Figure \ref{drill:fig}. The 3 edges of the new portion in Figure \ref{drill:fig}-(right) can be decorated with any half-integers, as long as the sum of the numbers on the two horizontal edges equal the number decorating the original one in Figure \ref{drill:fig}-(left). 

\begin{figure}
\begin{center}
\includegraphics[width = 5 cm]{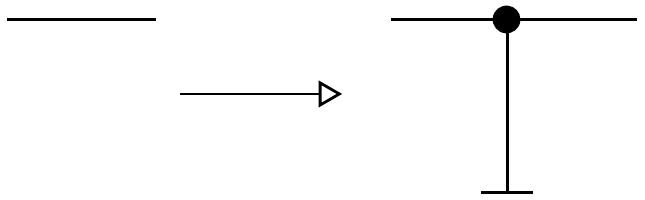}
\nota{Drilling along a curve determined by an edge of the decorated graph.}
\label{drill:fig}
\end{center}
\end{figure}

\subsection{The additional block}
We can finally exhibit a shadow for the additional block $M_{12}^1$. In Figure \ref{additional:fig} and in the rest of the paper we will use the following convention: 

\begin{center}
\emph{We indicate the fraction $\pm \frac 12$ simply via the sign $\pm $.}
\end{center}
So in Figure \ref{additional:fig} the edges decorated with $+$ and $-$ are actually decorated with $\frac 12$ and $-\frac 12$ respectively.

\begin{figure}
\begin{center}
\includegraphics[width = 3.5 cm]{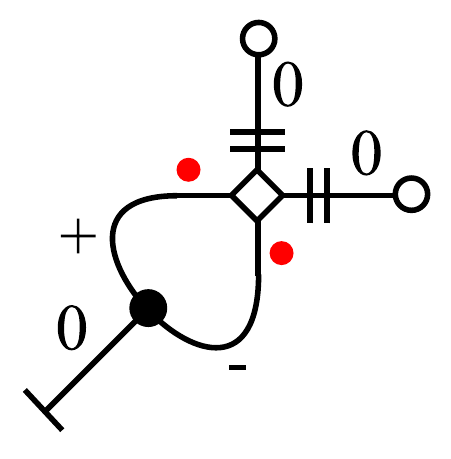}
\nota{A shadow $X_{12}$ for the block $M_{12}^1$. By our convention, the signs $+$ and $-$ represent the gleams $\frac 12$ and $-\frac 12$.}
\label{additional:fig}
\end{center}
\end{figure}

\begin{prop}
The shadow $X_{12}$ encoded in Figure \ref{additional:fig} is a shadow of $M_{12}^1$.
\end{prop}
\begin{proof}
We apply Figure \ref{drill:fig} to the shadow of $\matRP^3 \times S^1$ shown in
Figure \ref{examples_shadows:fig}. This amounts to drilling $\matRP^3 \times S^1$ along a curve that is isotopic to a line $l \times \{{\rm pt}\} \subset \matRP^3 \times S^1$.
\end{proof}

\subsection{Properties of the connected complexity}
We are now close to proving the constructive part of Theorem \ref{main:teo}.
We make an important observation: one reason for preferring the connected complexity in our investigation is 
that if we assemble an arbitrary number of blocks in $\calS_1$ we get as a result a new block $M$ with $c^*(M)\leq 1$, while $c(M)$ could be very large and hard to control.
More generally, the following holds.

\begin{prop} \label{closed:prop}
The set $\calB_n$ of all blocks having connected complexity $\leq n$ is closed under disjoint union, connected sum, assembling, and filling.
\end{prop}

As we mentioned, the set $\calB_n$ is not closed under drilling. This is quite reasonable: if we drill along a complicated curve, we get a more complicated manifold. 
Of course the closed oriented 4-manifolds with connected complexity $\leq n$ are precisely the blocks in $\calB_n$ with empty boundary. Our aim here is to understand $\calB_1$. 

\subsection{The constructive part}
We now prove the constructive part of Theorem \ref{main:teo}. 

\begin{teo} \label{easy:teo}
Let $M'$ be a graph manifold generated by $\calS_1$ and $h$ an integer. We have
$$c^*\big(M'\#_h \matCP^2\big) \leq 1.$$
\end{teo}
\begin{proof}
Connected sums and assemblings do not increase the connected complexity. All the blocks in $\calS_1$ and $\matCP^2$ have connected complexity $\leq 1$, so we are done.
\end{proof}

\subsection{Proof of the equivalence}
We can finally show that the Theorems \ref{main:teo} and \ref{main2:teo} are in fact equivalent.

\begin{prop}
A closed 4-manifold $M'$ is generated by $\calS_1$ if and only if $M' = D\big(N(\bar X)\big)$ for some simple polyhedron $X$ with $c^*(X) \leq 1$.
\end{prop}
\begin{proof}
Let $X$ be a simple polyhedron with $c^*(X)\leq 1$. It decomposes into pieces homeomorphic to 
$$D,\ A,\ P,\ Y_2,\ Y_{111},\ Y_{12},\ Y_3,\ X_1,\ \ldots,\ X_{11}.$$
The enriched polyhedron $\bar X$ decomposes into pieces of this kind, plus pieces homeomorphic to $\matRP^3$ with an annulus attached to a projective line.
The double $D\big(N(\bar X)\big)$ decomposes accordingly into blocks diffeomorphic to
$$N_1,\ N_2,\ N_3,\ M_2,\ M_{111},\ M_{12},\ M_3,\ M_1^1,\ \ldots,\ M_{11}^1$$
plus some blocks obtained by doubling $\matRP^3 \times [-1,1]$ and then drilling a line. These latter blocks are just copies of $M_{12}^1$. Therefore $D\big(N(\bar X)\big)$ is generated by $\calS_1$.

The converse is proved using the same argument in the opposite direction. The block $D^3 \times S^1$ is treated separately as in the proof of Proposition \ref{filling:prop}, noting that $\#_h(S^3 \times S^1)$ is the double of a 1-handlebody, and every 1-handlebody has a shadow $X$ with $c^*(X)\leq 1$.
\end{proof}

\begin{rem}
With similar techniques we can easily see that if we allow to enrich a simple polyhedron $X$ via copies of $S^3$ or $S^2 \times S^1$ instead of $\matRP^3$, attached to $\partial X$ along their Heegaard cores, we do not get any new manifold. 
In some sense $\matRP^3$ is the simplest 3-dimensional stratum that, when attached to a 2-dimensional simple polyhedron, may contribute in creating new manifolds like $\matRP^3 \times S^1$ (in fact, note that we already met $S^3 \times S^1$ and $S^2 \times S^1 \times S^1$ in complexity zero).
\end{rem}

\subsection{Explicit shadows}
It is now worth exhibiting an explicit shadow for any closed manifold $M$ with $c^*(M)\leq 1$. Let $M = D(N(\bar X)) \#_h \matCP^2$ be a manifold with connected complexity $\leq 1$, as described in Theorem \ref{main2:teo}. Here $X$ is any shadow (possibly with boundary) with connected complexity $\leq 1$ and $h\in \matZ$.

\begin{figure}
\begin{center}
\includegraphics[width = 12 cm]{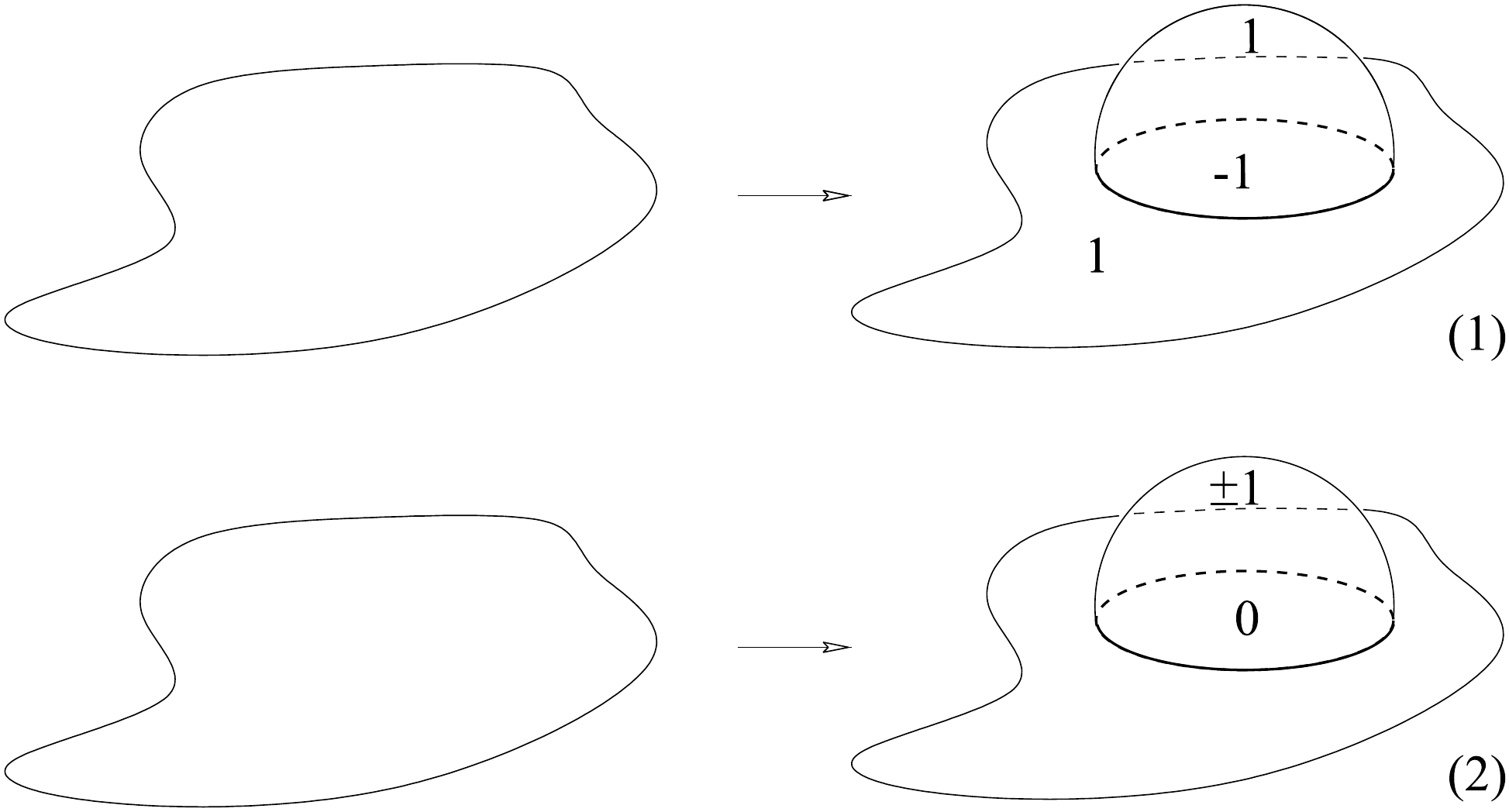}
\nota{Two bubble moves with different effects. The first one is useful to construct doubles, the second represents a connected sum with $\pm \matCP^2$.}
\label{bubble:fig}
\end{center}
\end{figure}

\begin{prop} \label{Xstar:prop}
A shadow $X_*$ for $M$ can be constructed from $X$ as follows:
\begin{enumerate}
\item Add one bubble as in Figure \ref{bubble:fig}-(1) to each region of $X$.
\item Add $|h|$ bubbles as in Figure \ref{bubble:fig}-(2) with signs coherent with $h$ to any region of $X$.
\item Attach one portion $X_{12}$ as in Figure \ref{additional:fig} at every boundary component of $X$.
\end{enumerate}
\end{prop}
\begin{proof}
The shadow $X_*$ arises when we assemble and connect-sum the shadows of the different blocks and $\matCP^2$, as prescribed by Figures \ref{assembling:fig} and \ref{sum_ass_new:fig}. If this construction produces more than one bubble as in Figure \ref{bubble:fig}-(1) on a single region, actually one bubble suffices (there are canceling pairs of 2- and 3-handles otherwise).

In other words, with (1) we construct a shadow for the double of $N(X)$ along its horizontal boundary, with (2) we add $\#_h(\matCP^2)$, and with (3) we attach the blocks $M_{12}^1$ to the doubled vertical boundary.
\end{proof}

\subsection{Proof of Proposition \ref{alpha:prop}} \label{alpha:proof:subsection}
Let $X$ be a shadow with $c^*(X)\leq 1$. Proposition \ref{Xstar:prop} shows that a shadow $X_*$ for $M'=D(N(\bar X))$ is constructed from $X$ by adding bubbles to regions and copies of $X_{12}$ to $\partial X$. (We mean here the bubble in Figure \ref{bubble:fig}-(1).)

We will next prove that the move shown in Figure \ref{moves4:fig} modifies $X_*$ into another shadow of the same manifold $M'$. Since a bubble is attached to each region of $X$, this easily implies that the following moves modify the gleams of $X$ without affecting the double $M'$:

\begin{itemize}
\item Change the gleam at some region of $X$ by adding $\pm 2$.
\item At some edge of $X$, modify the gleams of the 3 adjacent regions by adding $+1$ on each.
\end{itemize}

Moreover, we will also prove the move in Figure \ref{X11moves:fig}-(4), and when applied to $X_{12}$ it shows that we can also do the following without affecting the double $M'$:

\begin{itemize}
\item At every region adjacent to $\partial X$, modify the gleam by adding $\pm 1$.
\end{itemize}

In other words, the manifold $M'$ only depends on the cocycle $\alpha$ in $\matZ/_{2\matZ}$ induced by the gleams, considered up to coboundaries: this is precisely the canonical class $\alpha\in H^2\big(X,\matZ/_{2\matZ}\big)$ of $X$, which is the pull-back of the Stiefel-Whitney class $w_2(M')$ along the inclusions $X \hookrightarrow X_* \hookrightarrow N(X_*) \hookrightarrow M'$.

If $\alpha=0$ then $w_2(N(X_*))=0$ because the canonical class of $X_{12}$ also vanish. Then $w_2(M')=0$ since $M'$ is obtained from $N(X_*)$ by adding 3- and 4-handles: so $M'$ is spin.

\section{Decompositions of $\#_h(S^2 \times S^1)$ along tori} \label{tori:section}
Having proved the constructive part of Theorem \ref{main:teo}, we are left to complete the harder task: proving that every closed 4-manifold $M$ with $c^*(M)\leq 1$ is of the type described by the theorem. This will occupy the rest of the paper.

We show here that this problem leads us naturally to study some decompositions of $ \#_h (S^2 \times S^1)$ along tori.

\subsection{Decomposition along tori}
Let $X$ be a shadow of a block $M$. By hypothesis we have $\partial N(X) \isom \#_h(S^2 \times S^1)$ for some $h\geq 0$.

We know that $X$ decomposes into pieces homeomorphic to 
$$D,\ A,\ P,\ Y_2,\ Y_{111},\ Y_{12},\ Y_3,\ X_1,\ \ldots,\ X_{11}.$$
The horizontal boundary $\partial_hN(X)$ fibers over $X$ and decomposes accordingly into 3-manifolds bounded by tori, one 3-manifold lying above each piece.

The 3-manifolds fibering above the first seven pieces $D, A, P, Y_2, Y_{111}, Y_{12}, Y_3$ are all Seifert manifolds. As shown in \cite[Table 1]{Ma:zero} these seven manifolds are respectively:
$$D \times S^1,\ A \times S^1, \ P \times S^1,\ (D,2,2), \ P\times S^1,\ (A,2), \ (D,3,3).$$
Here $(D,n,n)$ is the Seifert manifold with parameters $\big(D, (n,1), (n,-1)\big)$,  and $(A,2)$ is the Seifert manifold with parameters $\big(A,(2,1)\big)$.
The last four Seifert manifolds are the complements in $S^2 \times S^1$ of the corresponding links in Figure \ref{blocchi:fig}. 

On the other hand, the 3-manifolds fibering above $X_1, \ldots, X_{11}$ are 11 cusped hyperbolic manifolds $W_1,\ldots, W_{11}$. This fact was originally proved by Costantino and Thurston in a more general setting \cite{CoThu}. Hyperbolic manifolds of this kind were also studied in \cite{CoFriMaPe}. Each $W_i$ decomposes into two regular ideal octahedra.

If $G$ is a decorated graph that describes $X$, the same graph also describes a decomposition of $\partial N(X) \isom \#_h (S^2 \times S^1)$ into 3-manifolds along tori. Every boundary vertex (B) contributes with a vertical solid torus in $\partial_hN(X)$.

\begin{example}
The six graphs in Figure \ref{examples_shadows:fig} describe decompositions of $S^2 \times S^1$, $S^3$, $S^3$, $\#_2(S^2 \times S^1)$, $S^2 \times S^1$, and $\#_2(S^2 \times S^1)$. In the last two examples the central vertex represents a cusped hyperbolic manifold.
\end{example}

\subsection{Compressing discs} \label{discs:subsection}
We make the following simple but crucial observation.

\begin{prop}
Every torus $T\subset \#_h(S^2 \times S^1)$ has a compressing disc.
\end{prop}
\begin{proof}
The induced map $\pi_1(T) \to \pi_1 \big(\#_h(S^2 \times S^1)\big)$ cannot be injective, so the Dehn Lemma applies.
\end{proof}

Let $G$ be a decorated graph that encodes a shadow $X$ of some block $M$. Every edge $e$ of $G$ determines a simple closed curve $\gamma$ in some region of $X$ and hence a torus $T=(\pi|_{\partial N(X)})^{-1}(\gamma)$ in the decomposition of $\partial N(X)$ described above. Here $\pi\colon N(X) \to X$ is the projection.

By the proposition just stated, the torus $T$ has a compressing disc $D\subset \partial N(X\cup \partial M) \isom \#_h(S^2 \times S^1)$. We now show that we can add $D$ to $X$. To do so, we enlarge $D$ to a disc $D' \supset D$ by adding a vertical annulus
contained in the vertical solid torus $\pi^{-1}(\gamma)$, so that $\partial D' \subset X$. Moreover, we slightly perturb $D'$ so that $\partial D'$ is a generic closed curve in $X$ contained in a neighbourhood of $\gamma$, to ensure that $X' = X \cup D'$ is a simple polyhedron.

\begin{prop}
The shadow $X' = X\cup D'$, equipped with appropriate gleams near $D'$, is again a shadow of $M$.
\end{prop}
\begin{proof}
The complement of $N(X\cup \partial M)$ in $M$ is a 1-handlebody, and $D'$ is parallel to its boundary by construction. Therefore adding $D'$ is like adding a trivial 2-handle that cancels with some 3-handle. The complement of $N(X' \cup \partial M)$ is still a 1-handlebody, with one 1-handle more. So $X'$ is a shadow for $M$.
\end{proof}

The closed curve $\partial D \subset T$ projects to a curve in $\gamma$ that winds some $p\geq 0$ times around $\gamma$. The cases $p=0, 1$, and $2$ are of particular interest for us: the modification from $X$ to $X'$ is shown in Figure \ref{add_disc:fig} in these cases. The half-integer $n$ is determined by how many times $\partial D$ winds along the fiber of the fibration $T\to \gamma$.

\begin{figure}
\begin{center}
\includegraphics[width = 16 cm]{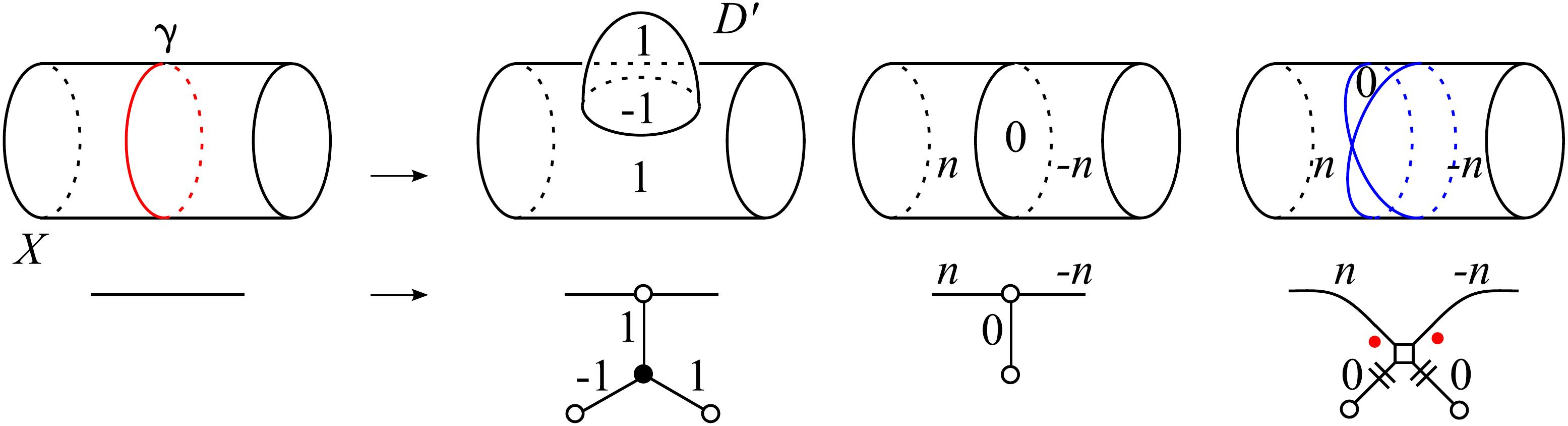}
\nota{How to add a disc $D'$ that winds $p=0, 1,$ and $2$ times respectively. The move is also shown below using graphs. Note the piece $X_{11}$ on the right. On the right, a 0-gleamed disc $D'$ is attached to the blue curve.}
\label{add_disc:fig}
\end{center}
\end{figure}

If $p=0$ or $p=1$ we say that the compressing disc $D$ is \emph{vertical} or \emph{horizontal} respectively. Both these cases were studied in \cite{Ma:zero}.

\begin{figure}
\begin{center}
\includegraphics[width =12 cm]{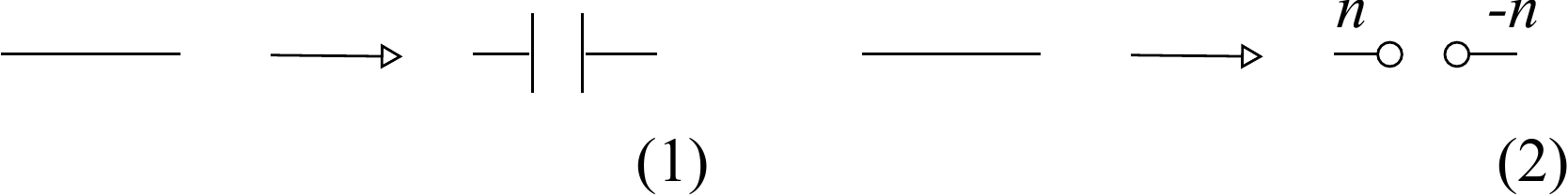}
\nota{If the torus $T$ above $\gamma$ has a vertical or horizontal compressing disc we can perform respectively the move (1) and (2).}
\label{hv:fig}
\end{center}
\end{figure}

\begin{prop}[Proposition 7.6 in \cite{Ma:zero}]
If $D$ is horizontal or vertical, the corresponding move in Figure \ref{hv:fig} transforms $X$ into a new shadow $X_*$ of a new block $M_*$. The block $M$ is obtained from $M_*$ by connected sum or assembling.
\end{prop}
\begin{proof}
Apply the converse of Figure \ref{sum_ass_new:fig} to $X'$ to get $X_*$.
\end{proof}

When $D$ is horizontal or vertical, we can cut the graph $G$ along the edge $e$ as shown in the figure. This leads to a simplification that will allow us to proceed by induction in many cases.

\subsection{Submanifolds of $\#_h(S^2 \times S^1)$}
Along the proof, we will use as a crucial tool the following lemma, which is peculiar of the manifolds of type $\#_h(S^2 \times S^1)$. Here $h\geq 0$ is any non-negative integer. Recall that a \emph{slope} in a torus is a non oriented non-trivial simple closed curve.

\begin{lemma} \label{peculiar:lemma}
Let $M\subset \#_h(S^2 \times S^1)$ be any connected submanifold with $\partial M$ consisting of tori $T_1,\ldots, T_k$. Each $T_i$ contains a slope $s_i$ that bounds a disc in $\#_h(S^2 \times S^1)$, such that 
by Dehn filling $M$ along $s_1, \ldots, s_k$ we get $\#_{h'}(S^2 \times S^1)$ for some $h' \geq 0$.
\end{lemma}
\begin{proof}
We prove the lemma by induction on $k$. The case $k=0$ is void, so we look at the generic case. 

Every $T_i$ has a slope $s_i$ that bounds a compressing disc in $\#_h(S^2 \times S^1)$. By an innermost argument there is one slope, say $s_1$, that bounds a disc $D$ entirely contained in $M$ or entirely outside $M$. By surgerying $T_1$ along $D$ we get a sphere $S$ that lies inside or outside $M$. 

Suppose that $S$ lies outside $M$. After surgerying $\#_h(S^2 \times S^1)$ along $S$ (that is, cutting along $S$ and capping the two new boundary components with balls: this operation transforms $\#_h(S^2 \times S^1)$ into one or two manifolds that are again of type $\#_{h'}(S^2 \times S^1)$) we may suppose that $S$ bounds a ball outside. Now $T_1$ bounds a solid torus outside $M$. We add the solid torus to $M$, to get a new $M'$ with one boundary component less, and we conclude by induction on $k$. 

Suppose that $S$ is inside $M$. Then $M = M' \# (D^2 \times S^1)$. We surger $\#_h(S^2 \times S^1)$ along $S$. The ambient manifold is still $\#_{h'}(S^2 \times S^1)$ for some $h'$, and $M$ has changed into $M'$, with one boundary torus less. We conclude by induction on $k$.
\end{proof}

\begin{rem}
In the statement of Lemma \ref{peculiar:lemma}, it may occur that the complement of $M$ in $\#_h(S^2 \times S^1)$ consists of solid tori: in this simple case the curves $s_i$ are the meridians of these solid tori and $h'=h$.

However, more complicated cases may also arise. It may be that $h'> h$ and some of the discs bounded by $s_i$ lie inside $M$. For instance, if $M\subset S^3$ is a knotted solid torus, then the slope $s$ is the meridian of the solid torus (there is no other choice) and by Dehn filling $M$ along $s$ we get $S^2 \times S^1$.
\end{rem}

\section{Exceptional fillings on the hyperbolic manifolds $W_1, \ldots, W_{11}$}
\label{exceptional:section}
We now study the hyperbolic manifolds $W_1,\ldots, W_{11}$ that fiber above the pieces $X_1,\ldots, X_{11}$. 
In light of Lemma \ref{peculiar:lemma}, we are interested in understanding when a Dehn filling of these manifolds gives rise to $\#_h(S^2\times S^1)$. We solve this problem completely in this section.

\subsection{Link surgery description} \label{link:surgery:subsection}
A presentation of $W_i$ as a link complement in $\#_2(S^2 \times S^1)$ is given in Figure \ref{M:fig} for all $i=1,\ldots, 11$. One important tool here is of course SnapPy \cite{Sna}. We have a fibration $W_i \to X_i$. We think of $W_i$ as a compact manifold bounded by tori, but sometimes we call $W_i$ also its hyperbolic interior for simplicity.

On every boundary torus $T$ of $W_i$ we will always use the meridian/longitude coordinates that are induced by this link diagram description. With this convention, the slope $\infty$ denotes the vertical simple closed curve (that is, the fibre of the fibration $T \to \gamma$ where $\gamma \subset \partial X_i$ is the boundary component corresponding to $T$) and all the horizontal curves (that is, the sections of the fibration $T \to \gamma$) will be integers.

\subsection{Cusp shapes}
As shown in \cite{CoThu}, each of the hyperbolic manifolds $W_i$ decomposes into two ideal regular octahedra and has volume $7.32772\ldots $ It fibers over the corresponding piece $X_i$, with one cusp for each boundary component $\gamma \subset \partial X$. Therefore $W_i$ has between 1 and 4 cusps, depending on $i=1,\ldots, 11$.

As shown in \cite{CoThu, CoFriMaPe}, the hyperbolic manifold $W_i$ has a maximal cusp section, obtained simply by matching the maximal (unit square) sections of the two regular ideal octahedra. The component of this maximal cusp section corresponding to the cusp lying above $\gamma$ is a flat torus $T$ as in Figure \ref{cusps:fig}. The flat torus $T$ is determined by two parameters: the length $q$ of $\gamma$ and whether the adjacent annulus $A$ is an even region or not. Both parameters can be found by looking at the vertex representing $X_i$ in Figure \ref{Xi:fig}, see Section \ref{encoding:graph:subsection}.

\begin{figure}
\begin{center}
\includegraphics[width = 11 cm]{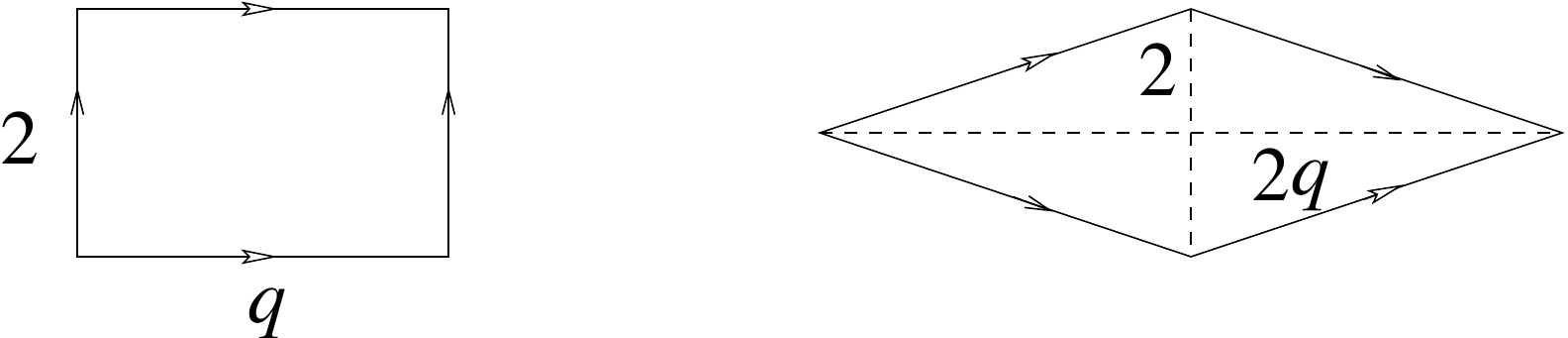}
\nota{The component $T$ of the maximal cusp above a boundary curve $\gamma$ of $\partial X_i$ of length $q$. The case depends on whether the adjacent annular region is even (left) or odd (right). In both cases opposite edges must be identified via a translation to get a flat torus.}
\label{cusps:fig}
\end{center}
\end{figure}

The vertical curve (that is, the fibre of the fibration $T\to \gamma$) is the length-2 vertical one in the picture, and the horizontal curves (that is, the sections of the fibration $T \to \gamma$) are those that intersect the vertical curve in one point. In Figure \ref{cusps:fig}-(left) there is a single shortest horizontal curve of length $q$. In Figure \ref{cusps:fig}-(right) there are two shortest horizontal curves, both of length $\sqrt{1+q^2}$. There is also a curve of length $2q$, that despite being horizontal in the picture it is not horizontal according to our definition, since it intersects the vertical curve in two points: indeed this curve winds twice along $\gamma$.

Recall that a Dehn filling on a hyperbolic manifold is \emph{exceptional} if the resulting manifold is not hyperbolic. The manifold $\#_h(S^2 \times S^1)$ is of course not hyperbolic.

\subsection{Two notable manifolds}
It is shown in \cite{IsKo} that
the manifolds $W_8$ and $W_{11}$ are diffeomorphic to the complements of two notable links in $S^3$, the Borromean link and the minimally twisted chain link with 4 components drawn in Figure \ref{chain_four:fig}. 

\begin{figure}
\begin{center}
\includegraphics[width = 8 cm]{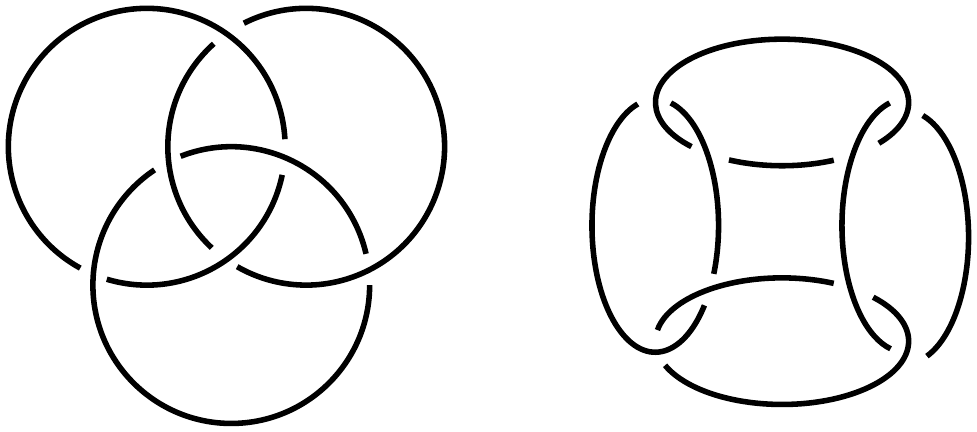}
\nota{The manifolds $W_8$ and $W_{11}$ are diffeomorphic to the complements of two notable links in $S^3$ shown here: the Borromean rings and the minimally twisted chain link $L_4$ with 4 components.}
\label{chain_four:fig}
\end{center}
\end{figure}

Note that all the 3 cusp shapes in $W_8$ are similar to a $1\times 2$ rectangle, while the 4 cusp shapes in $W_{11}$ are all squares.

Armed with patience, we now start to classify all the Dehn fillings of $W_1,\ldots, W_{11}$ that produce a manifold diffeomorphic to $\#_h(S^2 \times S^1)$.

\subsection{The manifolds $W_1$ and $W_2$}
The pieces $X_1$ and $X_2$ have each one boundary component, and determine two hyperbolic manifolds $W_1, W_2$ with one cusp. A Dehn filling is determined by a slope $\alpha \in \matQ\cup \{\infty\}$.

\begin{prop} \label{W12:prop}
A Dehn filling $\alpha$ on $W_1$ or $W_2$ gives $\#_h(S^2 \times S^1)$ if and only if $\alpha = \infty$ and $h=2$.
\end{prop}
\begin{proof}
If $\alpha\neq 0,\infty$ the slope length in the Euclidean maximal cusp section is $>6$ and hence the Dehn filling cannot be $\#_h(S^2 \times S^1)$ by the ``6 Theorem'' of Agol and Lackenby \cite{Ag, La}. If $\alpha = 0$ we get a Haken manifold \cite{CoFriMaPe}.
\end{proof}

\subsection{The manifolds $W_3$ and $W_4$}
The pieces $X_3$ and $X_4$ have two boundary components, of length 1 and 5. Therefore $W_3$ and $W_4$ have each two cusps. A Dehn filling is determined by a pair $(\alpha, \beta)$ of slopes $\alpha, \beta \in \matQ \cup \{\infty\}$. Let $\alpha$ and $\beta$ be the slopes corresponding respectively to the boundary components of length 1 and 5.

\begin{prop} \label{W34:prop}
A Dehn filling $(\alpha, \beta)$ on $W_3$ or $W_4$ gives $\#_h(S^2 \times S^1)$ if and only if one of the following holds: 
\begin{itemize}
\item $\alpha = \infty , \beta = \infty$, and $h=2$, or 
\item $\alpha \in \matZ$, $\beta = \infty$, and $h=1$, or
\item $\alpha = 0, \beta \in \matZ$, and $h=0$.
\end{itemize}
\end{prop}
\begin{proof}
From the surgery description, we can compute the first homology groups of the Dehn fillings on $W_3$ and $W_4$. 
In both cases, these are $(\matZ/_{q_1\matZ})\oplus(\matZ/_{q_2\matZ})$ 
where we write $\alpha=\frac{p_1}{q_1}$ and $\beta=\frac{p_2}{q_2}$ as irreducible fractions. 
This group has no torsion if and only if $\alpha, \beta\in\matZ\cup\{\infty\}$. 
The case $\alpha, \beta \in \matZ$ was proved in \cite{Co}. 
Suppose $\alpha=\infty$ and $\beta\in\matZ$. 
We see that the Dehn fillings on $W_3$ and $W_4$ are obtained by knot surgeries as shown in Figures \ref{Kirby_L3:fig} and \ref{Kirby_L4:fig}, respectively. 
Those knots are not the unknot for any $\beta\in\matZ$ (as one can see by calculating their Alexander polynomial) and thus they do not produce $\#_h(S^2\times S^1)$. 
\begin{figure}
\begin{center}
\includegraphics[width = 15 cm]{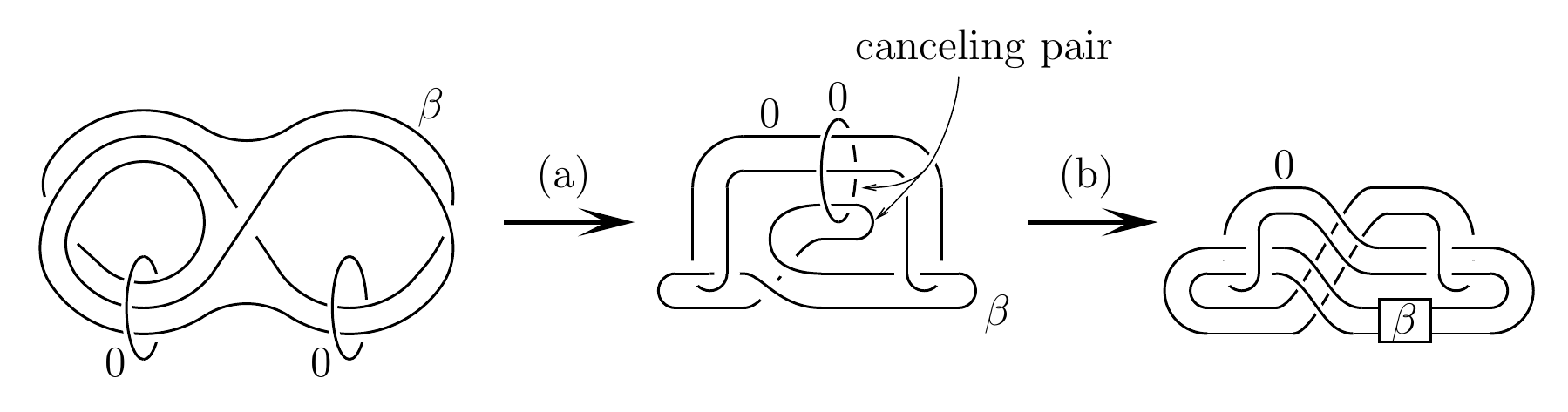}
\nota{We modify the diagram by isotopy (a) and deleting a cancelling pair (b)}
\label{Kirby_L3:fig}
\end{center}
\end{figure}
\begin{figure}
\begin{center}
\includegraphics[width = 15 cm]{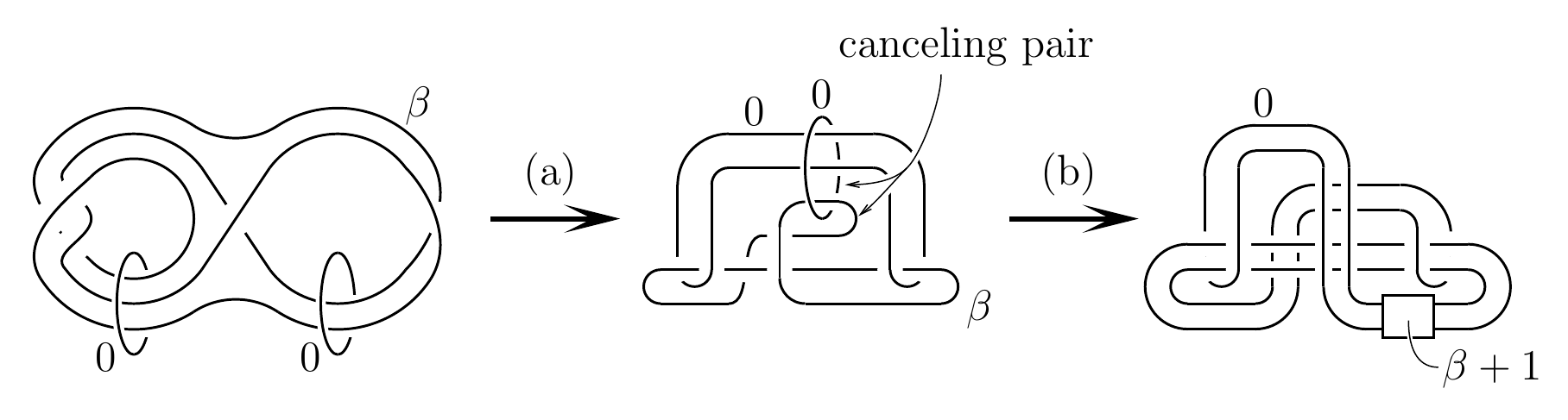}
\nota{We modify the diagram by isotopy (a) and deleting a cancelling pair (b)}
\label{Kirby_L4:fig}
\end{center}
\end{figure}
\end{proof}

\subsection{The manifolds $W_5$ and $W_6$}
The piece $X_5$ has two boundary components, of length 1 and 5, while $X_6$ has two boundary components of length 2 and 4. Therefore $W_5$ and $W_6$ have each two cusps. A Dehn filling is determined by a pair $(\alpha, \beta)$ of slopes $\alpha, \beta \in \matQ \cup \{\infty\}$. Let $\alpha$ and $\beta$ be the slopes corresponding respectively to the boundary components of length 1 (2) and 5 (4).

\begin{prop} \label{W56:prop}
A Dehn filling $(\alpha, \beta)$ on $W_5$ or $W_6$ gives $\#_h(S^2 \times S^1)$ if and only if one of the following holds: 
\begin{itemize}
\item $\alpha = \infty , \beta = \infty$, and $h=2$, or 
\item $\alpha \in \matZ$, $\beta = \infty$, and $h=1$.
\end{itemize}
\end{prop}
\begin{proof}
Let us consider the $4$-component link in $S^3$ shown in Figure \ref{M:fig}-(5,6) and let $W^0_i$ be the complement of this link. 
Set $\alpha=\frac{p_1}{q_1}$ and $\beta=\frac{p_2}{q_2}$. 

The first homology group of the Dehn filling on $W_5$ is isomorphic to $\matZ^4/_{\mathrm{Im}f_5}$, 
for some linear map $f_5\colon \matZ^4\to H_1(W^0_5,\matZ) \isom \matZ^4$ (the isomorphism is obtained by taking the meridians as a basis) that one can infer from the diagram. The map $f_5$ is represented by the following matrix, which changes by elementary transformations as indicated: 
$$    \left(\begin{matrix}
      p_1 & 0 & 1 & 0 \\
      0 & p_2 & 0 & 3 \\
      q_1 & 0 & 0 & 0 \\
      0 & 3q_2 & 0 & 0 
    \end{matrix}\right) \to    \left(\begin{matrix}
      0 & 0 & 1 & 0 \\
      0 & p_2 & 0 & 3 \\
      q_1 & 0 & 0 & 0 \\
      0 & 3q_2 & 0 & 0 
    \end{matrix}\right).$$ 
Suppose that the homology group of the Dehn filling on $W_5$ is tosion free. 
Then $3$ and $p_2$ are coprime and hence
$p_2\equiv \pm1\ (\mathrm{mod}\ 3)$. 
The matrix further transforms as follows. 
\[
    \left(\begin{matrix}
      0 & 0 & 1 & 0 \\
      0 & \pm1 & 0 & 3 \\
      q_1 & 0 & 0 & 0 \\
      0 & 3q_2 & 0 & 0 
    \end{matrix}\right) \to 
    \left(\begin{matrix}
      0 & 0 & 1 & 0 \\
      0 & \pm1 & 0 & 1 \\
      q_1 & 0 & 0 & 0 \\
      0 & 3q_2 & 0 & \mp6q_2 
    \end{matrix}\right) \to 
    \left(\begin{matrix}
      0 & 0 & 1 & 0 \\
      0 & 0 & 0 & 1 \\
      q_1 & 0 & 0 & 0 \\
      0 & 9q_2 & 0 & \mp6q_2 
    \end{matrix}\right) \to 
    \left(\begin{matrix}
      0 & 0 & 1 & 0 \\
      0 & 0 & 0 & 1 \\
      q_1 & 0 & 0 & 0 \\
      0 & 9q_2 & 0 & 0 
    \end{matrix}\right)
\]
We have $(q_1,q_2)=(0,0)$ or $(1,0)$, that is, 
$\alpha = \infty , \beta = \infty$ or $\alpha \in \matZ$, $\beta = \infty$. 

We next turn to the case $W_6$. 
Similarly, the first homology group of the Dehn filling on $W_6$ is $\matZ^4/_{\mathrm{Im}f_6}$, 
where $f_6$ is represented by the matrix
$$    \left(\begin{matrix}
      p_1 & 0 & -1 & 1 \\
      0 & p_2 & 2 & 2 \\
      -q_1 & 2q_2 & 0 & 0 \\
      q_1 & 2q_2 & 0 & 0 
    \end{matrix}\right).$$
The determinant of this matrix is $16q_1q_2$. 
Suppose that the homology has no torsion. 
Then we have $q_1q_2=0$. 
In the case $p_1=1, q_1=0$, 
the matrix changes as follows 
\[
    \left(\begin{matrix}
      1 & 0 & -1 & 1 \\
      0 & p_2 & 2 & 2 \\
      0 & 2q_2 & 0 & 0 \\
      0 & 2q_2 & 0 & 0 
    \end{matrix}\right)\to
    \left(\begin{matrix}
      1 & 0 & 0 & 0 \\
      0 & p_2 & 2 & 0 \\
      0 & 2q_2 & 0 & 0 \\
      0 & 0 & 0 & 0 
    \end{matrix}\right).
\] 
Thus $p_2$ and $2$ are coprime. 
\[
    \left(\begin{matrix}
      1 & 0 & 0 & 0 \\
      0 & 1 & 2 & 0 \\
      0 & 2q_2 & 0 & 0 \\
      0 & 0 & 0 & 0 
    \end{matrix}\right)\to
    \left(\begin{matrix}
      1 & 0 & 0 & 0 \\
      0 & 1 & 0 & 0 \\
      0 & 2q_2 & -4q_2 & 0 \\
      0 & 0 & 0 & 0 
    \end{matrix}\right)\to
    \left(\begin{matrix}
      1 & 0 & 0 & 0 \\
      0 & 1 & 0 & 0 \\
      0 & 0 & -4q_2 & 0 \\
      0 & 0 & 0 & 0 
    \end{matrix}\right)
\]
Then $q_2$ is also $0$. 
Hence $\alpha = \infty , \beta = \infty$. 
In the case $p_2=1, q_2=0$, 
we have 
\[
    \left(\begin{matrix}
      p_1 & 0 & -1 & 1 \\
      0 & 1 & 2 & 2 \\
      -q_1 & 0 & 0 & 0 \\
      q_1 & 0 & 0 & 0 
    \end{matrix}\right)\to
    \left(\begin{matrix}
      0 & 0 & -1 & 0 \\
      0 & 1 & 0 & 0 \\
      -q_1 & 0 & 0 & 0 \\
      0 & 0 & 0 & 0 
    \end{matrix}\right). 
\]
Then $q_1$ is $0$ or $1$. Hence 
$\alpha = \infty , \beta = \infty$, or  $\alpha \in \matZ$, $\beta = \infty$. 

Conversely, in all the cases listed we easily check that the Dehn filled manifold is indeed homeomorphic to $\#_h(S^2 \times S^1)$ as stated.
\end{proof}

\subsection{The manifold $W_7$}
The piece $X_7$ has two boundary components, both of length 3, and a symmetry that interchanges them. Therefore $W_7$ has two cusps, and an isometry that interchanges them. A Dehn filling is determined by a pair $(\alpha, \beta)$ of slopes $\alpha, \beta \in \matQ \cup \{\infty\}$. The order does not matter. 

\begin{prop} \label{W7:prop}
A Dehn filling $(\alpha, \beta)$ on $W_7$ gives $\#_h(S^2 \times S^1)$ if and only if one of the following holds: 
\begin{itemize}
\item $\alpha = \infty , \beta = \infty$, and $h=2$, or 
\item $\alpha \in \matZ$, $\beta = \infty$, and $h=1$, or
\item $\alpha = \infty$, $\beta \in \matZ$, and $h=1$.
\end{itemize}
\end{prop}
\begin{proof}
Set $\alpha=\frac{p_1}{q_1}$ and $\beta=\frac{p_2}{q_2}$. 
As in the proof of Proposition \ref{W56:prop}, 
it is easy to check that the first homology group of the Dehn filling $(\alpha,\beta)$ on $W_7$ is isomorphic to
$\matZ^4/_{\mathrm{Im}  f}$ where $f\colon \matZ^4 \to \matZ^4$ is encoded by the matrix
$$\left(\begin{matrix}
       p_1 & -q_2 & 1 & 2 \\
      -q_1 &  p_2 & 2 & -1 \\
       q_1 & 2q_2 & 0 & 0 \\
      2q_1 & -q_2 & 0 & 0 
    \end{matrix}\right).$$
This matrix has determinant $25q_1q_2$. 
If this homology group has no torsion, 
then $q_1q_2=0$. 
Up to symmetry we may suppose that $p_1=1, q_1=0$. The matrix changes as follows: 
\[
 \left(\begin{matrix}
      1 & -q_2 & 1 & 2 \\
      0 &  p_2 & 2 & -1 \\
      0 & 2q_2 & 0 & 0 \\
      0 & -q_2 & 0 & 0 
    \end{matrix}\right) \to
 \left(\begin{matrix}
      1 & 0 & 0 & 0 \\
      0 &  0 & 0 & -1 \\
      0 & 0 & 0 & 0 \\
      0 & -q_2 & 0 & 0 
    \end{matrix}\right).
\]
Thus we have $q_2=0$ or $1$. Conversely, in all the cases listed we easily check that the Dehn filled manifold is indeed homeomorphic to $\#_h(S^2 \times S^1)$ as stated.
\end{proof}

\subsection{The manifold $W_8$}
As already mentioned, the manifold $W_8$ is diffeomorphic to the complement of the 
Borromean rings shown in Figure \ref{chain_four:fig}-(left). We therefore study its Dehn surgeries.

\begin{prop} \label{W8:prop}
A Dehn surgery $(\alpha, \beta, \gamma)$ of the Borromean rings produces $\#_h(S^2 \times S^1)$ if and only if up to interchanging $\alpha, \beta$ and $\gamma$ one of the following holds: 
\begin{itemize}
\item $\alpha =\infty$, $\beta = 0$, $\gamma = 0$, and $h=2$, or
\item $\alpha =\infty$, $\beta = \frac 1m, m \in \matZ$, $\gamma = 0$, and $h=1$, or
\item $\alpha =\infty$, $\beta = \frac 1m$, $\gamma = \frac 1n, m, n \in \matZ$, and $h=0$. 
\end{itemize}
\end{prop}
\begin{proof}
Assume that a Dehn surgery $(\alpha, \beta, \gamma)$ of the Borromean rings produces $\#_h(S^2 \times S^1)$ for some $h$. 
Its first homology group has no torsion. 
Since the pairwise linking numbers of the Borromean rings are $0$, 
the coefficients $\alpha, \beta, \gamma$ are in $\{0\}\cup\{\frac 1n\mid n\in\matZ\}$. 

If $\alpha = \beta = \gamma = 0$, the Dehn surgery is the $3$-torus. 

In the case $\alpha = \frac 1n, n \in \matZ$, $\beta = 0$, $\gamma = 0$, 
the Dehn surgery is diffeomorphic to the Seifert manifold $\big(T, (n,-1)\big)$ except for $n=0$. 
If $n=0$, the Dehn surgery is actually $(S^2 \times S^1)\# (S^2 \times S^1)$. 

Suppose that at least two of $\alpha, \beta, \gamma$ are in $\{\frac 1n\mid n\in\matZ\}$. Up to symmetry we may suppose $\alpha = \frac 1l$ and $\beta = \frac 1m$ for $l, m \in \matZ$. 
By performing two Rolfsen twists along the components with coefficients $\alpha$ and $\beta$, 
we see that the resulting manifold is obtained by a knot surgery as shown in Figure \ref{double_twist:fig}. 
This knot must be the unknot. Hence we have $l=0$ or $m=0$. 
If $\gamma = 0$ it gives $S^2 \times S^1$, and $S^3$ otherwise. 
\begin{figure}
\begin{center}
\includegraphics[width = 3 cm]{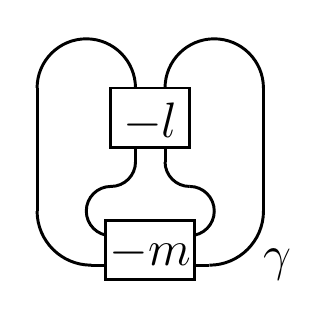}
\nota{The double twist knot with $-l$ and $-m$ full twists. }
\label{double_twist:fig}
\end{center}
\end{figure}
\end{proof}

We now turn back to $W_{8}$ with the usual meridian/longitude basis described in Section \ref{link:surgery:subsection}. The piece $X_8$ has 3 boundary components, of order 1, 1, and 4, and a symmetry that interchanges the first two. 
Therefore $W_8$ has 3 cusps, and an isometry that interchanges the first two (it has also more isometries that are not apparent from this description).
A Dehn filling is determined by a triple $(\alpha, \beta, \gamma)$ of slopes $\alpha, \beta, \gamma \in \matQ \cup \{\infty\}$ where $\alpha$ and $\beta$ correspond to the boundary components of order 1. 

As already mentioned, we can regard this Dehn filling as a Dehn surgery of the Borromean rings by performing slam-dunks as in Figure \ref{W8slamdunk:fig}. 

\begin{figure}
\begin{center}
\includegraphics[width = 13 cm]{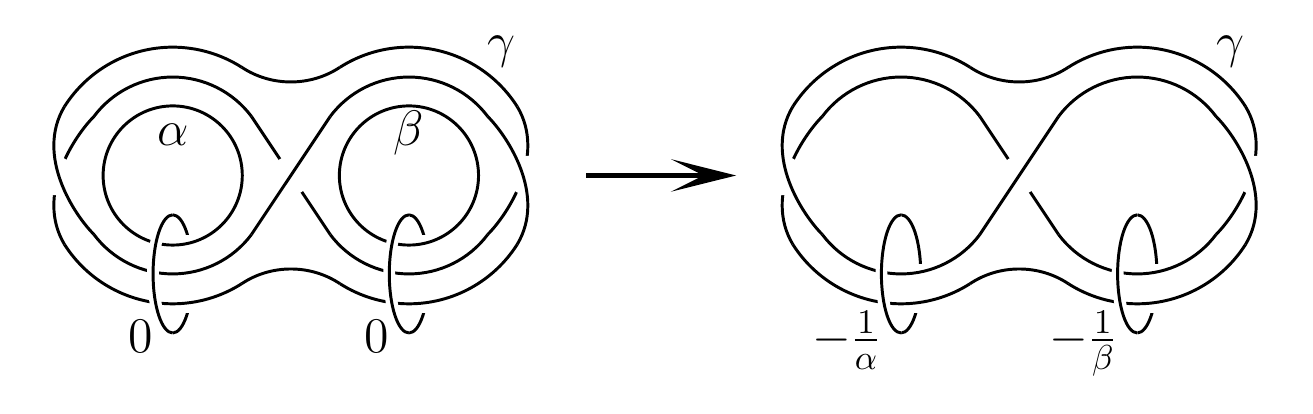}
\nota{This figure shows slam-dunk operations on the components with coefficients $\alpha$ and $\beta$. The link pictured in the right part is the Borromean rings. The manifold $W_8$ is diffeomorphic to the complement of the Borromean rings.}
\label{W8slamdunk:fig}
\end{center}
\end{figure}

\begin{cor} \label{W8:cor}
A Dehn filling $(\alpha, \beta, \gamma)$ on $W_8$ gives $\#_h(S^2 \times S^1)$ if and only if up to interchanging $\alpha$ and $\beta$ one of the following holds: 
\begin{itemize}
\item $\alpha = \infty , \beta = \infty, \gamma = \infty$, and $h=2$, or 
\item $\alpha = 0, \beta = \infty, \gamma = 0$, and $h=2$, or
\item $\alpha =0$, $\beta \in \matZ$, $\gamma = 0$, and $h=1$, or
\item $\alpha =0$, $\beta = \infty$, $\gamma = \frac 1n, n \in \matZ$, and $h=1$, or
\item $\alpha \in \matZ$, $\beta = \infty$, $\gamma = \infty$, and $h=1$, or
\item $\alpha = 0$, $\beta \in \matZ$, $\gamma = \frac 1n, n \in \matZ$, and $h=0$, or
\item $\alpha, \beta \in \matZ$, $\gamma = \infty$, and $h=0$.
\end{itemize}
\end{cor}
\begin{proof}
A Dehn filling $(\alpha, \beta, \gamma)$ on $W_8$ is the same as a Dehn surgery $(-\frac{1}{\alpha},-\frac{1}{\beta},\gamma)$ of the Borromean rings as shown in Figure \ref{W8slamdunk:fig}. 
\end{proof}

\subsection{The manifold $W_9$}
The piece $X_9$ has 3 boundary components, of length 1, 1, and 4, where the first is adjacent to an even region and the second to an odd one. Therefore $W_9$ has 3 cusps. A Dehn filling is determined by a triple $(\alpha, \beta, \gamma)$ of slopes $\alpha, \beta, \gamma \in \matQ \cup \{\infty\}$ where $\alpha$ and $\beta$ correspond to the boundary components of length 1, with $\alpha$ to the one adjacent to the even region.

\begin{prop} \label{W9:prop}
If a Dehn filling $(\alpha, \beta, \gamma)$ on $W_9$ gives $\#_h(S^2 \times S^1)$ then one of the following holds: 
\begin{itemize}
\item $\alpha=0$, $\beta\in\matZ$, $\gamma = 0$ and $h=1$, or
\item $\alpha=\infty$, $\beta\in\matZ$, $\gamma = \infty$ and $h=1$, or
\item $\alpha=0$, $\beta\in\matZ$, $\gamma = \frac 1n, n\in\matZ$ and $h=0$, or
\item $\alpha\in\matZ$, $\beta\in\matZ$, $\gamma = \infty$ and $h=0$, or
\item $\beta=\infty$. 
\end{itemize}
\end{prop}
\begin{proof}
The Dehn filling $(\alpha,\beta,\gamma)$ on $W_9$ is described in Figure \ref{W9_diag:fig}, 
and it is equivalent to a Dehn surgery along a $3$-component link as shown in the rightmost figure. 
\begin{figure}
\begin{center}
\includegraphics[width = 15 cm]{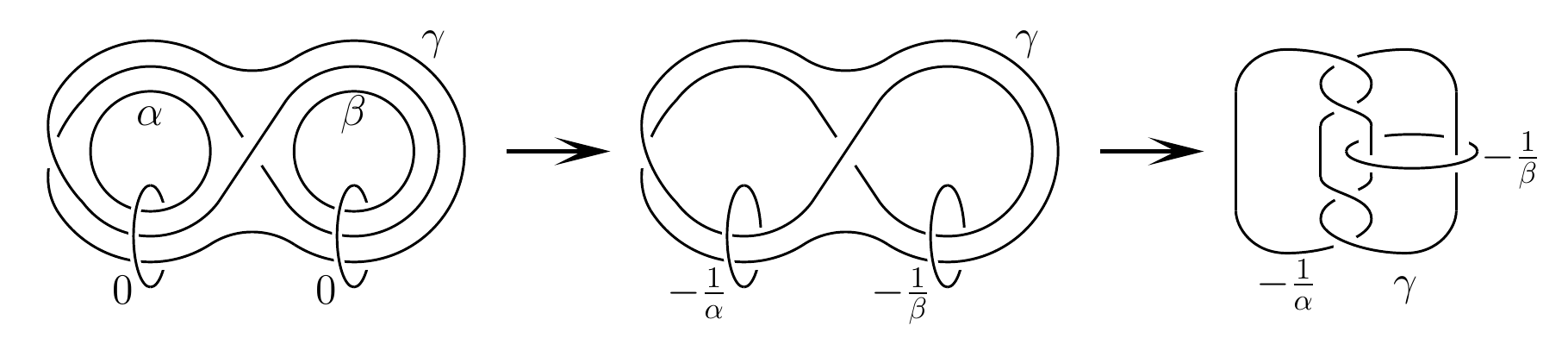}
\nota{The Dehn filling $(\alpha,\beta,\gamma)$ on $W_9$.}
\label{W9_diag:fig}
\end{center}
\end{figure}
Set $\alpha=\frac{p_1}{q_1}$, $\beta=\frac{p_2}{q_2}$ and $\gamma=\frac{p_3}{q_3}$. 
From the surgery diagram, it is easy to check that 
the first homology of the Dehn filling $(\alpha,\beta,\gamma)$ on $W_9$ is
$\matZ^3/_{\mathrm{Im} f}$ with $f$ represented by the matrix
$$\left(\begin{matrix}
       q_1 &   0 & 2q_3 \\
         0 & q_2 &    0 \\
     -2p_1 &   0 &  p_3 
      \end{matrix}\right).$$
This matrix has determinant $q_2(4p_1q_3+q_1p_3)$. 
If the homology group has no torsion, 
then one of the following holds:
\begin{enumerate}
\item $|q_2|=1$ and $4p_1q_3+q_1p_3=0$, or 
\item $|q_2|=1$ and $|4p_1q_3+q_1p_3|=1$, or 
\item $q_2=0$ and $4p_1q_3+q_1p_3=0$, or 
\item $q_2=0$ and $|4p_1q_3+q_1p_3|=1$. 
\end{enumerate}

In the cases (1) and (2), we have $|q_2|=1$ and
the Dehn filling $(\alpha,\beta,\gamma)$ on $W_9$ is obtained by a Dehn surgery $(-\frac{1}{\alpha},\gamma)$ along a $2$-bridge link. 
By Wu's result \cite[Theorem 5.1]{Wu}, 
the resulting manifold is a laminar $3$-manifold if $-\frac{1}{\alpha}$ and $\gamma$ are both different from $\infty$. 
Assume $\frac{1}{\alpha}=\infty$, that is $\alpha=0$. 
Then we have $\gamma=\frac 1n,n\in\matZ$ in case (1), and $\gamma=0$ in case (2). 
In these cases, the Dehn filling is actually $S^3$ or $S^2\times S^1$. 
Assume $\gamma=\infty$. 
Then we have $\alpha=\infty$ if (1), and $\alpha \in \matZ$ if (2). 
Also in these cases, the Dehn filling is $S^3$ or $S^2\times S^1$. 
\end{proof}

\subsection{The manifold $W_{10}$}
Using SnapPy we see that $W_{10}$ is diffeomorphic to the manifold shown in Figure \ref{chain5:fig}, namely the complement of the minimally twisted chain link $L_5$ with 5 components, two of which are surgered with coefficients $-2$ and $-\frac 12$.

\begin{figure}
\begin{center}
\includegraphics[width = 5.5 cm]{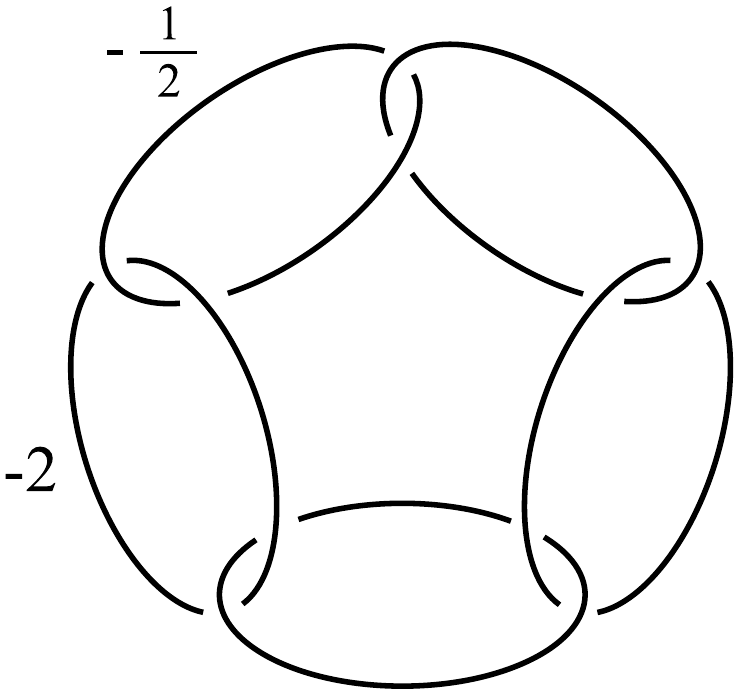}
\nota{The manifold $W_{10}$ is diffeomorphic to the complement of the minimally twisted chain link $L_5$ with 5 components, two of which are surgered as shown.}
\label{chain5:fig}
\end{center}
\end{figure}

The exceptional Dehn surgeries on $L_5$ have been completely classified in \cite{MaPeRo}, so to understand which fillings of $W_{10}$ give $\#_h(S^2 \times S^1)$ we only need to apply carefully the results stated there. We discover the following. The components in $L_5$ are oriented in clockwise order, and so are the Dehn filling coefficients.

\begin{prop}
If a Dehn surgery $(-2, -\frac 12, \alpha, \beta, \gamma)$ of $L_5$ produces $\#_h(S^2 \times S^1)$, then at least one of the following 3 conditions holds:
$$\alpha \in \{0,1,\infty\}, \qquad \beta \in \{0,1,2,\infty\}, \qquad \gamma \in \{0,1,\infty\}.$$
Moreover:
\begin{itemize}
\item if $\beta =0$ we either get $\alpha = \infty$ and $\gamma \in \matZ$, or $\alpha \in \matZ$ and $\gamma = \infty$,
\item if $\gamma = \infty$ we either get $\alpha = \infty$ and $\beta \in \matZ$, or $\alpha \in \matZ$ and $\beta = \infty$.
\end{itemize}
\end{prop}
\begin{proof}
We first note from \cite[Theorem 0.1]{MaPeRo} that an isometry of $W_{10}$ induces the following symmetry on Dehn fillings:
\begin{equation} \label{sym:eqn}
\big(-2, -\tfrac 12, \alpha, \beta, \gamma\big) \longmapsto \big(-2, - \tfrac 12, 1-\alpha, \tfrac \beta{\beta -1}, 1-\gamma\big).
\end{equation}
Theorem 4.2 in \cite{MaPeRo} furnishes a complete list of exceptional surgeries on $L_5$ up to the action of the isometry group of its complement. By analysing carefully this list we find that every exceptional surgery of type $(-2, -\frac 12, \alpha, \beta, \gamma)$
fulfills one of the following requirements, up to applying the symmetry (\ref{sym:eqn}):
\begin{equation} \label{abc:eqn}
\alpha \in \{0,1,\infty\}, \qquad \beta \in \{0,1,2,\infty\}, \qquad \gamma \in \{0,1,\infty\},
\end{equation}
$$ (\alpha, \gamma) \in \big\{ (-3,2),\ (-2,2),\ (-1,2),\ (-1,\tfrac 32),\ (\tfrac 12, \tfrac 12) \big\}.$$
$$ (\alpha, \beta) \in \big\{ (-1,-1) \big\}, \qquad
(\beta, \gamma) \in \big\{ (-2,-1), \ (-1,-1) \big\},
$$
or $(\alpha, \beta, \gamma)$ belongs to a list of 15 triples. Using SnapPy we see that of these 15 triples, only two produce a closed manifold whose homology has no torsion, namely:
$$(\alpha, \beta, \gamma) \in \big\{(-3,-1,-2), \quad (-2,-1,-2)\big\}.$$
To conclude, it remains to show that in these cases we never get $\#_h(S^2 \times S^1)$ unless any of the conditions in (\ref{abc:eqn}) is satisfied. 

If $\gamma = -1$, the surgery $(-2, -\frac 12, \alpha, \beta, -1)$ on $L_5$ is equivalent to the surgery $(\frac 12, \alpha, \beta +2)$ on the chain link $L_3$ with 3 components. We prove this in two steps:
$$L_5\big(-2, -\tfrac 12, \alpha, \beta, -1\big) = L_4\big(-1, -\tfrac 12, \alpha, \beta+1\big) = L_3\big(\tfrac 12, \alpha, \beta+2\big)$$
where $L_4$ is the minimally twisted chain link with 4 components. Here we use two Fenn-Rourke moves, see \cite[Figure 5]{MaPeRo}.
The complement $S^3 \setminus L_3$ is the \emph{magic manifold} and its exceptional Dehn filling are fully described in \cite{MaPe}. By looking at \cite[Tables 2 and 3]{MaPe} we deduce that we never get $\#_h(S^2 \times S^1)$ unless $\alpha$ or $\beta$ equals $\infty$. (Note that all signs must be reversed when looking at the tables in \cite{MaPe} because the mirrored chain link is considered there.)
Using (\ref{sym:eqn}), the previous discussion applies also to the case $\gamma = 2$.

If $\alpha = -1$, the surgery $(-2, -\frac 12, -1, \beta, \gamma)$ is equivalent to the surgery $(5, 2-\beta, \frac \gamma{\gamma-1})$ on $L_3$. We prove this as follows:
$$L_5\big(-2, -\tfrac 12, -1, \beta, \gamma\big) = L_4\big(-2, \tfrac 12, \beta+1, \gamma\big) = 
L_4\big(4,-1,1-\beta,\tfrac{\gamma}{\gamma-1}\big) = L_3\big(5, 2-\beta, \tfrac \gamma{\gamma -1}\big)$$
where in the middle equality we use the symmetry \cite[Equation (3.15)]{MaPeRo}. We are only interested in the cases $(\alpha,\beta) = (-1,-1)$ and $(\alpha,\gamma)=(-1,\frac 32)$ and they both lead to $L_3(5,3, \delta)$ for $\delta = 1-\alpha$ or $\frac \beta{\beta-1}$. Again from \cite[Tables 2 and 3]{MaPe} we see that we do not get $\#_h(S^2 \times S^1)$ unless $\delta = \infty$, that is $\alpha = \infty$ or $\beta =1$.

If $(\alpha, \gamma) = (\tfrac 12, \tfrac 12)$, we get 
$$L_5\big(-2, -\tfrac 12, \tfrac 12, \beta, \tfrac 12\big) =
L_5\big(-\tfrac 12, 2, \tfrac 32, -1, 1-\beta \big) = L_4\big(-\tfrac 12, 2, \tfrac 52, 2-\beta \big) = L_4\big(\tfrac 53, 0, \tfrac 13, \tfrac{\beta}{\beta-1}\big)
$$
where we have used \cite[Theorem 0.1 and Equation (3.15)]{MaPeRo}. The manifold $L_4\big(\tfrac 53, 0, \tfrac 13, \tfrac{\alpha}{\alpha-1}\big)$ is a graph manifold and  \cite[Corollary 3.6]{MaPeRo} easily implies that it is not $\#_h(S^2 \times S^1)$.

Finally, if $(\alpha, \beta, \gamma) \in \{(-3,-1,-2), (-2,-1,-2)\}$ we can use the same techniques to show that the filled manifold is not $\#_h(S^2 \times S^1)$.

We now turn to the last assertion. If $\beta = 0$, and $\alpha =\frac pq, \gamma = \frac rs$, we get 
$$L_5\big(-2, -\tfrac 12, \tfrac pq, 0, \tfrac rs\big) = 
L_5\big(3, \tfrac 13, \tfrac{q-p}q, \tfrac sr, \infty \big)
 = \big( D, (3,-1), (3,1) \big) \bigcup_{\tiny{\matr 0110}} \big(D, (q,q-p), (s,-r)\big)$$
where we have used \cite[Equation (1.3) and Corollary 1.3]{MaPeRo}. To get $\#_h(S^2 \times S^1)$ here we must either have $\alpha = \infty$ and $\gamma \in \matZ$, or $\alpha \in \matZ$ and $\gamma = \infty$. Similarly, if $\gamma = \infty$, and $\alpha = \frac pq, \beta = \frac rs$, we get
$$L_5\big(-2, -\tfrac 12, \tfrac pq, \tfrac rs, \infty\big) =
\big( D, (2,1), (2,-1) \big) \bigcup_{\tiny{\matr 0110}} \big(D, (q,p), (r,-s)\big)$$
and we conclude analogously.
\end{proof}

We now turn back to $W_{10}$ with the usual meridian/longitude basis described in Section \ref{link:surgery:subsection}. A Dehn filling is determined by a triple $(\alpha, \beta, \gamma)$ of slopes $\alpha, \beta, \gamma \in \matQ \cup \{\infty\}$ where $\alpha, \beta$, and $\gamma$ correspond to the boundary components of $X_{10}$ of length 1, 2, and 3.

\begin{cor} \label{W10:cor}
If a Dehn filling $(\alpha, \beta, \gamma)$ on $W_{10}$ gives $\#_h(S^2 \times S^1)$ then one of the following 3 conditions holds: 
$$\alpha \in \big\{\infty, 0, \tfrac 12, 1\big\}, \qquad \beta \in \big\{\infty, -1, 0\big\}, \qquad
\gamma \in \big\{\infty, -3, -2 \big\}.$$
Moreover:
\begin{itemize}
\item  if $\alpha =\infty$ we either get $\beta = \infty$ and $\gamma \in \matZ$, or $\beta \in \matZ$ and $\gamma = \infty$,
\item if $\beta = \infty$ we either get $\alpha = \infty$ and $\gamma \in \matZ$, or $\alpha \in \matZ$ and $\gamma = \infty$.
\end{itemize}
\end{cor}
\begin{proof}
Use SnapPy to figure out the correct change of basis. 
\end{proof}

\subsection{The manifold $W_{11}$}
Using SnapPy we discover that $W_{11}$ is diffeomorphic to the complement of the minimally twisted chain link $L_4$  with 4 components, shown in Figure \ref{chain_four:fig}-(right). This is the smallest orientable hyperbolic manifold with 4 cusps \cite{Yo}, and its exceptional fillings have already been classified in \cite{MaPeRo}. 

All the cusp shapes are squares. Therefore at every cusp we have two shortest slopes and two second shortest slopes. These are respectively $(\infty,1)$ and $(0,2)$. 

\begin{prop} \label{X11:surgery:prop}
If a Dehn surgery along $L_4$ gives $\#_h(S^2 \times S^1)$ for some $h\geq 0$, then at least one of the 4 Dehn surgery coefficients is in the set $\{0,1,2,\infty\}$.
Moreover:
\begin{itemize}
\item if $\alpha = 0$ one of the following holds:
$$\beta \in \matZ, \quad \delta \in \matZ, \quad \gamma \in \{1,\infty\}, \quad {\rm or} \quad \beta = \delta = \infty.$$
\item if $\alpha = \infty$ one of the following holds:
$$\tfrac 1 \beta \in \matZ \cup \{\infty\}, \quad \gamma \in \matZ\cup \{\infty\}, \quad {\rm or} \quad \tfrac 1 \delta \in \matZ \cup \{\infty\},$$
\end{itemize}
\end{prop}
\begin{proof}
It is shown in \cite[Section 3.5]{MaPeRo} that the isometries of the hyperbolic manifold $S^3 \setminus L_4$ permute the cusps and the slopes $\{0,1,2,\infty\}$ in them. 

Let $\alpha = (\alpha_1,\alpha_2, \alpha_3,\alpha_4)$ be 4 coefficients that yield $\#_h(S^2\times S^1)$. This manifold is not hyperbolic, so the discussion in \cite[Section 3.5]{MaPeRo} shows that, up to isometries of $S^3\setminus L_4$, either $\alpha = (2,2,2,2)$ or $\alpha_i \in \{-1,0,1,2,\infty\}$ for some $i$, say $i=1$. Moreover $(2,2,2,2)$ does not yield $\#_h(S^2 \times S^1)$, so it can be discarded.

It remains to consider the case $\alpha_1 = -1$. It is shown in \cite[Section 3.5]{MaPeRo} that the $(-1,\alpha_2,\alpha_3,\alpha_4)$-Dehn surgery on $L_4$ is diffeomorphic to the
$(\beta_1, \beta_2, \beta_3)$-Dehn surgery on the chain link $L_3\subset S^3$ with 3 components, where
$(\beta_1,\beta_2,\beta_3)= (\alpha_2+1, \alpha_3, \alpha_4+1)$. The tables in \cite[Theorem 1.3]{MaPe} easily show that to get $\#_h(S^2\times S^1)$ we must have $\beta_i \in \{1,2, \infty\}$ for some $i$. (Note that all signs must be reversed in these tables because the mirrored link is considered there.) Therefore we are done.

The last assertions are easy consequences of \cite[Corollary 3.6]{MaPeRo}.
\end{proof}

We now turn back to $W_{11}$ with the usual meridian/longitude basis described in Section \ref{link:surgery:subsection}. A Dehn filling is determined by a 4-tuple $(\alpha, \beta, \gamma, \delta)$ of slopes $\alpha, \beta, \gamma, \delta \in \matQ \cup \{\infty\}$ where $\alpha, \beta$, and $\gamma, \delta$ correspond to boundary components of $X_{11}$ of length 1 and 2 respectively.

\begin{cor} \label{W11:cor}
If a Dehn filling $(\alpha, \beta, \gamma, \delta)$ on $W_{11}$ gives $\#_h(S^2 \times S^1)$ then one of the following 4 conditions holds: 
$$\alpha \in \big\{\infty, -1, -\tfrac 12, 0\big\}, \qquad 
\beta \in \big\{\infty, -1, -\tfrac 12, 0\big\}, \qquad
\gamma \in \big\{\infty, 0, 1, 2 \big\}, \qquad
\delta \in \big\{\infty, 0, 1, 2 \big\}.$$
Moreover:
\begin{itemize}
\item if $\alpha = \infty$ one of the following holds:
$$\gamma \in \matZ, \quad \delta \in \matZ, \quad
\beta \in \{-1,0\}, \quad {\rm or} \quad \gamma = \delta = \infty,$$
\item if $\gamma = \infty$ one of the following holds:
$$
\alpha \in \matZ \cup \{\infty\}, \quad
\beta \in \matZ \cup \{\infty\}, \quad {\rm or} \quad
\delta \in \matZ \cup \{\infty\}.
$$
\end{itemize}
\end{cor}
\begin{proof}
Use SnapPy to figure out the correct change of basis.
\end{proof}

\section{Moves on shadows} \label{moves:section}
Every block $M$ has infinitely many different shadows. As is customary in low-dimensional topology, whenever we have a combinatorial representation of an object (like a knot in $S^3$ or a manifold), there are some local moves that one can use to transform the combinatorial representation without varying the object. 

We introduce here a number of moves that transform a shadow $X$ into another shadow $X'$ of the same block. These will be used in the subsequent sections to prove Theorem \ref{main:teo}. Maybe in contrast with other contexts, we are forced here to consider more than 30 different moves: we consider this to be a manifestation of the intrinsic difficulty one has to manipulate smooth 4-manifolds.

We are particularly interested in the moves that involve the pieces $X_{10}$ and $X_{11}$, for a reason that will be clarified in the next section, related to the fact that $W_{10}$ and $W_{11}$ have many Dehn fillings that yield $\#_h(S^2 \times S^1)$, as discovered in Section \ref{exceptional:section}.

\subsection{Basic moves}
A \emph{move} is an operation that modifies only a portion of a shadow $X$ leaving the rest unaltered, thus transforming it into another shadow $X'$. A move may modify the topological structure of $X$ and/or its gleams.

As proved by Turaev \cite{Tu}, the moves in Figure \ref{move_all:fig} transform a shadow $X$ of a framed block $M$ into another shadow $X'$ of the same block $M$. 
We will use these 5 ``basic moves'' to construct more complicated moves in the next pages. As a start, the construction of the slightly more elaborated moves shown in Figure \ref{more_moves:fig} is left as an exercise.

\begin{figure}
\begin{center}
\includegraphics[width = 12 cm]{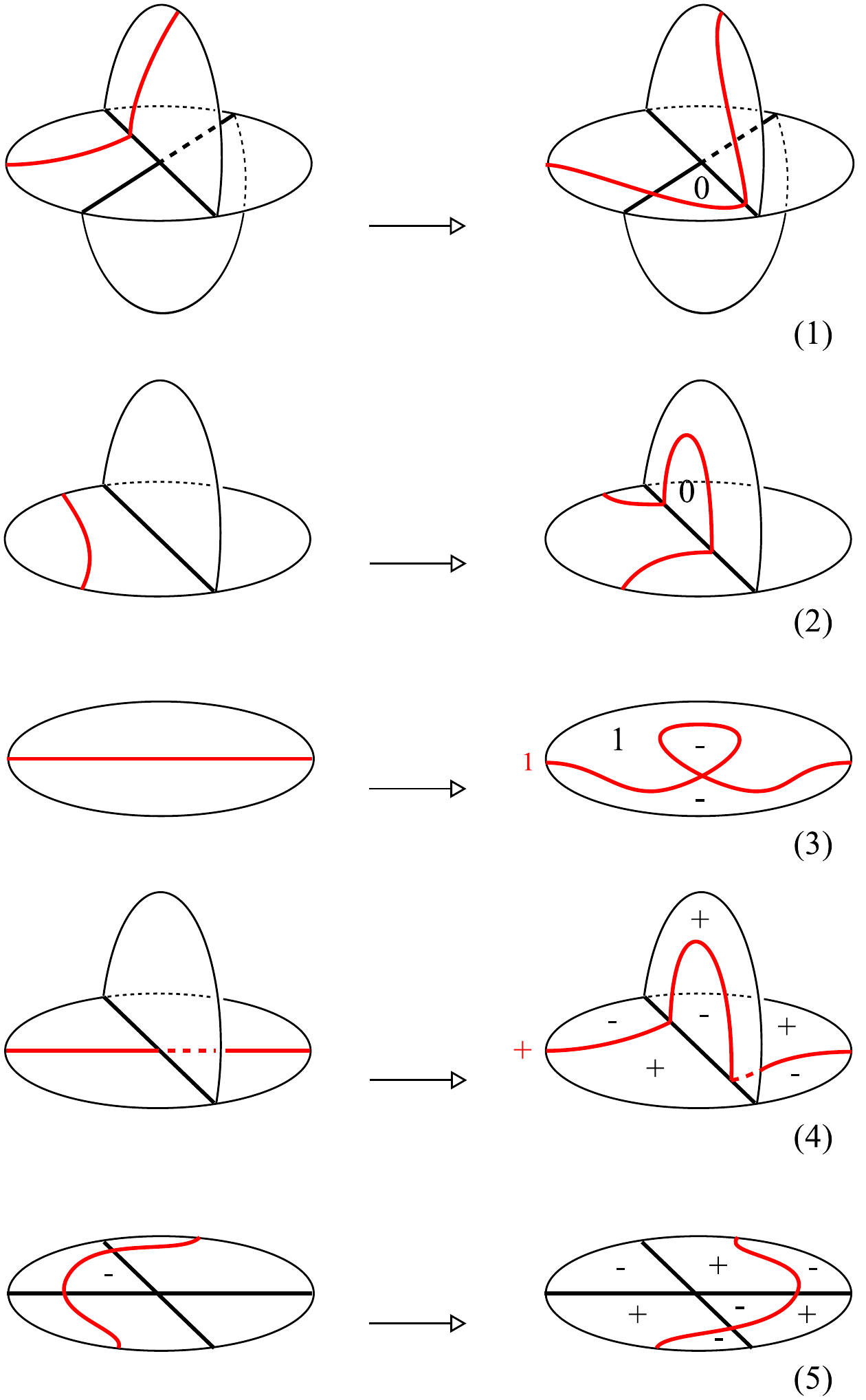}
\nota{Five ``basic moves'' that transform a shadow $X$ into another shadow $X'$ of the same framed block. A disc is attached along each red arc. Moves (1) and (2) can be embedded in a 3-dimensional slice, while the moves (3) and (4) cannot. In moves (3) and (4), the gleam of the red region is modified after the move respectively by adding $1$ and $\frac 12$ (the number is pictured in red).  In (3), (4), (5) we can also apply the same move with all signs reversed.
Move (5) is obtained by composing multiple times the moves (1), (2), (4) and their inverses.
}
\label{move_all:fig}
\end{center}
\end{figure}

\begin{figure}
\begin{center}
\includegraphics[width = 16 cm]{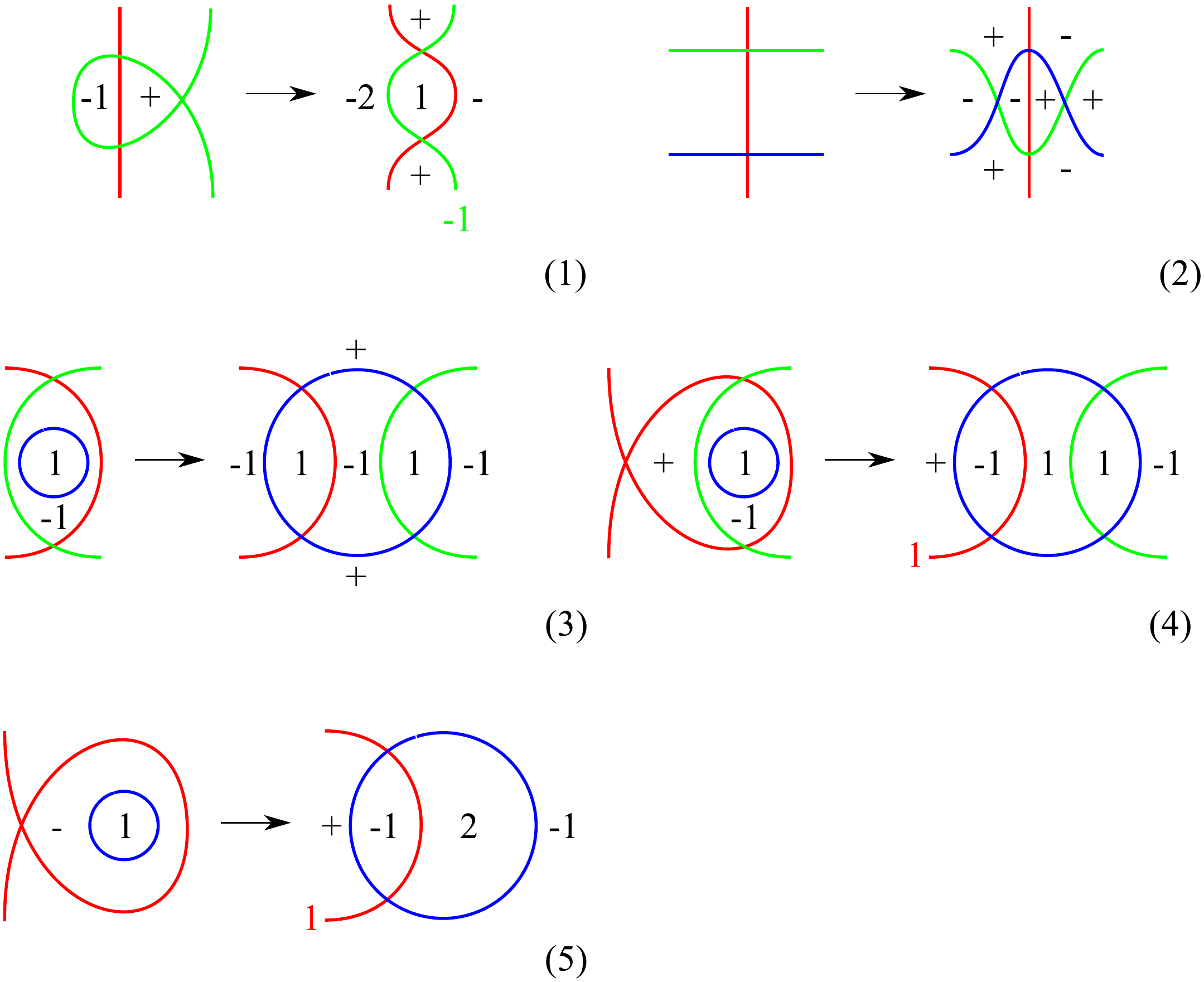}
\nota{These moves transform a shadow $X$ into another shadow $X'$ of the same block. They can be easily obtained by composing a few of the basic moves of Figure \ref{move_all:fig} and their inverses. A region is attached to each coloured arc. Where a coloured number appears, that is the green $-1$ in (1) and the red 1 in (4) and (5), the region attached to the coloured arc changes its gleam by adding this number. The moves also hold with all signs reversed.
}
\label{more_moves:fig}
\end{center}
\end{figure}

\subsection{Collapsing regions}
Four-manifolds are intrinsically more complicated than 3-manifolds, so it is not surprising that we are forced to discover many different kinds of moves to prove our main theorem, and quite frustratingly different moves often require different proofs. We try in our exposition to select whenever possible a few number of moves that somehow ``generate'' all the others.

A simple way to generate more moves from a given one is by collapsing some regions. Whenever a move transforms a portion $X_*$ into another portion $X_*'$, more moves can be found by collapsing (that is, removing) a disc region of $X_*$, when possible. For instance, in Figure \ref{move_all:fig} the move (3) is generated by (4) after collapsing the bottom-right region. In Figure \ref{more_moves:fig}, move (5) is obtained from (4). We will employ this technique quite often.

\subsection{Moves without vertices}
A table of moves that involve portions of shadow without vertices is shown in Figure \ref{thickening:fig} using the decorated graph notation. Each of the moves shown there transforms a shadow $X$ into another shadow $X'$ of the same block $M$. These were proved in \cite{Ma:zero} and used there as important tools to classify the manifolds with complexity zero.

\begin{figure}
\begin{center}
\includegraphics[width = 15 cm]{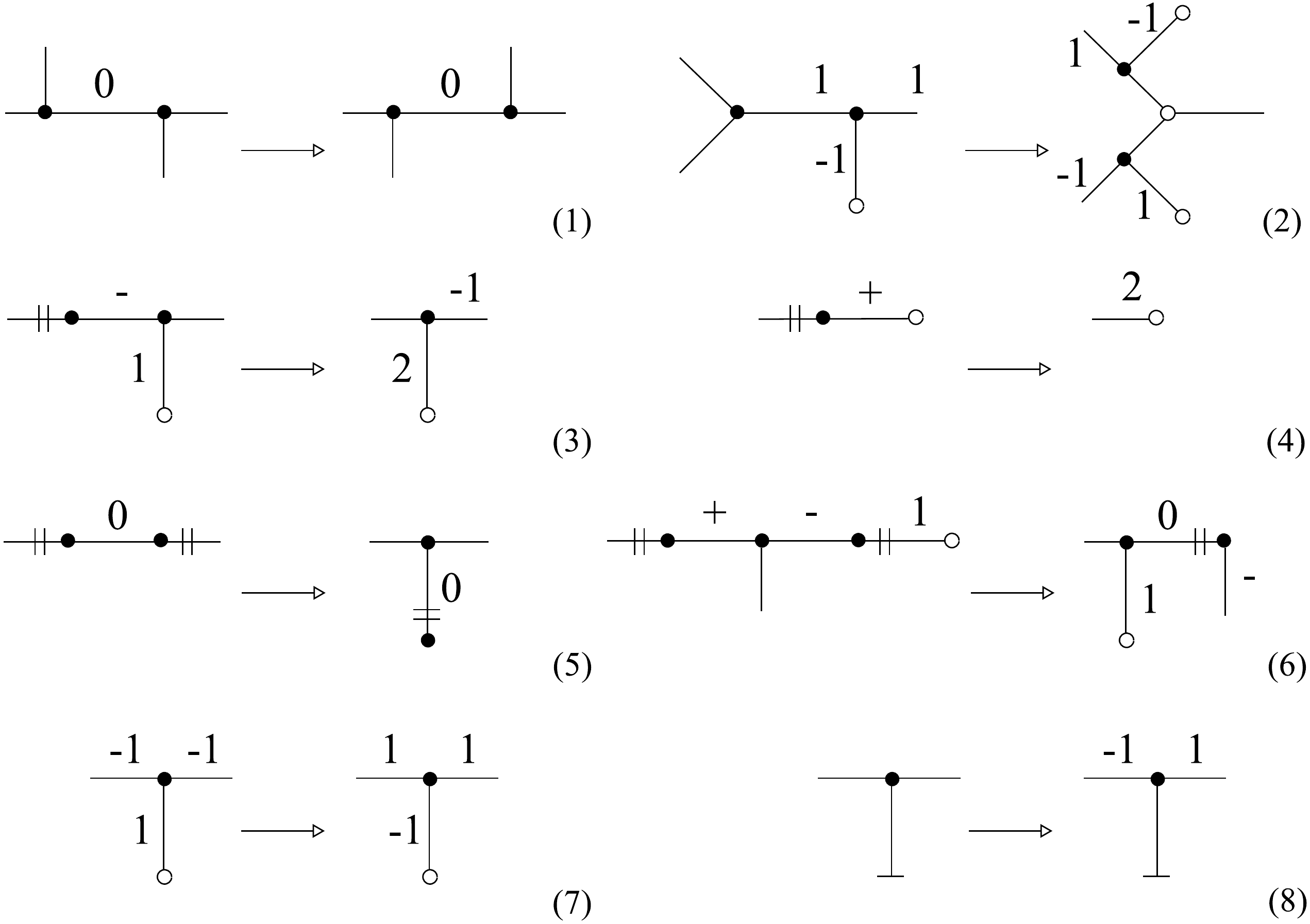}
\nota{Some moves on portions without vertices. They transform a shadow $X$ into another shadow $X'$ of the same block. Similar moves hold if we reverse all the signs.} 
\label{thickening:fig}
\end{center}
\end{figure}

Note that (3) generates (4) by collapsing the right edge.

\subsection{New moves with few vertices}
We now add more moves that involve portions with a small number of vertices. In the first two moves in Figures \ref{moves1:fig} and \ref{m1piu:fig} we have two adjacent bigons with some particular gleams. In Figures \ref{moves2:fig} and \ref{moves3:fig}
the dashed opposite sides should be identified via a translation, so that squares and rectangles represent annuli. 

\begin{figure}
\begin{center}
\includegraphics[width = 11 cm]{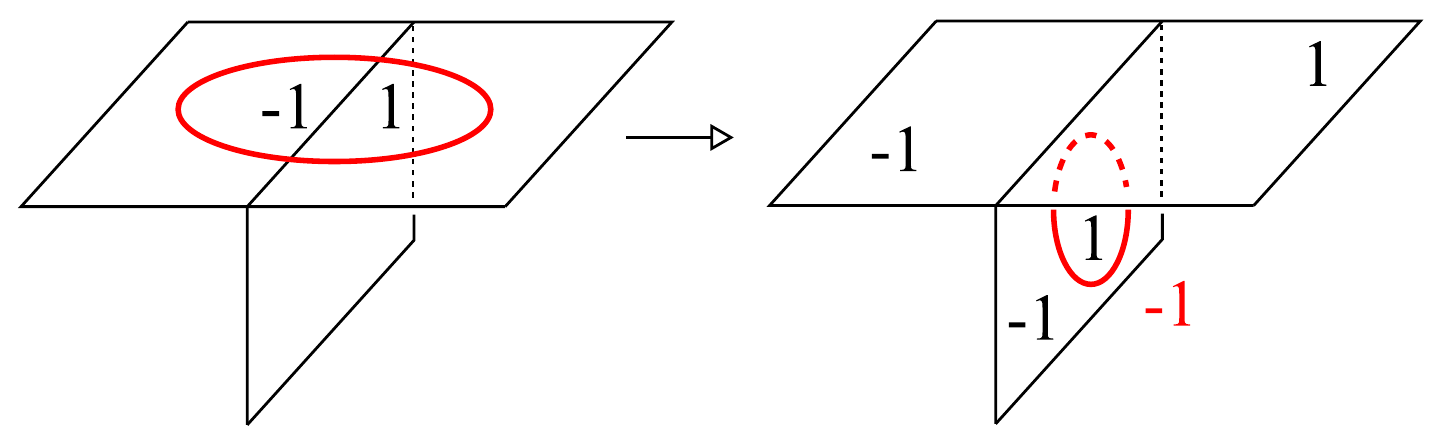}
\nota{A move that transforms a shadow $X$ into another shadow $X'$ of the same block. A region $f$ is attached along the red curve, and the gleam of $f$ changes by adding $-1$ after the move, as suggested by the red $-1$ in the right figure. A similar move holds by reversing all signs.}
\label{moves1:fig}
\end{center}
\end{figure}

\begin{figure}
\begin{center}
\includegraphics[width = 9 cm]{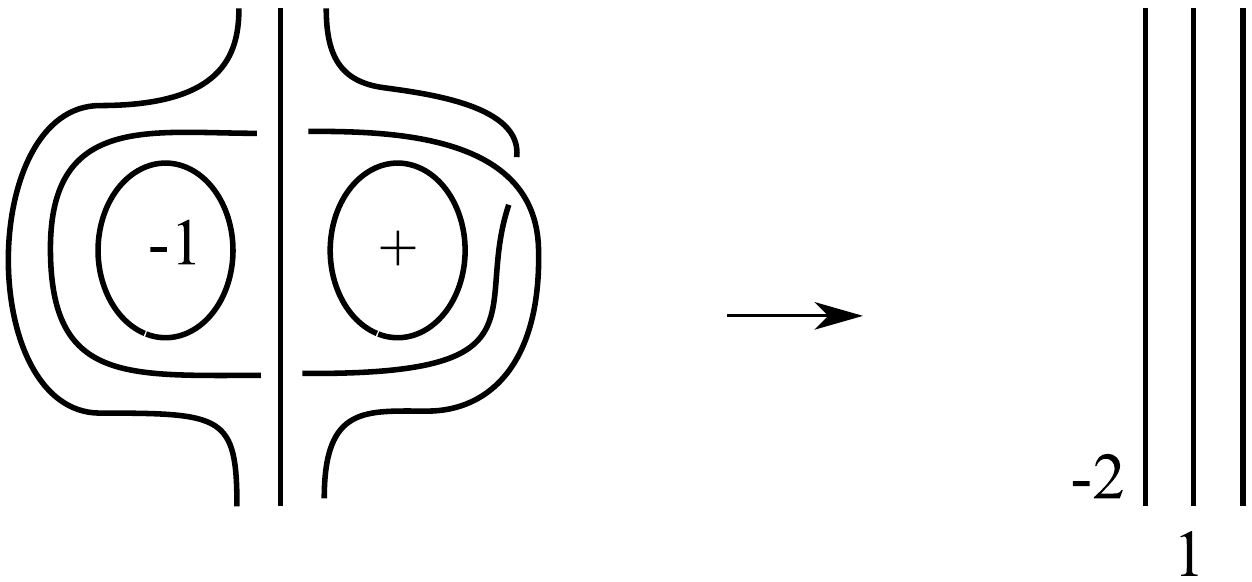}
\nota{A move that transforms a shadow $X$ into another shadow $X'$ of the same block. The picture in the left shows two adjacent bigons with gleams $-1$ and $+\frac 12$, that are eliminated in the move. In $X'$ the gleams of the 3 regions change as shown (by adding $-2$, $1$, or zero). A similar move holds by reversing all signs.}
\label{m1piu:fig}
\end{center}
\end{figure}

\begin{figure}
\begin{center}
\includegraphics[width = 13 cm]{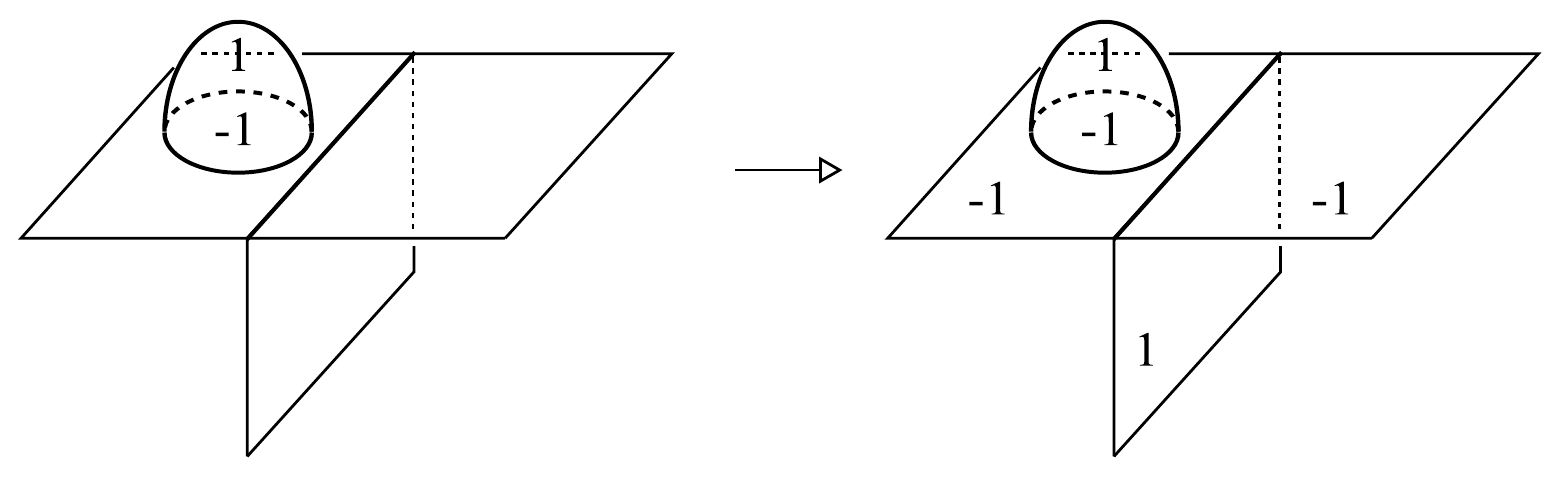}
\nota{A move that transforms a shadow $X$ into another shadow $X'$ of the same block. A similar move holds by reversing all signs.}
\label{moves4:fig}
\end{center}
\end{figure}

\begin{figure}
\begin{center}
\includegraphics[width = 16 cm]{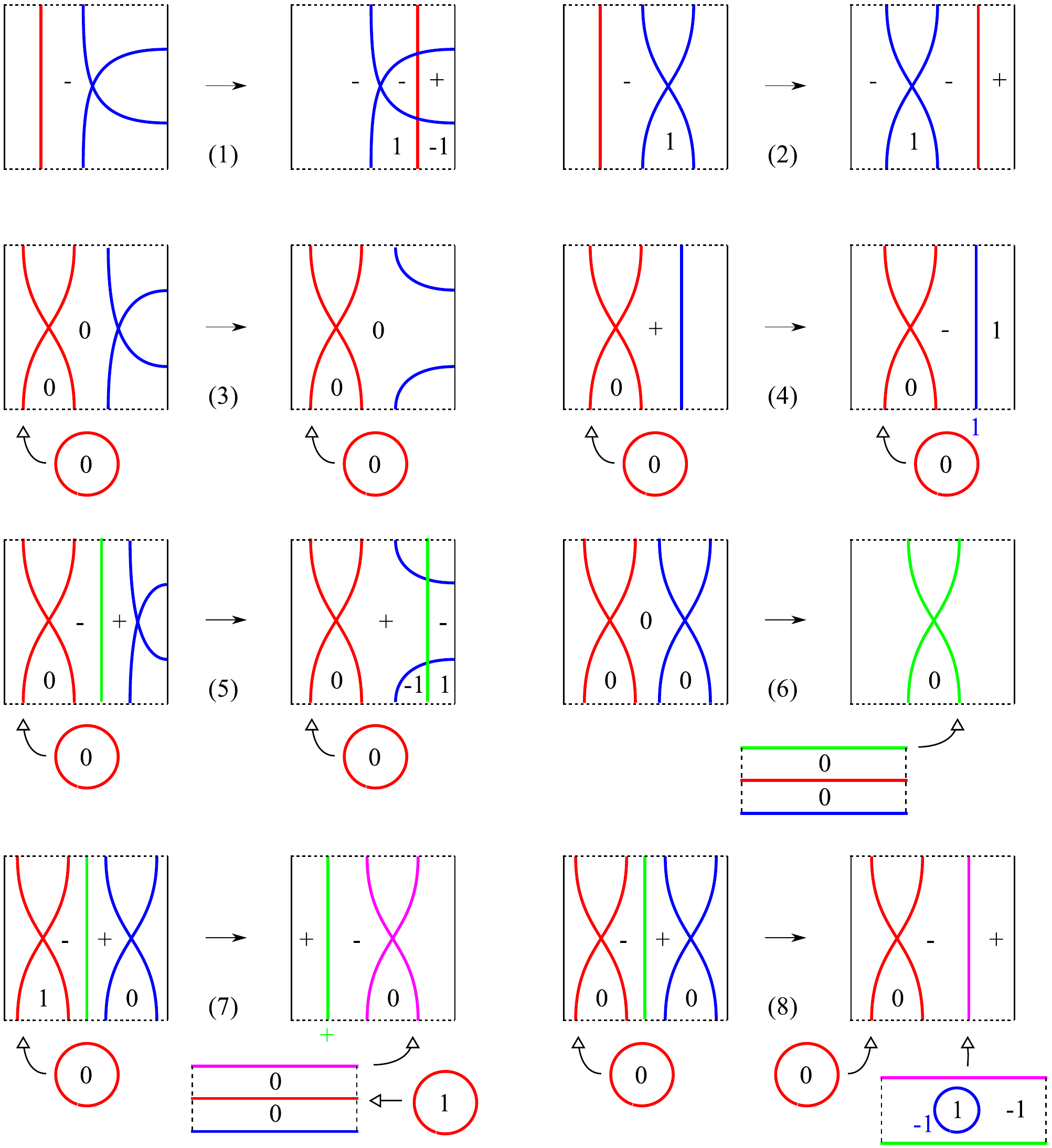}
\nota{Some moves that transform a shadow $X$ into another shadow $X'$ of the same block. In all the pictures the dashed opposite sides should be identified via a translation, so that squares and rectangles represent annuli. In (3), \ldots, (8) the arrow indicates that a disc or a more complicated portion should be attached as indicated. The blue number $\pm 1$ in (4) and (8) must be added to the gleam of the region attached to the blue curve. Similar moves hold by reversing all signs.}
\label{moves2:fig}
\end{center}
\end{figure}

\begin{figure}
\begin{center}
\includegraphics[width = 16 cm]{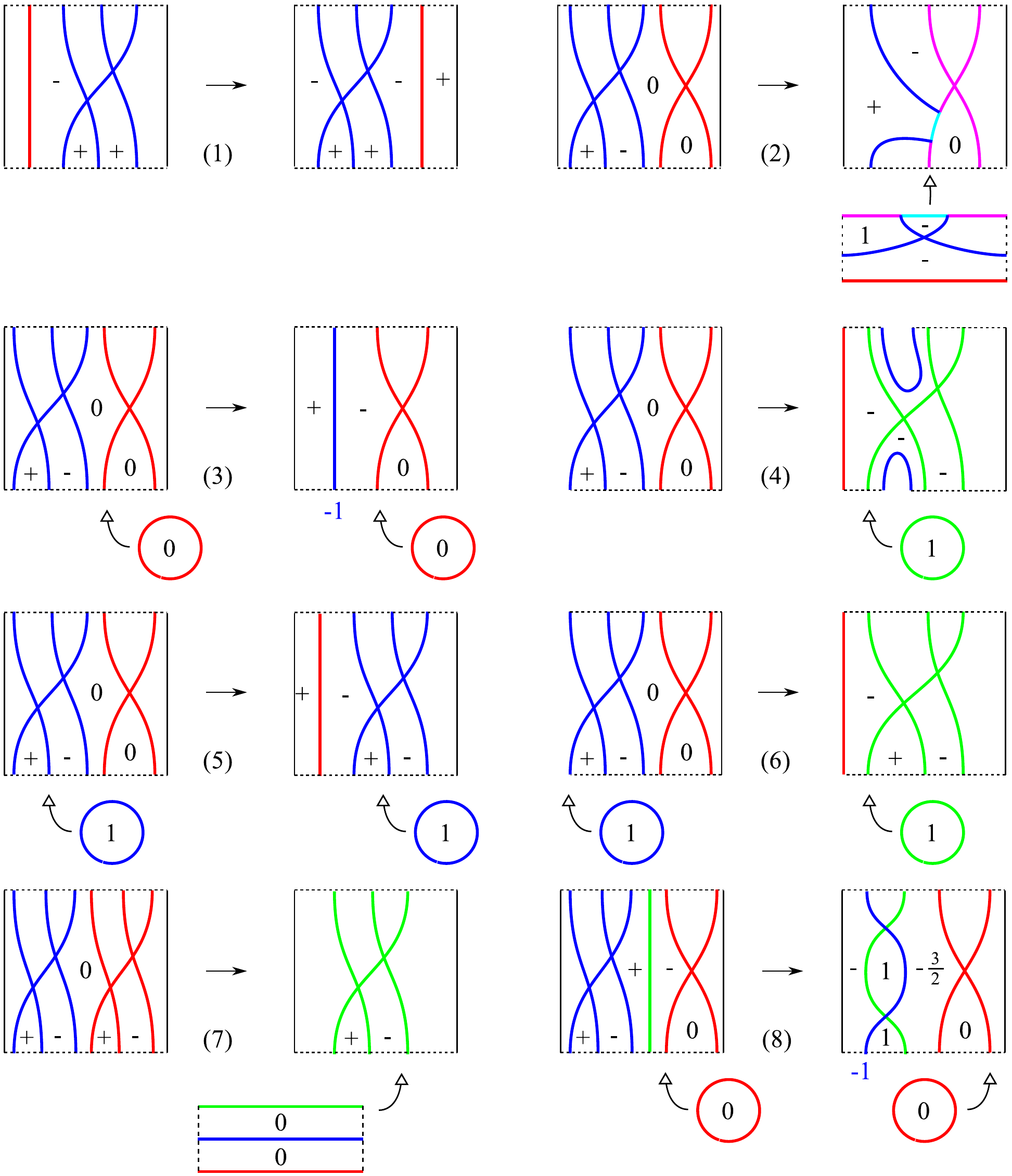}
\nota{Some moves that transform a shadow $X$ into another shadow $X'$ of the same block. In all the pictures the dashed opposite sides should be identified via a translation, so that squares and rectangles represent annuli. The blue number $-1$ in (3) and (8) must be added to the region attached to the blue curve.
Similar moves hold by reversing all signs.
}
\label{moves3:fig}
\end{center}
\end{figure}

\begin{prop}
The moves in Figures \ref{moves1:fig}, \ref{m1piu:fig}, \ref{moves4:fig}, \ref{moves2:fig}, and \ref{moves3:fig} modify a shadow $X$ into another shadow $X'$ of the same block.
\end{prop}
\begin{proof}
The move in Figure \ref{moves1:fig} is proved in Figure \ref{bubble_slide:fig}. 
The move in Figure \ref{m1piu:fig} is obtained by performing the opposite of Figure \ref{move_all:fig}-(4) and (3). The move in Figure \ref{moves4:fig} is obtained from Figure \ref{moves1:fig} and the inverse of Figure \ref{move_all:fig}-(2).

We now turn to Figure \ref{moves2:fig}. Move (1) is obtained from Figure \ref{move_all:fig}-(2, 5). Move (2) is obtained from (1) using the inverse of Figure \ref{move_all:fig}-(2).
Move (3) is obtained using the moves in Figure \ref{move_all:fig}-(1,2) multiple times to slide the blue curve over the red disc as in Figure \ref{moves2_proof:fig}. Move (4) is proved in Figure \ref{moves2_proof2:fig} using (3). Move (5) is obtained by composing (1) (with reversed signs) and (3). Move (6) is obtained by sliding completely the blue curve above the red, similarly as in (3). Moves (7) and (8) are proved in Figures \ref{moves2_proof3:fig} and \ref{moves2_proof4:fig}.

We now consider Figure \ref{moves3:fig}. Moves (1) and  (2) are proved in Figures \ref{moves3_proof:fig} and \ref{moves3_proof2:fig}. Moves (3) and (4) follow from (2) respectively by adding a disc and by collapsing a region. Move (5) is proved in Figure \ref{moves3_proof3:fig}. Move (6) is then obtained by collapsing the left region. Move (7) is obtained by sliding entirely the red curve onto the blue region and is left as an exercise. Move (8) is proved with Figure \ref{moves2:fig}-(5) plus Figure \ref{more_moves:fig}-(1).
\end{proof}

\begin{figure}
\begin{center}
\includegraphics[width = 9 cm]{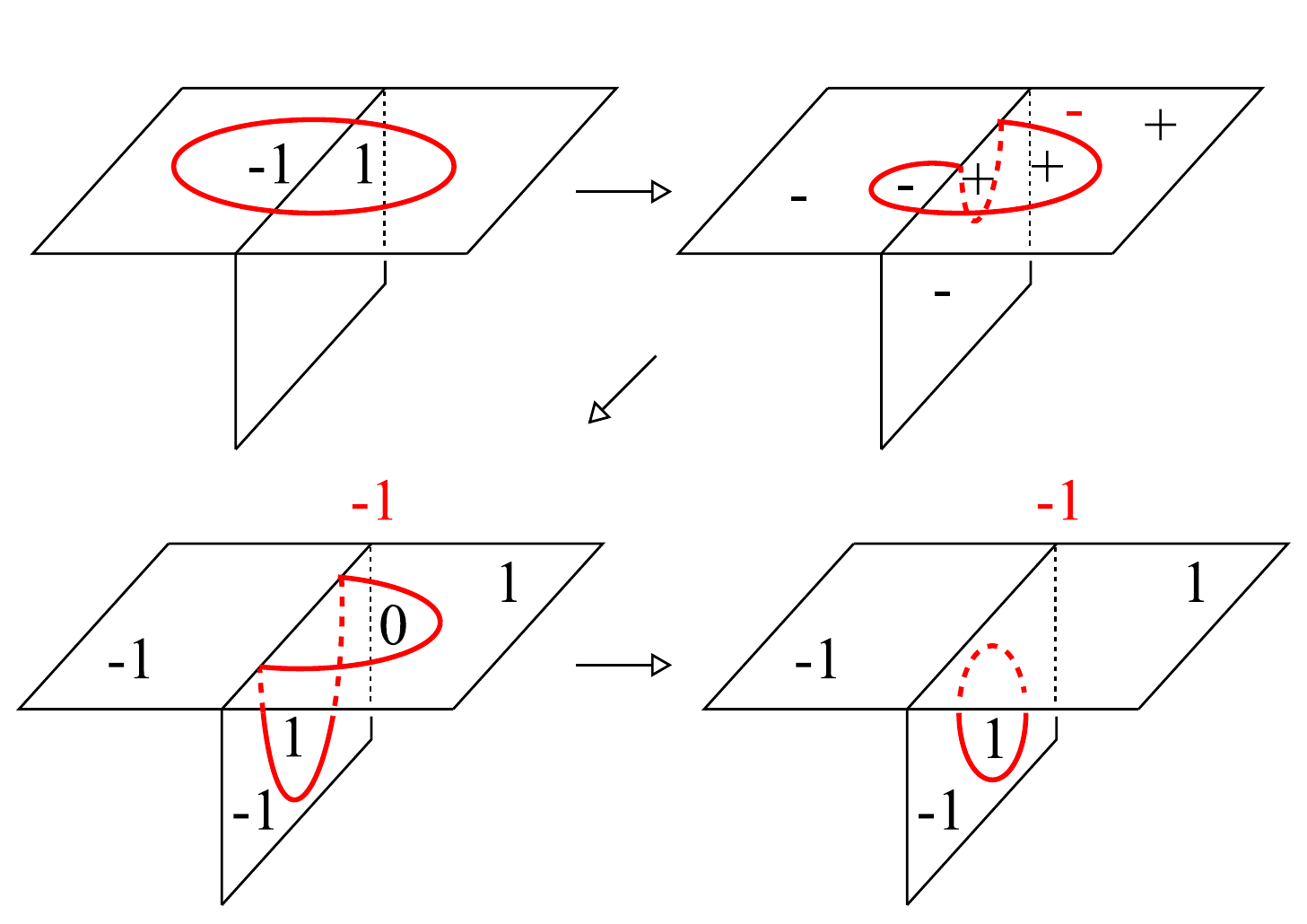}
\nota{We apply Figure \ref{move_all:fig}-(4) (with reversed signs), its inverse, and the inverse of Figure \ref{move_all:fig}-(2). The red gleams $-\frac 12$ and $-1$ are assigned to the region attached to the red closed curve.}
\label{bubble_slide:fig}
\end{center}
\end{figure}

\begin{figure}
\begin{center}
\includegraphics[width = 15 cm]{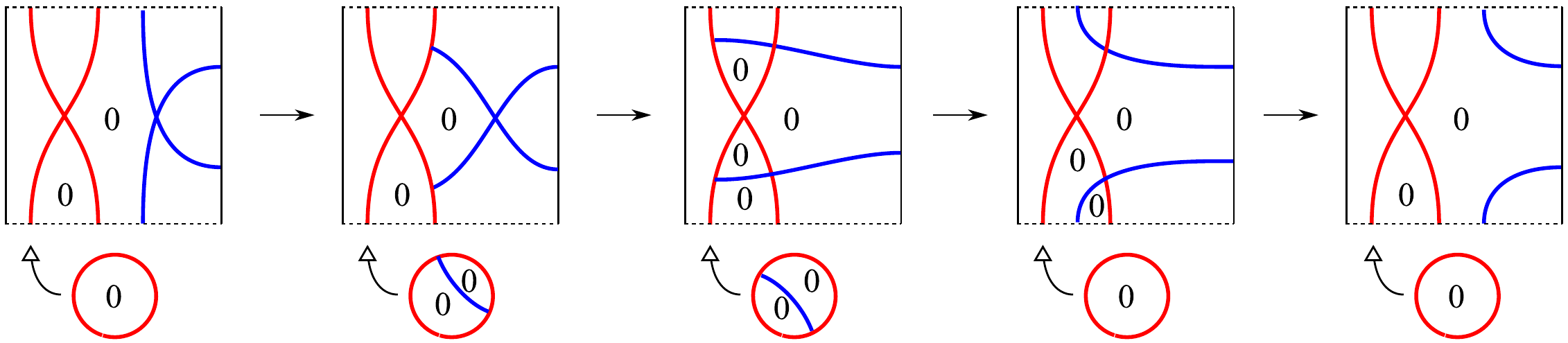}
\nota{A proof of Figure \ref{moves2:fig}-(3).}
\label{moves2_proof:fig}
\end{center}
\end{figure}

\begin{figure}
\begin{center}
\includegraphics[width = 9 cm]{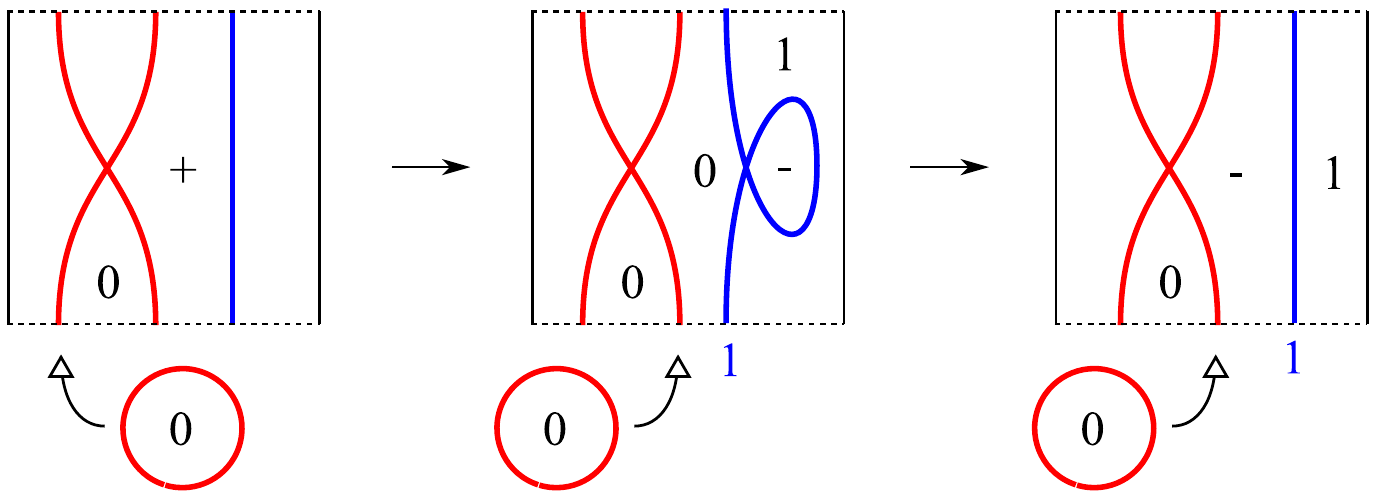}
\nota{A proof of Figure \ref{moves2:fig}-(4).}
\label{moves2_proof2:fig}
\end{center}
\end{figure}

\begin{figure}
\begin{center}
\includegraphics[width = 12 cm]{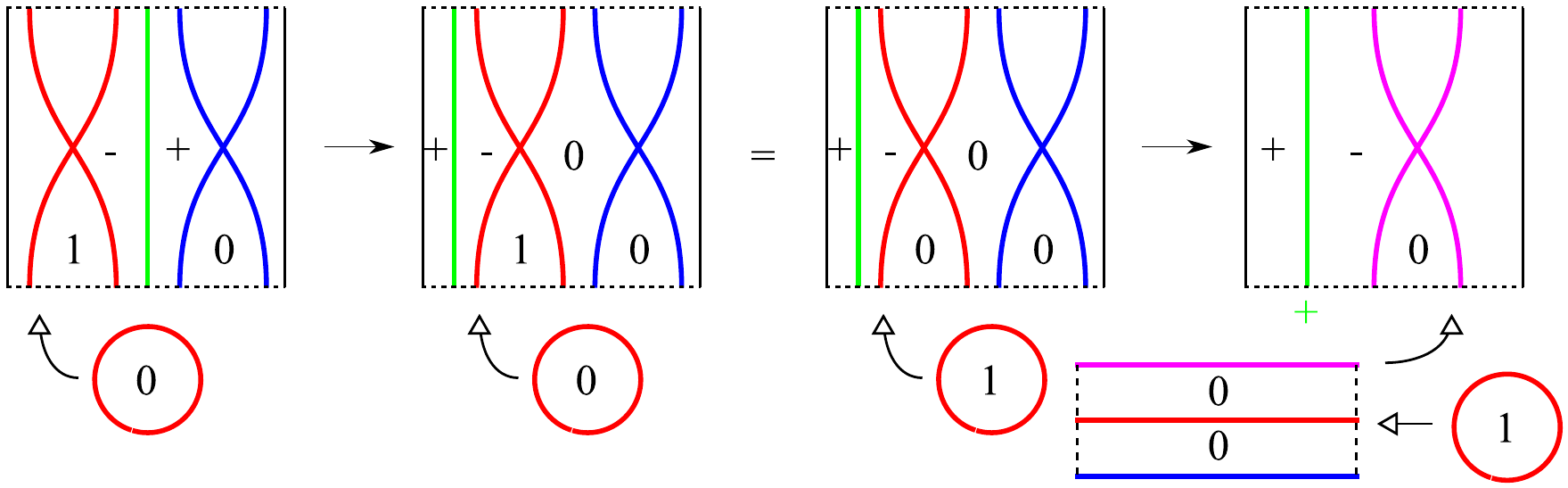}
\nota{A proof of Figure \ref{moves2:fig}-(7).}
\label{moves2_proof3:fig}
\end{center}
\end{figure}

\begin{figure}
\begin{center}
\includegraphics[width = 16 cm]{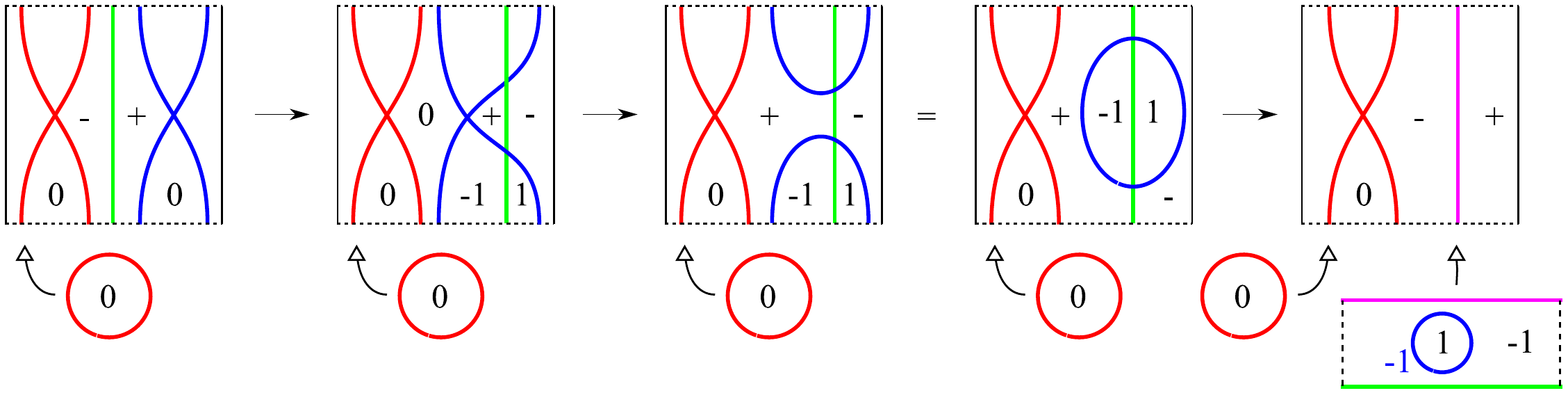}
\nota{A proof of Figure \ref{moves2:fig}-(8).}
\label{moves2_proof4:fig}
\end{center}
\end{figure}

\begin{figure}
\begin{center}
\includegraphics[width = 12.5 cm]{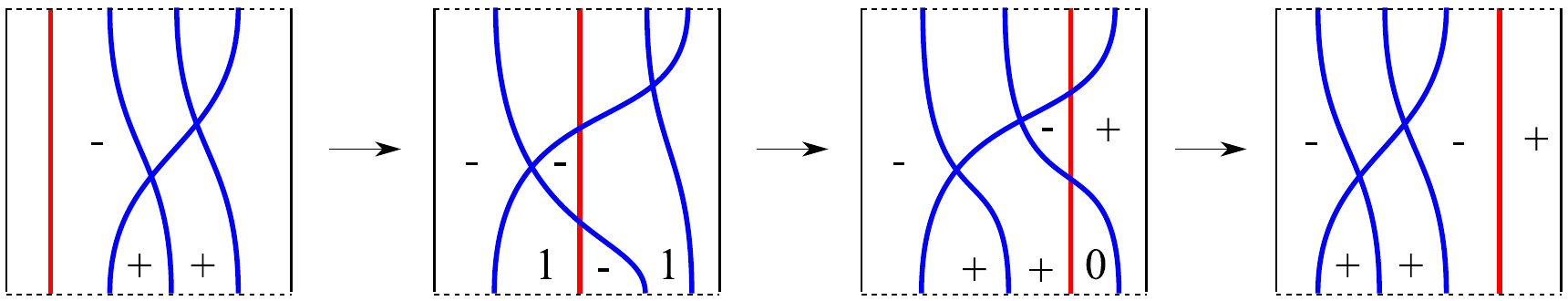}
\nota{A proof of Figure \ref{moves3:fig}-(1).}
\label{moves3_proof:fig}
\end{center}
\end{figure}

\begin{figure}
\begin{center}
\includegraphics[width = 16 cm]{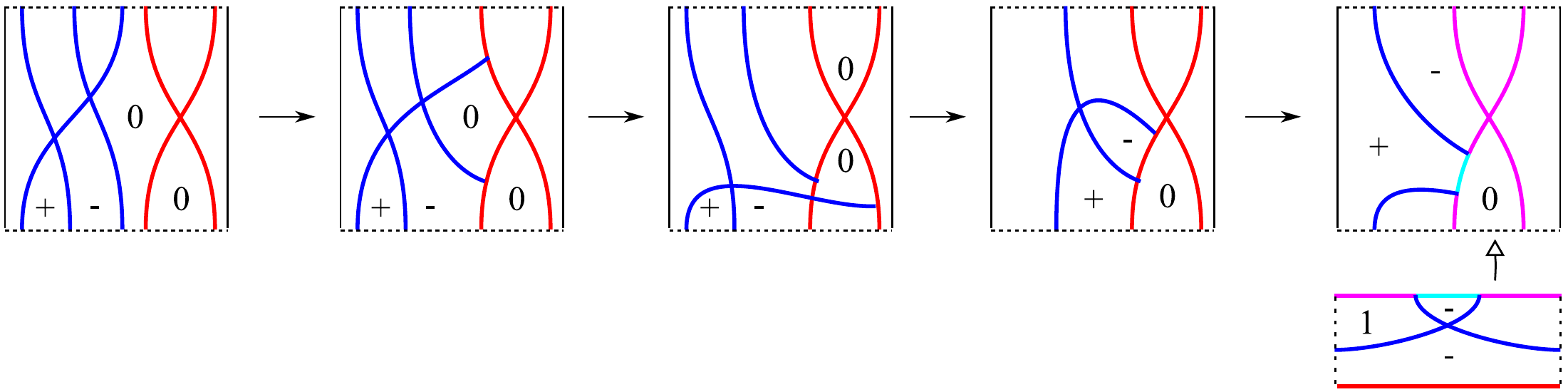}
\nota{A proof of Figure \ref{moves3:fig}-(2).}
\label{moves3_proof2:fig}
\end{center}
\end{figure}

\begin{figure}
\begin{center}
\includegraphics[width = 16 cm]{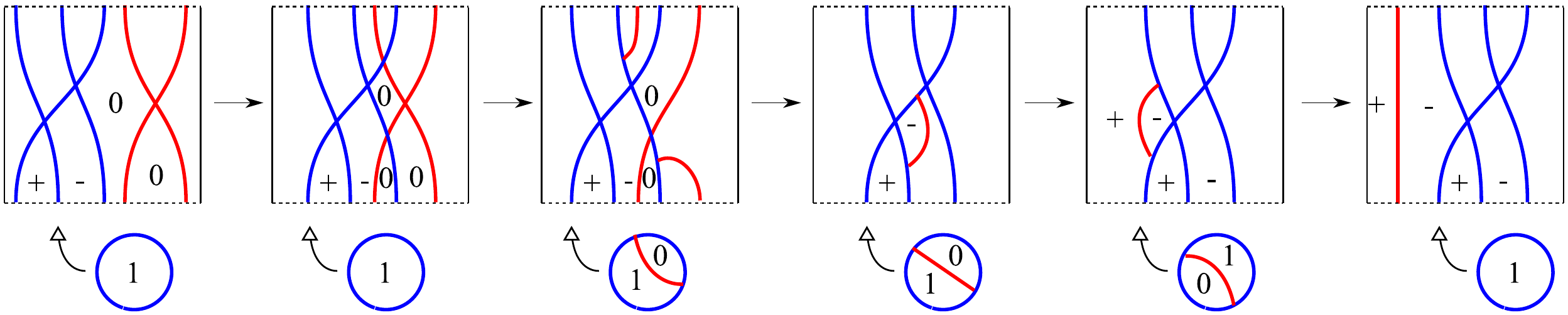}
\nota{A proof of Figure \ref{moves3:fig}-(5).}
\label{moves3_proof3:fig}
\end{center}
\end{figure}

\subsection{Moves that involve $X_{11}$}
We now use the many moves of the previous section to build a table of moves that involve $X_{11}$ and are described using the decorated graph language. We will use these moves to prove Theorem \ref{main:teo}.

\begin{figure}
\begin{center}
\includegraphics[width = 15 cm]{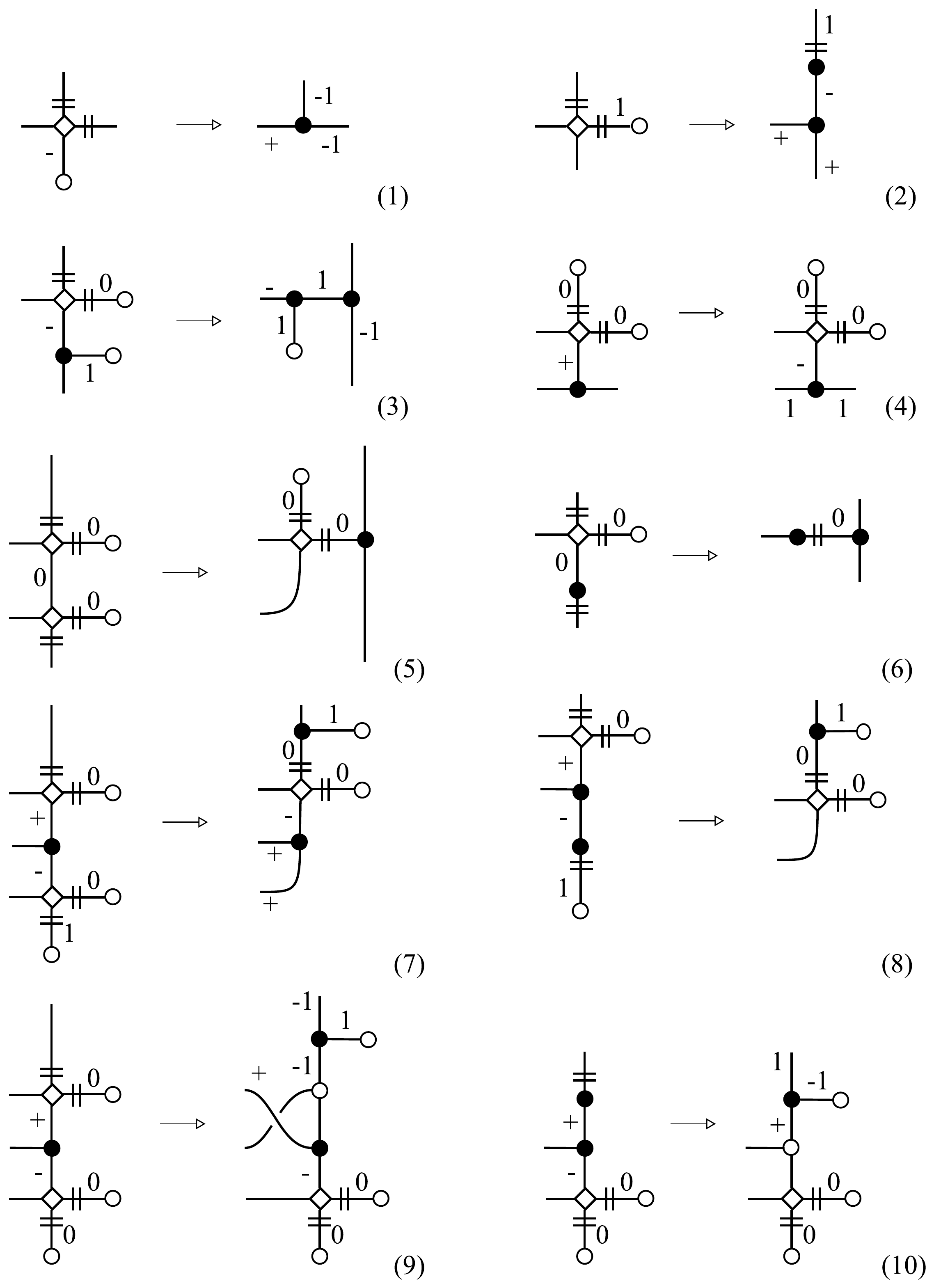}
\nota{These moves transform a shadow $X$ into another shadow $X'$ of the same block. Analogous moves hold with all signs reversed.}
\label{X11moves:fig}
\end{center}
\end{figure}

\begin{prop}
The moves in Figure \ref{X11moves:fig} modify a shadow $X$ into another shadow $X'$ of the same block.
\end{prop}
\begin{proof}
Move (1) is proved using Figure \ref{move_all:fig}-(3), and (2) follows from Figure \ref{moves2:fig}-(2) by collapsing the left region. Move (3) is proved in Figure \ref{X11coso:fig}, where we use Figures \ref{more_moves:fig}-(5) and \ref{moves1:fig}. 
Move (4) is Figure \ref{moves2:fig}-(4).

Move (5) is Figure \ref{moves2:fig}-(6) and move (6) is obtained by collapsing one region. Move (7) is Figure \ref{moves2:fig}-(7) and (8) is obtained by collapsing. Move (9) is Figure \ref{moves2:fig}-(8) and (10) is again obtained by collapsing (and we use Figure \ref{thickening:fig}-(7)).
\end{proof}

\begin{figure}
\begin{center}
\includegraphics[width = 16 cm]{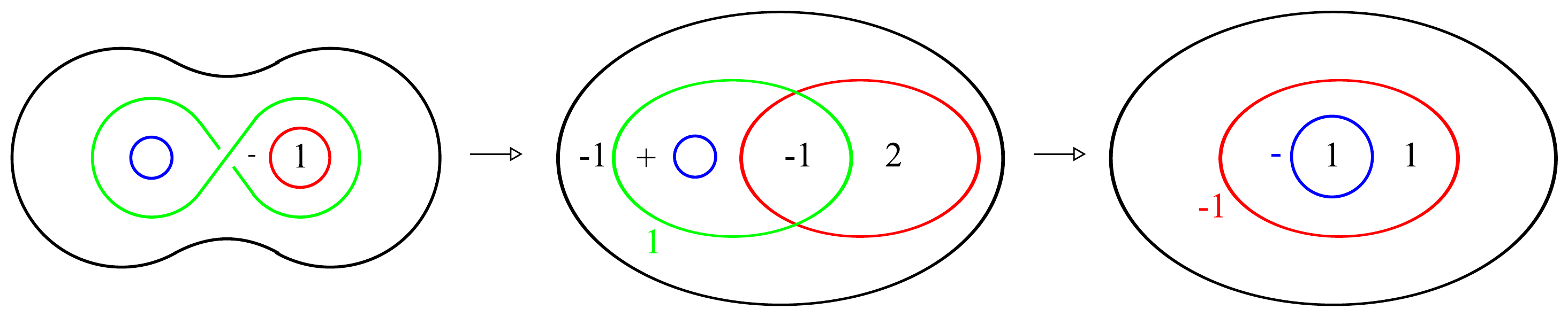}
\nota{Proof of Figure \ref{X11moves:fig}-(3). There is a disc with gleam zero attached to the green curve; the gleam becomes 1 after the first move, as indicated. In the final position the gleam of the region attached to the blue curve changes by $-\frac 12$, as indicated.}
\label{X11coso:fig}
\end{center}
\end{figure}

\subsection{Moves that involve $X_{10}$}
We now build a table of moves that involve $X_{10}$. The moves are described using the decorated graph language. We will use these moves to prove Theorem \ref{main:teo}.

\begin{figure}
\begin{center}
\includegraphics[width = 15 cm]{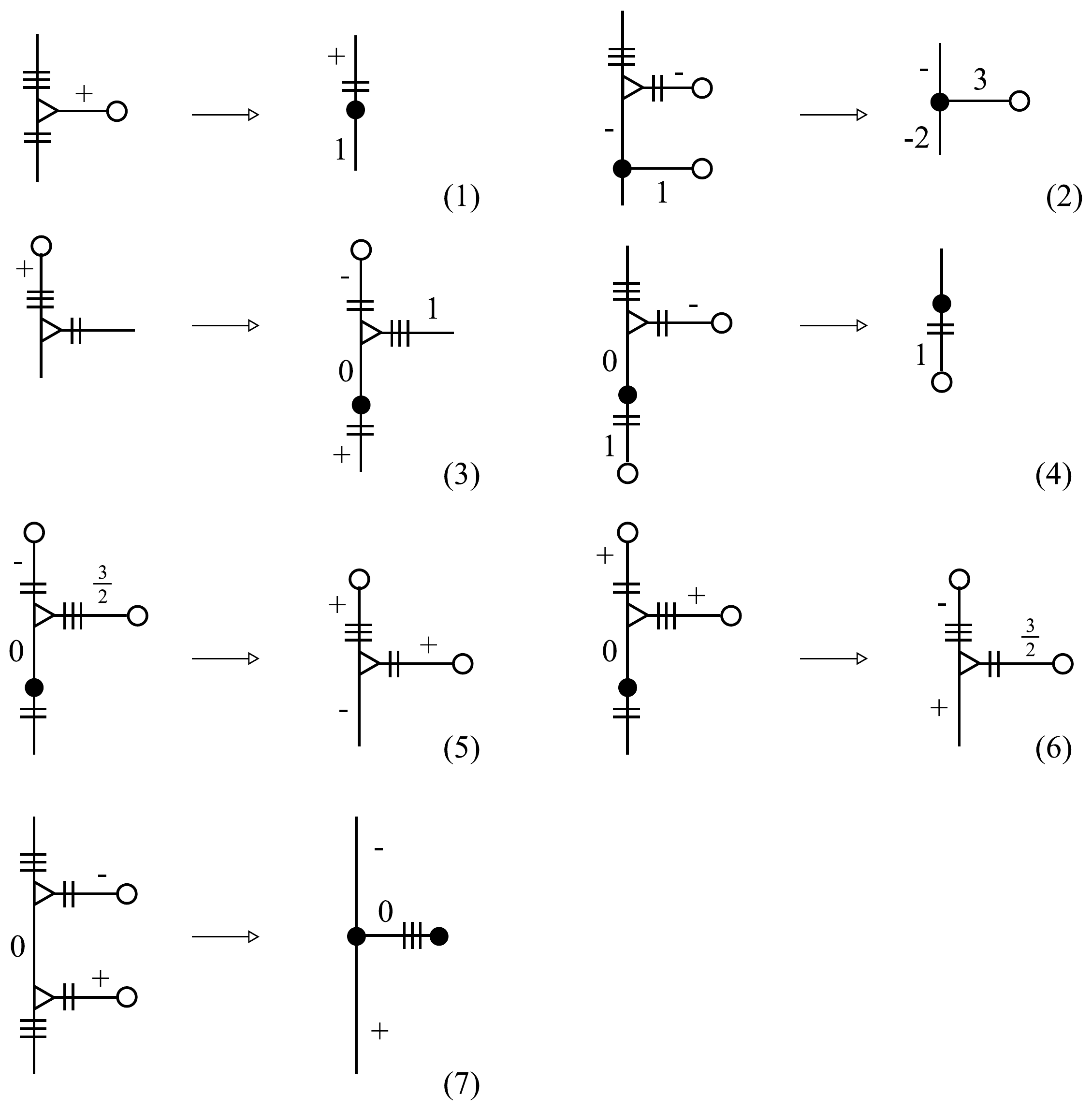}
\nota{These moves transform a shadow $X$ into another shadow $X'$ of the same block. Analogous moves hold with all signs reversed.}
\label{X10moves:fig}
\end{center}
\end{figure}

\begin{prop}
The moves in Figure \ref{X10moves:fig} modify a shadow $X$ into another shadow $X'$ of the same block.
\end{prop}
\begin{proof}
Move (1) is obtained by applying Figure \ref{move_all:fig}-(3). Move (2) is proved in Figure \ref{X10coso:fig}, where we use Figures \ref{more_moves:fig}-(5) and \ref{m1piu:fig}. Move (3) is obtained from the inverse of Figure \ref{moves3:fig}-(4) by collapsing the right region. Move (4) is the inverse of (3) plus (1). Moves (5) and (6) are consequences of the inverse of (3).
Move (7) is Figure \ref{moves3:fig}-(7) after collapsing the left and right regions.
\end{proof}

\begin{figure}
\begin{center}
\includegraphics[width = 16 cm]{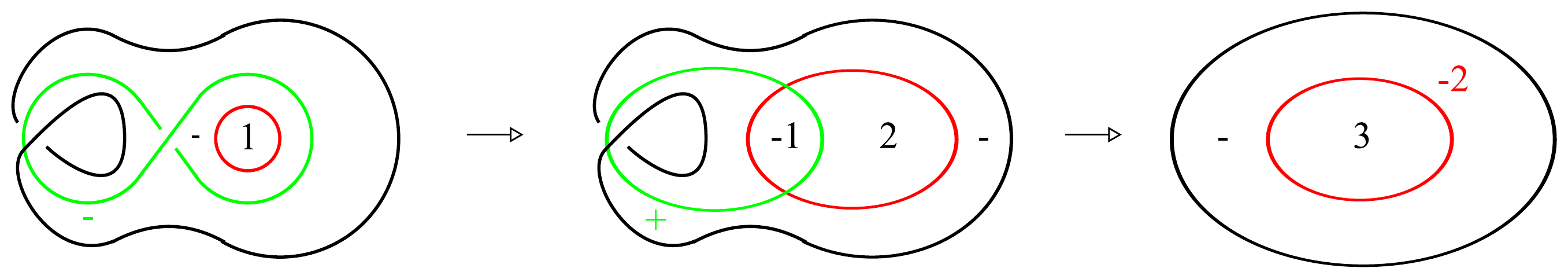}
\nota{Proof of Figure \ref{X10moves:fig}-(2). There is a disc with gleam $-\frac 12$ attached to the green curve, as indicated; the gleam becomes $\frac 12$ after the first move. In the final position the gleam of the region attached to the red curve changes by $-2$.}
\label{X10coso:fig}
\end{center}
\end{figure}

\subsection{Moves that involve $X_{10}$ and $X_{11}$}
We now build a table of moves that involve both $X_{10}$ and $X_{11}$. The moves are described using the decorated graph language. We will use these moves to prove Theorem \ref{main:teo}.

\begin{prop}
The moves in Figures \ref{X10e11moves:fig} and \ref{X10e11moves2:fig} modify a shadow $X$ into another shadow $X'$ of the same block.
\end{prop}
\begin{proof}
Move (1) is Figure \ref{moves3:fig}-(8) with the left region collapsed. Move (2) follows from Figure \ref{moves2:fig}-(3). Move (3) is proved in Figure \ref{oioia:fig}, followed by Figure \ref{X11moves:fig}-(2). Move (4) follows from (3). To get (5) we 
first apply Figure \ref{moves3:fig}-(4), then Figure \ref{moves2:fig}-(3) and finally Figure \ref{X11moves:fig}-(2). 

To get (6), we apply Figure \ref{moves3:fig}-(6), the inverse of Figure \ref{moves3:fig}-(5), the inverse of Figure \ref{moves3:fig}-(6), and finally (5). Move (7) is proved in Figure \ref{oimmena:fig}; in that picture, we start like in Figure \ref{moves2_proof:fig}, then we apply Figures \ref{more_moves:fig}-(4,1) and \ref{moves1:fig}. 

The proof of move (8) is more elaborated: we first use Figure \ref{moves3:fig}-(6) to transform the portion as in Figure \ref{diamine2:fig}-(left) with the left region collapsed. Then we conclude as shown there, using Figures \ref{moves2:fig}-(2) and 
\ref{diamine:fig}. In Figure \ref{diamine:fig} we use Figure \ref{moves3:fig}-(2) and conclude via basic moves.

Finally, Figure \ref{X10e11moves2:fig} is proved by combining Figures \ref{moves3:fig}-(6, 7, 1) and Figure \ref{moves4:fig}. Details are left as an exercise.
\end{proof}

\begin{figure}
\begin{center}
\includegraphics[width = 15 cm]{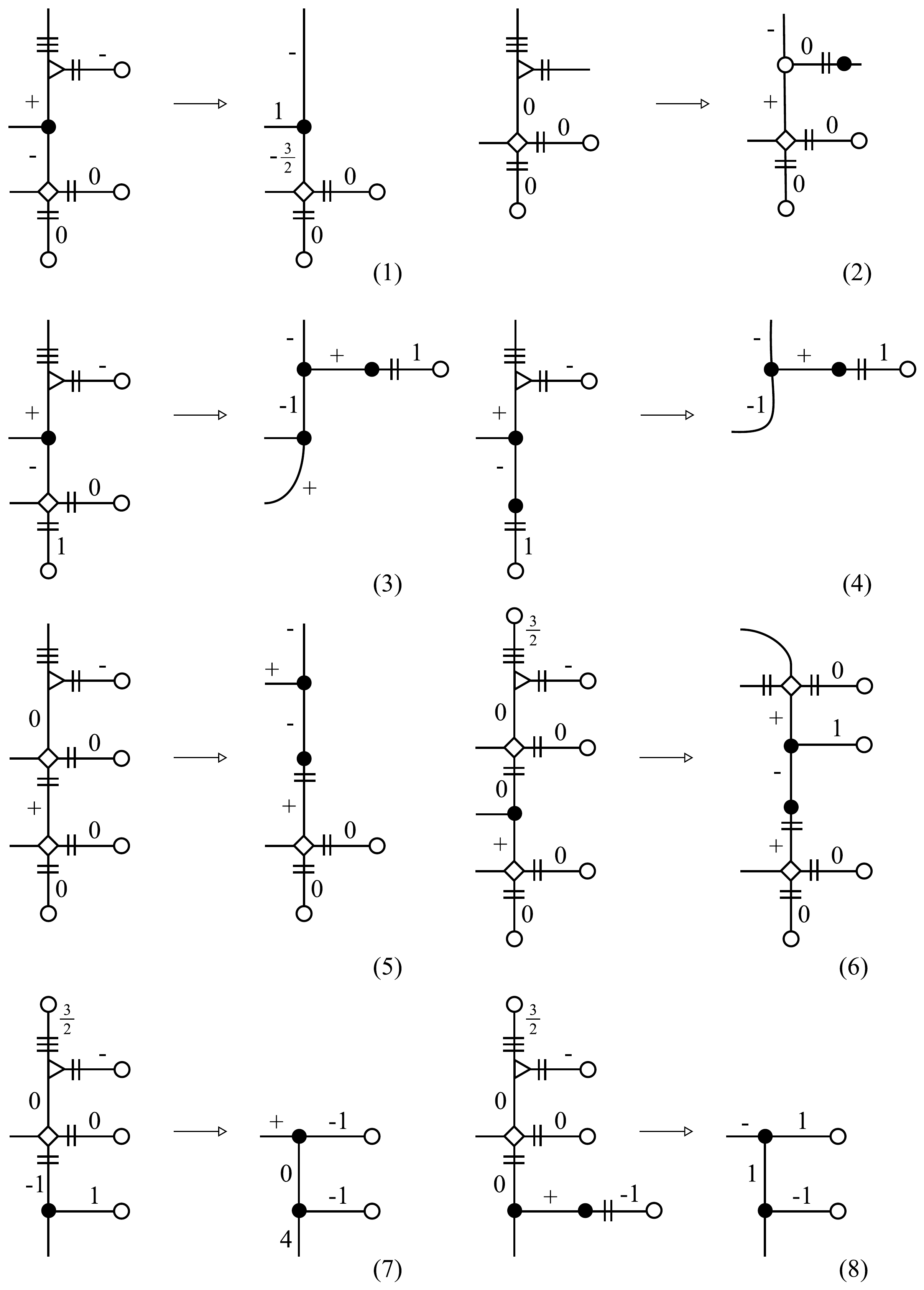}
\nota{These moves transform a shadow $X$ into another shadow $X'$ of the same block. Analogous moves hold with all signs reversed.}
\label{X10e11moves:fig}
\end{center}
\end{figure}

\begin{figure}
\begin{center}
\includegraphics[width = 9 cm]{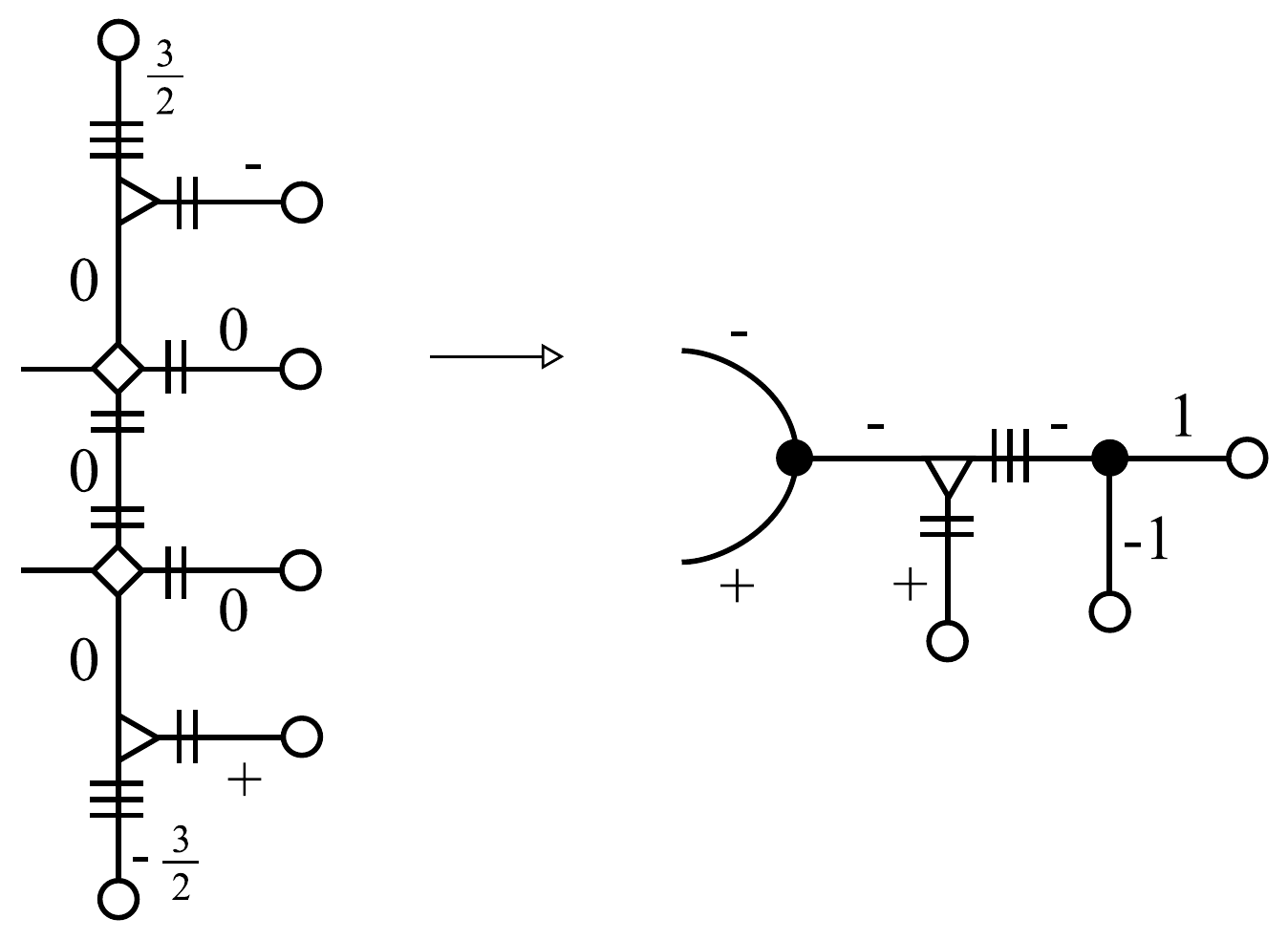}
\nota{This move transforms a shadow $X$ into another shadow $X'$ of the same block.}
\label{X10e11moves2:fig}
\end{center}
\end{figure}

\begin{figure}
\begin{center}
\includegraphics[width = 15 cm]{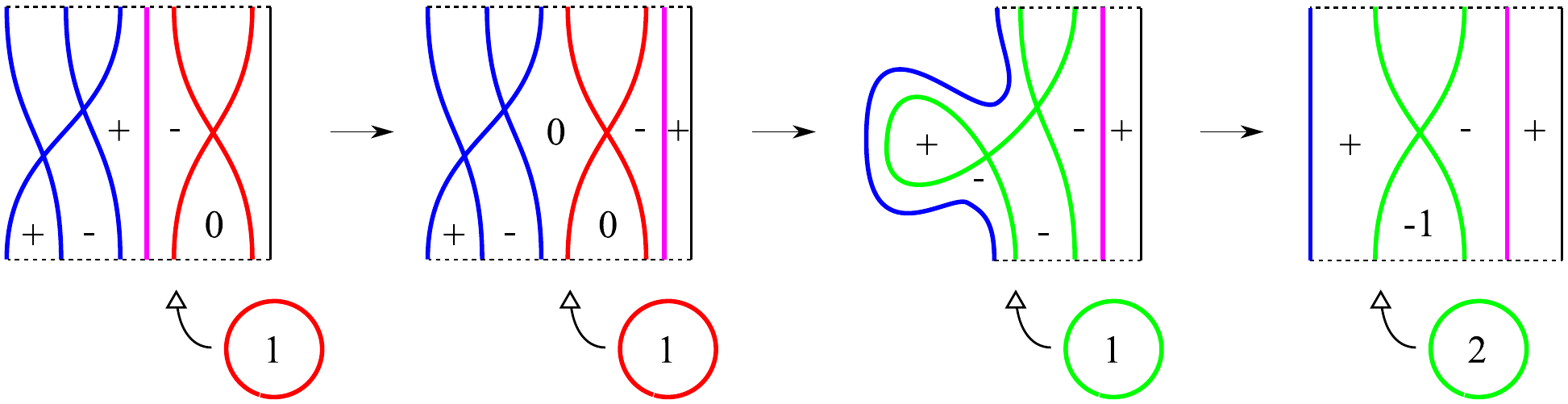}
\nota{Proof of Figure \ref{X10e11moves:fig}-(3). }
\label{oioia:fig}
\end{center}
\end{figure}

\begin{figure}
\begin{center}
\includegraphics[width = 16 cm]{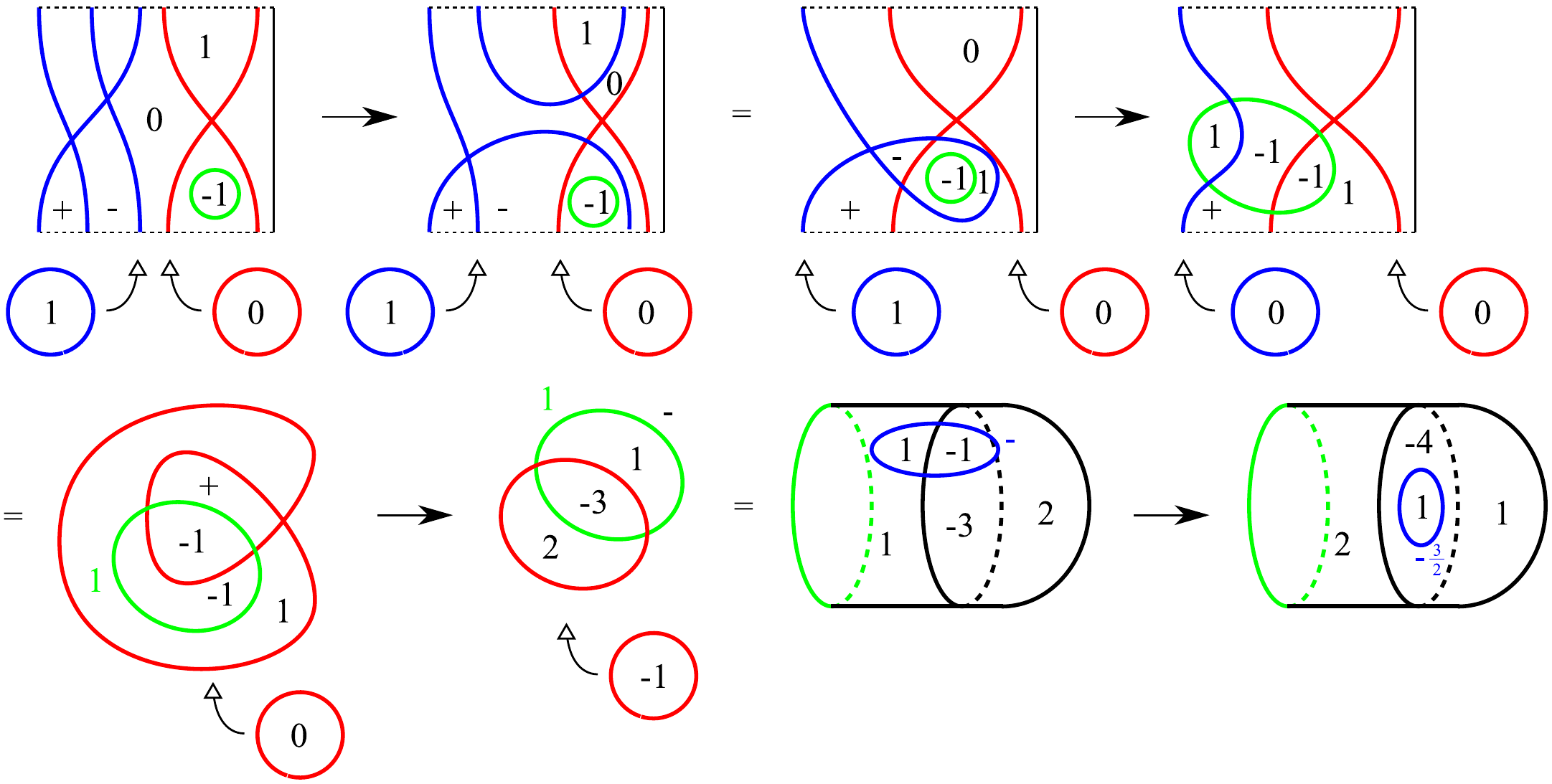}
\nota{Proof of Figure \ref{X10e11moves:fig}-(7). }
\label{oimmena:fig}
\end{center}
\end{figure}

\begin{figure}
\begin{center}
\includegraphics[width = 12 cm]{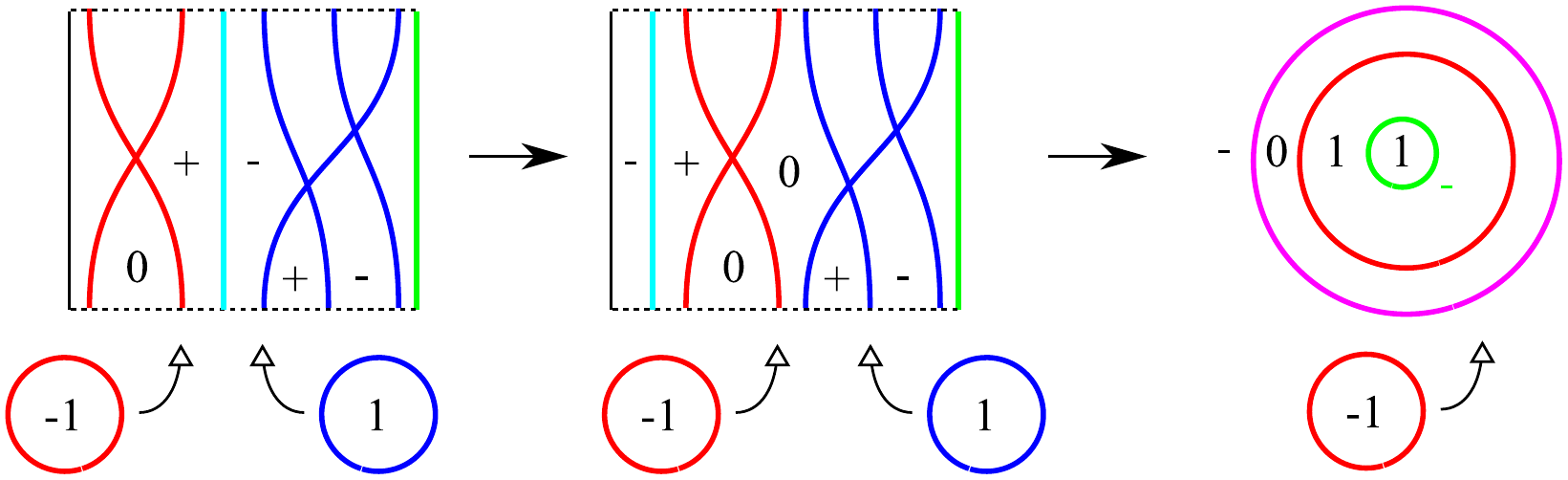}
\nota{Proof of Figure \ref{X10e11moves:fig}-(8). The second move is proved in Figure \ref{diamine:fig}.}
\label{diamine2:fig}
\end{center}
\end{figure}

\begin{figure}
\begin{center}
\includegraphics[width = 16 cm]{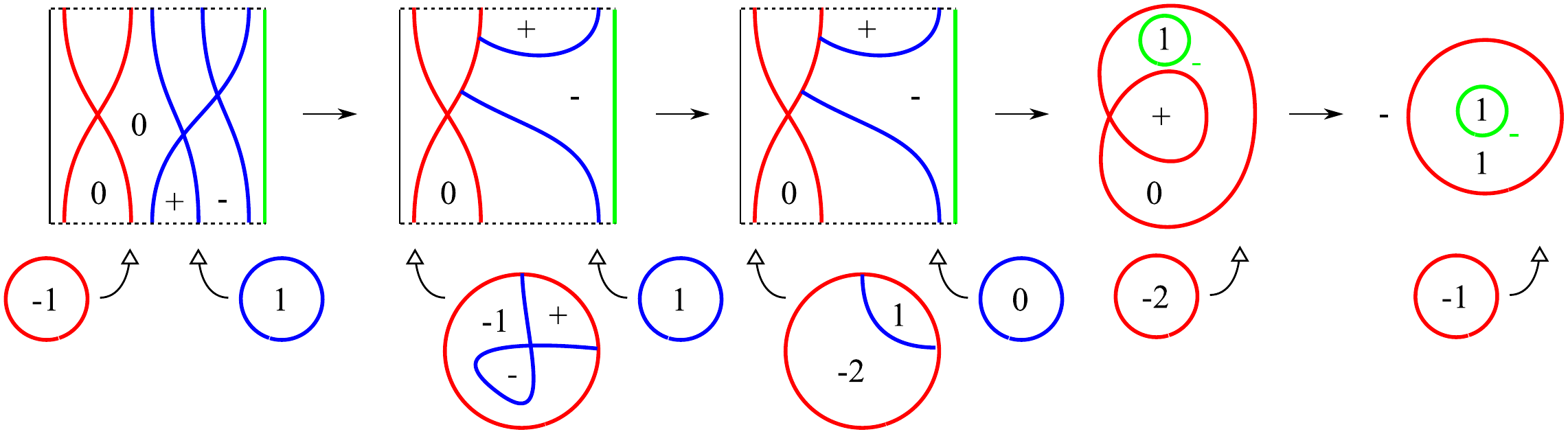}
\nota{Proof of the second move in Figure \ref{diamine2:fig}.}
\label{diamine:fig}
\end{center}
\end{figure}

\subsection{Moves with the vertex B}
Recall that the vertex B represents a boundary component of $X$. This vertex has a peculiar behaviour.

\begin{figure}
\begin{center}
\includegraphics[width = 15 cm]{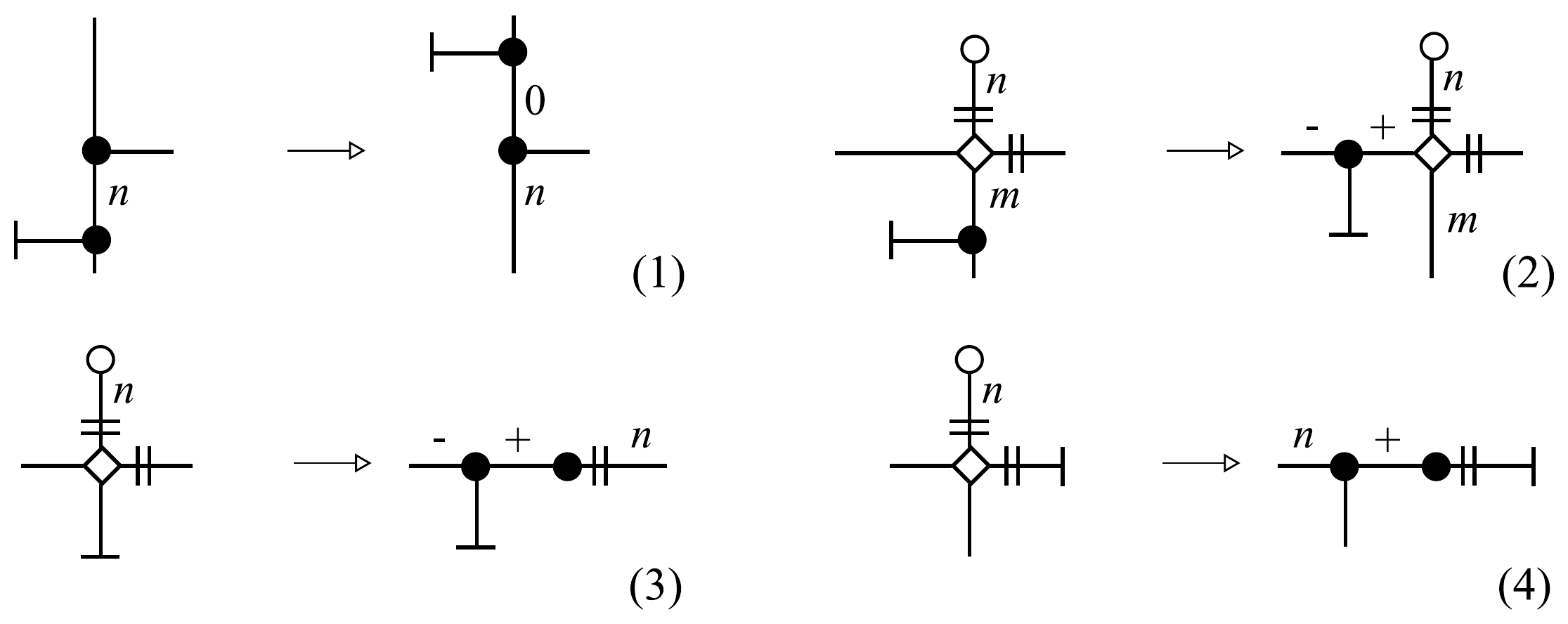}
\nota{These moves transform a shadow $X$ into another shadow $X'$ of the same block.}
\label{Bmoves:fig}
\end{center}
\end{figure}

\begin{prop}
The moves in Figure \ref{Bmoves:fig} modify a shadow $X$ into another shadow $X'$ of the same block.
\end{prop}
\begin{proof}
The moves (1) and (2) are drawn more explicitly in Figure \ref{movesB:fig}. In each move, both pieces are the result of drilling along homotopic (and hence isotopic) closed curves in $M$, so they are the same block. Move (3) is obtained from (2) by collapsing a region. In move (4), both pieces represent a thickened annulus drilled along a simple closed curve that runs twice along the annulus.
\end{proof}

\begin{figure}
\begin{center}
\includegraphics[width = 15 cm]{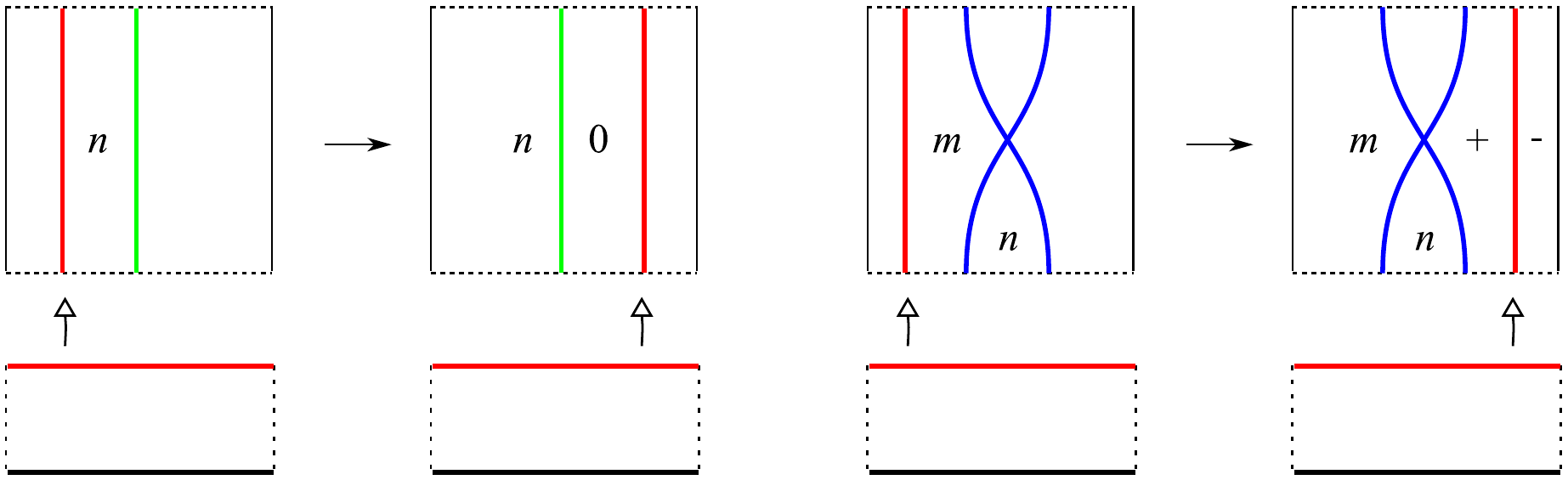}
\nota{The moves in Figures \ref{Bmoves:fig}-(1,2). The black curve indicates a component of $\partial X$.}
\label{movesB:fig}
\end{center}
\end{figure}

\section{Proof of Theorem \ref{main:teo} I: elimination of pieces.} \label{proof:section}
Having prepared all the necessary ingredients, we can now finally enter into the proof of Theorem \ref{main:teo}.

\subsection{The theorem}
Our aim is to prove the following half of Theorem \ref{main:teo}, that we state in the larger context of blocks.

\begin{teo} \label{main:block:teo}
Let $X$ be a shadow of some block $M$ with $c^*(X)\leq 1$. Then $M = M' \#_h \matCP^2$ for some $h\in \matZ$ and $M'$ generated by $\calS_1$.
\end{teo}

The rest of this paper is devoted to the proof of this theorem. By hypothesis $X$ decomposes into pieces 
$$D,\ P,\ Y_2,\ Y_{111},\ Y_{12},\ Y_3,\ X_1, \ \ldots,\  X_{11}$$ 
and can be described by a graph $G$ with vertices as explained in Section \ref{encoding:graph:subsection}. 

\subsection{Elimination of most pieces}
In the first step of the proof we quickly eliminate many of the possible pieces.

\begin{prop} \label{only:five:prop}
We can suppose that $X$ decomposes only in the following pieces:
$$D,\ Y_{111},\ Y_{12},\ X_{10},\  X_{11}.$$ 
\end{prop}
\begin{proof}
We extend the proof of \cite[Propositions 7.7 and 7.8]{Ma:zero} to this context. 
As shown there, every piece $P$, $Y_{2}$, or $Y_3$ has some boundary component $\gamma$ whose fibre torus $T\to \gamma$ has either a vertical or a horizontal compressing disc (the latter case only with $P$). If the disc is horizontal we apply to $G$ the corresponding move in Figure \ref{hv:fig}-(2) and eliminate $P$. If it is vertical, then all the boundary components of the piece have vertical compressing discs, so we apply Figure \ref{hv:fig}-(1) to each and then remove the piece from $X$. In both cases we obtain a shadow $X'$ of a block $M'$ such that $M$ is obtained from $M'$ by connected sum and assembling. After finitely many steps we have eliminated all the pieces $P, Y_2,$ and $Y_3$.

Let now $X$ contain a piece $X_i$ with $i=1,\ldots, 9$. Therefore $\partial N(X) = \#_h(S^2\times S^1)$ contains a hyperbolic manifold $W_i$. By Lemma \ref{peculiar:lemma} the tori of $W_i$ have slopes $s_i$ that bound discs in $\partial N(X)$ and give $\#_{h'}(S^2 \times S^1)$.

Propositions \ref{W12:prop}, \ref{W34:prop}, \ref{W56:prop}, \ref{W7:prop}, \ref{W8:prop}, and \ref{W9:prop} show that one of the following holds:
\begin{enumerate}
\item We have $s_j= \infty$ for all $j$,
\item We have $s_j=0$ for some $s_j$ that corresponds to an even component of $\partial X_i$ of length 1,
\item We have $s_j = \infty$ for some $s_j$ that corresponds to a component of $\partial X_i$ of length $\geq 3$.
\end{enumerate}

The slope $\infty$ corresponds to a vertical compressing disc. In case (1) all the boundary components have vertical compressing discs; therefore
we can apply Figure \ref{hv:fig}-(1) to all of them and we discard $X_i$ from $X$. 

In case (2) we get a horizontal disc. We apply the move in Figure \ref{hv:fig}-(2). After this move the length-1 component of $\partial X_i$ bounds a disc with gleam $s_j=0$. Now we can apply the move in \cite[Figure 6]{Co} that substitutes $X_i$ with a portion without vertices that should be further collapsed. After the collapsing and the removal of 1-dimensional portions (that correspond to removing $S^3 \times S^1$ summands) we end up with a shadow with a smaller number of vertices. 

Case (3) is similar to (1) because of the following general \emph{propagation principle}: if two regions adjacent to an edge of $SX'$ have a vertical compressing disc, then also the third one has. Since the component having the vertical compressing disc $s_j=\infty$ has length $\geq 3$, it runs twice on some edge of $SX_i$, and in a couple of steps we deduce that all the boundary components of $X_i$ have vertical compressing discs, so we conclude as above.

In all cases after finitely many steps we eliminate all the pieces $X_1,\ldots, X_9$.
\end{proof}

\subsection{The piece $X_{10}$.}
We have happily eliminated most of the pieces and we are left with only 4 of them. From now on, and until the end of this paper, we suppose that $X$ is a shadow of some block $M$ that decomposes into pieces homeomorphic to $D, Y_{111}, Y_{12}, X_{10}$, and $X_{11}.$ Let $G$ be a graph representing $X$. 

\begin{figure}
\begin{center}
\includegraphics[width = 6 cm]{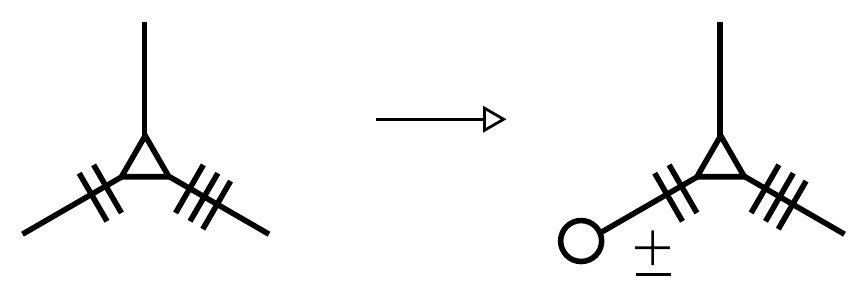}
\nota{We may suppose that every vertex in $G$ that represents $X_{11}$ is adjacent to a vertex representing a disc $D$ as shown here, with gleam $\pm \frac 12$.}
\label{X10_condition:fig}
\end{center}
\end{figure}

\begin{prop} \label{suppose:10:prop}
We can suppose that every vertex of $G$ of type $X_{10}$ is adjacent to a vertex $D$ as in Figure \ref{X10_condition:fig}.
\end{prop}
\begin{proof}
Let a vertex of $G$ denote a piece $X_{10} \subset X$.
The manifold $\partial N(X) = \#_h(S^2\times S^1)$ contains a hyperbolic manifold $W_{10}$. By Lemma \ref{peculiar:lemma} the boundary tori of $W_{10}$ have slopes $\alpha, \beta, \gamma$ that bound discs in $\partial N(X)$ and give $\#_{h'}(S^2 \times S^1)$.
Corollary \ref{W10:cor} shows that one of the following holds:
$$\alpha \in \big\{\infty, 0, \tfrac 12, 1\big\}, \qquad \beta \in \big\{\infty, -1, 0\big\}, \qquad
\gamma \in \big\{\infty, -3, -2 \big\}.$$
The corollary also says that if either $\alpha$ or $\beta$ equals to $\infty$, then some other slope is also $\infty$. Therefore, if any of the slopes $\alpha, \beta, \gamma$ is $\infty$, the propagation principle used in the proof of Proposition \ref{only:five:prop} shows that there are vertical compressing discs everywhere and we can discard $X_{10}$ from $X$. We are left with the cases
$$\alpha \in \big\{0, \tfrac 12, 1\big\}, \qquad \beta \in \big\{-1, 0\big\}, \qquad
\gamma \in \big\{-3, -2 \big\}.$$
If $\alpha = 0$ or $1$, we get a horizontal disc. We apply the move in Figure \ref{hv:fig}-(1). After this move the first component of $\partial X_{10}$ bounds a disc with gleam $\pm \frac 12$. Now we can apply Figure \ref{X10moves:fig}-(1) that destroys $X_{10}$.

If $\alpha = \frac 12$, we can add the disc as in Figure \ref{add_disc:fig}-(right). This operation adds to $G$ a new vertex $X_{11}$ near $X_{10}$, the two vertices being separated by a 0-gleamed edge. Then we apply Figure \ref{X10e11moves:fig}-(2) to destroy $X_{10}$. Thus we have substituted a $X_{10}$ with a $X_{11}$.

If $\beta = -1$ or $0$ we can add the horizontal disc and apply the move in Figure \ref{hv:fig}-(2). After the move we get a portion as in Figure \ref{X10_condition:fig}, as stated.

If $\gamma =-3$ or $-2$ we can add the horizontal disc and apply the move in Figure \ref{hv:fig}-(2). We get (up to reversing signs) a portion as in Figure \ref{X10moves:fig}-(3) and by applying the move there we transform it into a portion where $X_{10}$ is as stated.
\end{proof}

We will henceforth suppose that every vertex of $G$ of type $X_{10}$ is adjacent to a vertex $D$ as in Figure \ref{X10_condition:fig}.

\begin{rem}
When a piece $X_i$ has a boundary curve $\gamma$ of length 1, its cusp shape is one of the two shown in Figure \ref{shape1:fig}. The picture shows all the slopes $a,b,c,d,e,f$ of length $\leq 2$ in both cases. In the previous proofs we have shown that if either of $a,b,c,d,e,f$ bounds a compressing disc, then we can often construct a move that simplifies the shadow by destroying or discarding $X_i$.
\end{rem}

\begin{figure}
\begin{center}
\includegraphics[width = 6 cm]{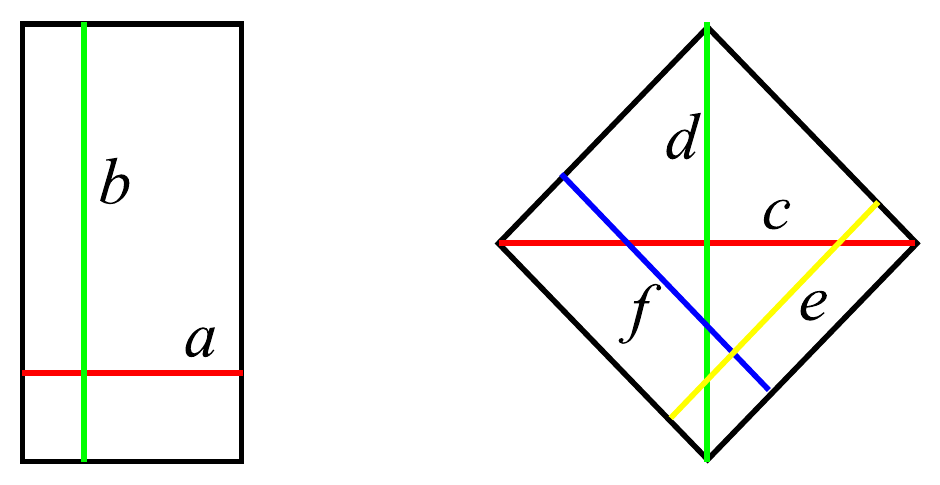}
\nota{The maximal cusp shape corresponding to a boundary curve $\gamma$ of length one is of one of these two types. The slopes of length $\leq 2$ are indicated as $a, b, c, d, e$ and have length 1, 2, 2, 2, $\sqrt 2$, $\sqrt 2$ respectively.}
\label{shape1:fig}
\end{center}
\end{figure}

\subsection{The piece $X_{11}$}
The piece $X_{11}$ is the one with the maximum number of boundary components.
We prove a fact similar to Proposition \ref{suppose:10:prop}.

\begin{figure}
\begin{center}
\includegraphics[width = 6 cm]{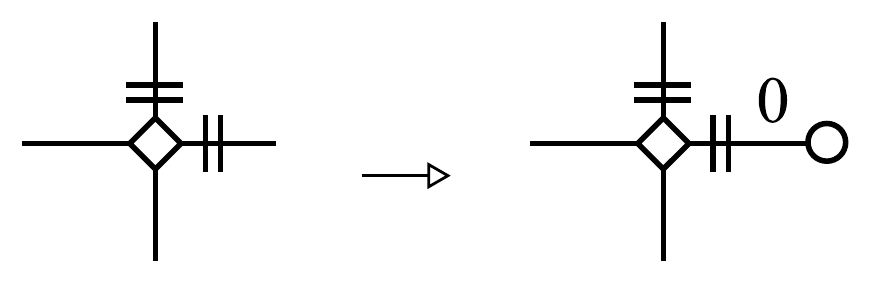}
\nota{We may suppose that every vertex in $G$ that represents $X_{11}$ is adjacent to a vertex representing a disc $D$ as shown here, with gleam $0$.}
\label{X11_condition:fig}
\end{center}
\end{figure}

\begin{prop} \label{suppose:11:prop}
We can suppose that every vertex of $G$ of type $X_{11}$ is adjacent to a vertex $D$ as in Figure \ref{X11_condition:fig}.
\end{prop}
\begin{proof}
Let a vertex in $G$ denote a piece $X_{11} \subset X$.
The piece $X_{11}$ determines a submanifold $W_{11} \subset \partial N(X) \cong \#_h(S^2 \times S^1)$. By Lemma \ref{peculiar:lemma} the tori of $W_{11}$ have slopes $\alpha, \beta, \gamma, \delta$ that bound discs in $\partial N(X)$ such that by filling $W_{11}$ along them we get $\#_{h'}(S^2 \times S^1)$. Corollary \ref{W11:cor} shows 
that one of the following 4 conditions holds: 
$$\alpha \in \big\{\infty, -1, -\tfrac 12, 0\big\}, \qquad 
\beta \in \big\{\infty, -1, -\tfrac 12, 0\big\}, \qquad
\gamma \in \big\{\infty, 0, 1, 2 \big\}, \qquad
\delta \in \big\{\infty, 0, 1, 2 \big\}.$$

If $\alpha$ or $\beta$ is equal to $-1$ or $0$ we add a horizontal disc as in Figure \ref{hv:fig}-(2) and we simplify using Figure \ref{X11moves:fig}-(1). If it is equal to $-\frac 12$ we add the disc as in Figure \ref{add_disc:fig}-(right). This operation adds to $G$ a new vertex $X_{11}'$ near the original one $X_{11}$, the two vertices being separated by a 0-gleamed edge. Then we apply Figure \ref{moves2:fig}-(3) to destroy the original $X_{10}$. Thus we have substituted an old $X_{10}$ with a new $X_{10}'$ that has the advantage of being adjacent to a couple of 0-gleamed discs, so it is as in Figure \ref{X11_condition:fig}.

If $\gamma$ or $\delta$ equals to $x=0,1,2$ we add the horizontal disc and apply Figure \ref{hv:fig}-(2). Now the piece $X_{11}$ is adjacent to a disc with gleam $x-1 = -1,0,1$. If the gleam is zero we get Figure \ref{X11_condition:fig} and we are done. If it is $\pm 1$ we destroy $X_{11}$ using Figure \ref{X11moves:fig}-(2).

If $\alpha$ or $\beta$ equals $\infty$, we get a vertical disc and we apply Figure \ref{hv:fig}-(1).
Corollary \ref{W11:cor} says that either $\gamma\in \matZ$, or $\delta \in \matZ$, or $\beta \in\{-1,0\}$, or $\gamma = \delta = \infty$. 
In the first (or second) case, we also add a horizontal disc to the boundary component corresponding to $\gamma$ (or $\delta$), apply Figure \ref{hv:fig}-(2) and get a portion as in Figure \ref{Bmoves:fig}-(3). That move destroys the vertex $X_{11}$.
The third case has already been considered, and in the fourth case the propagation principle used in the proof of Proposition \ref{only:five:prop} shows that there are vertical compressing discs everywhere and we can discard $X_{11}$ from $X$. 

If $\gamma$ or $\delta$ (say $\gamma$) equals $\infty$, we get a vertical disc and we apply Figure 
\ref{hv:fig}-(1). Corollary \ref{W11:cor} says that either $\alpha$, $\beta$, or $\delta$ belongs to $\matZ \cup \{\infty\}$. If it is $\infty$, we conclude by the propagation principle as above. If it is an integer, we add a horizontal disc and apply Figure \ref{hv:fig}-(2). We get a portion that describes an annulus drilled along a simple closed curve that runs either one or twice along the annulus. If it runs once we may substitute everything with a portion without vertices. If it runs twice we apply Figure \ref{Bmoves:fig}-(4).
\end{proof}

We will henceforth suppose that every vertex of $G$ of type $X_{11}$ is adjacent to a vertex $D$ as in Figure \ref{X11_condition:fig}.

\subsection{Boundary contributions}
Now that every $X_{10}$ and $X_{11}$ is adjacent to a disc, it is important to understand the contribution that the two pieces altogether give to $\partial N(X) = \#_h(S^2 \times S^1)$. 
 
\begin{figure}
\begin{center}
\includegraphics[width = 12.5 cm]{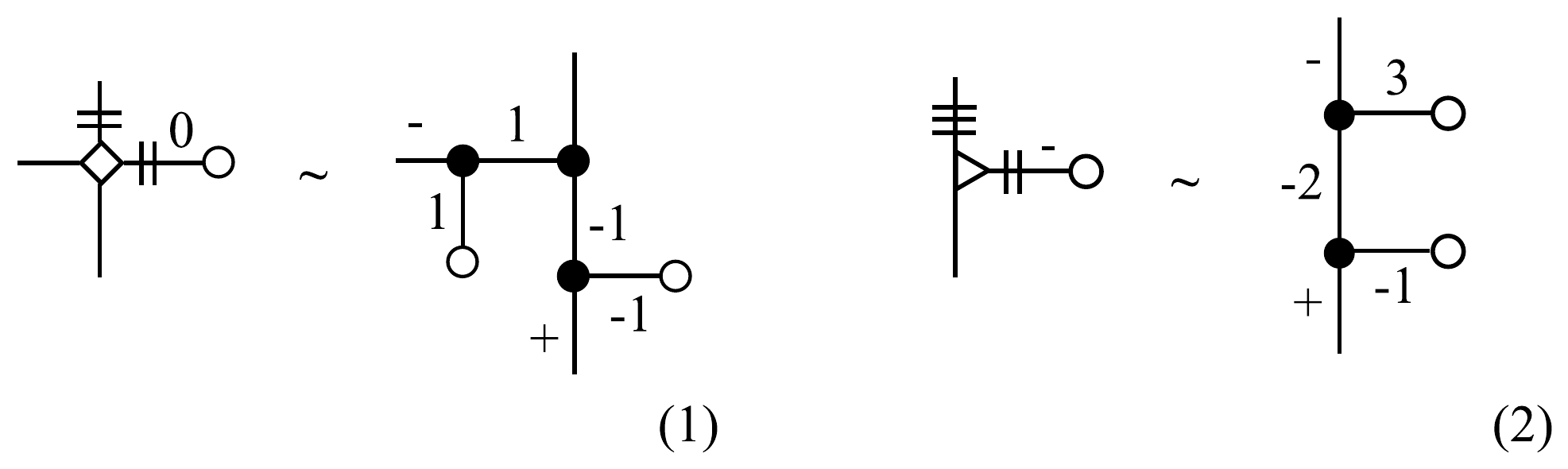}
\nota{These portions of graph contribute the same to $\partial N(X)$.} 
\label{boundary_contribution:fig}
\end{center}\end{figure}

\begin{prop}
Each pair of pieces in Figure \ref{boundary_contribution:fig} contributes the same to $\partial N(X)$.
\end{prop}
\begin{proof}
This is proved in Figure \ref{boundary_contribution_proof:fig}.
These portions contribute the same to $\partial N(X)$. In (1) we use \cite[Figure 71-(4)]{Ma:zero} and Figure \ref{X11moves:fig}-(3). In (2) we use Figure \ref{X10moves:fig}-(2).
\end{proof}

\begin{figure}
\begin{center}
\includegraphics[width = 16 cm]{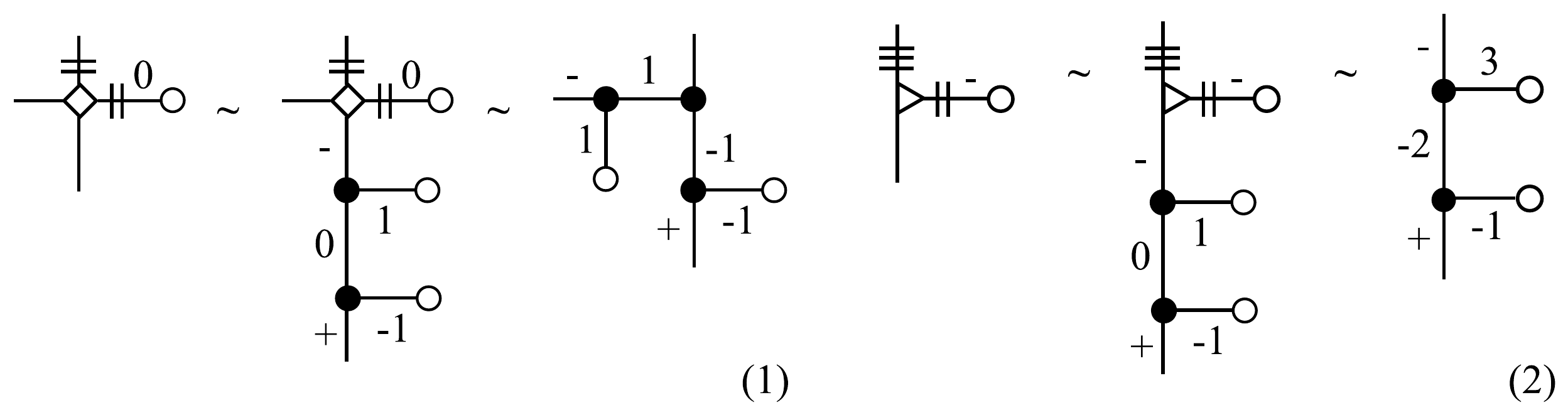}
\nota{Proof of Figure \ref{boundary_contribution:fig}.}
\label{boundary_contribution_proof:fig}
\end{center}\end{figure}

\begin{figure}
\begin{center}
\includegraphics[width = 8 cm]{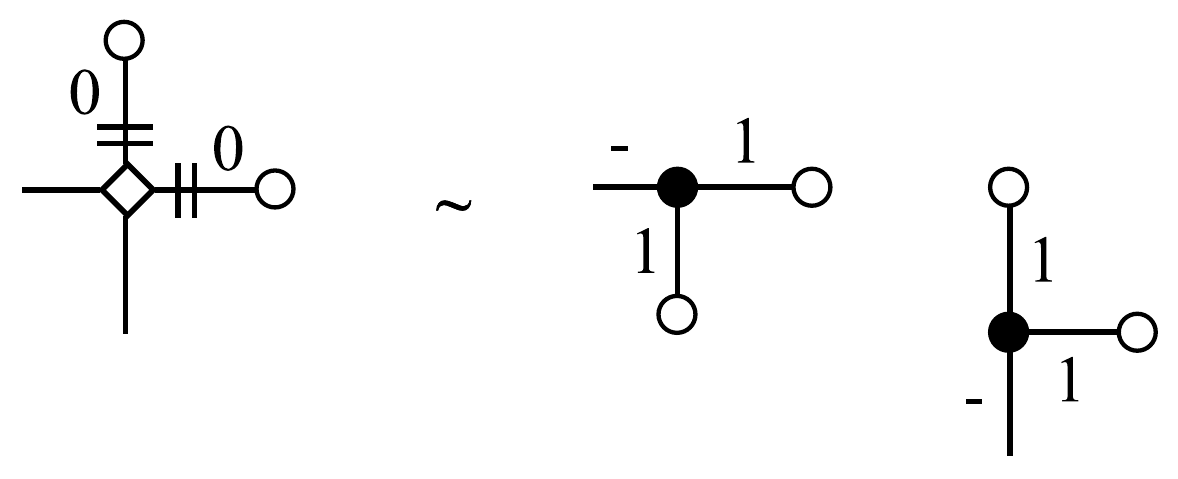}
\nota{These portions of graph contribute the same to $\partial N(X)$. There is a connected sum between the two components that is not indicated, so the manifold is in fact a connected sum of two solid tori.}
\label{Z00_boundary:fig}
\end{center}\end{figure}

We have just discovered that the two portions in Figure \ref{boundary_contribution:fig} contribute to $\partial N(X)$ with submanifolds diffeomorphic respectively to $P\times S^1$ and $(A,3)$, where the latter denotes the 
Seifert manifold with parameters $\big(A,(3,1)\big)$. 

Another piece that plays an important role in our proof is shown in Figure \ref{Z00_boundary:fig}-(left). 

\begin{prop}
The two pieces in Figure \ref{Z00_boundary:fig} contribute the same to $\partial N(X)$. \end{prop}
\begin{proof}
This is proved in Figure \ref{Z00_boundary_proof:fig}. On (1) we use Figure \ref{boundary_contribution:fig} and then apply the opposite of Figure \ref{sum:fig}. To get (2) we use \cite[Figure 71-(3)]{Ma:zero}.
\end{proof}

\begin{figure}
\begin{center}
\includegraphics[width = 16 cm]{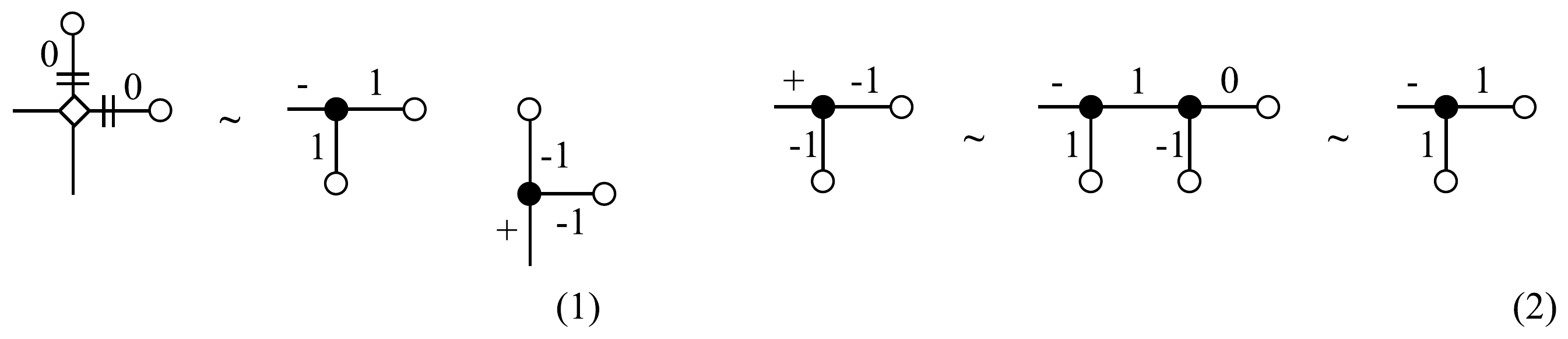}
\nota{Proof of Figure \ref{Z00_boundary:fig}.}
\label{Z00_boundary_proof:fig}
\end{center}\end{figure}

The piece shown in Figure \ref{Z00_boundary:fig} contributes with a connected sum of two solid tori, obtained from $P\times S^1$ by a fiber-parallel Dehn filling.

\section{Proof of Theorem \ref{main:teo} II: decorated tree with levels.}
We conclude in this section the proof of Theorem \ref{main:teo}.

\subsection{Conditions on the shadow $X$}
Let $X$ be a shadow with $c^*(X) \leq 1$. As proved in the previous section, we may suppose that $X$ decomposes only into pieces homeomorphic to
$$D,\ Y_{111},\ Y_{12},\ X_{10},\ {\rm or}\ X_{11}.$$ 
Moreover, each piece $X_{10}$ or $X_{11}$ is adjacent to a disc as in Figures \ref{X10_condition:fig} and \ref{X11_condition:fig}.

\subsection{The decorated graphs $G$ and $G'$}
Let $G$ be a decorated graph that represents $X$, with vertices of type 
$D, Y_{111}, Y_{12}, X_{10}$, and $X_{11}$. 
We now construct a new graph $G'$ from $G$ as shown in Figure \ref{G':fig}. Every time we find a vertex $X_{10}$ together with a $\pm$-gleamed disc as in (1), we substitute it with a new type of vertex as shown there. When we find a vertex $X_{11}$ that is attached to \emph{two} 0-gleamed discs along its length-2 boundaries as in (2), we substitute it with two new vertices as shown. If it is adjacent to only one, we substitute it with one new vertex as in (3).

\begin{figure}
\begin{center}
\includegraphics[width = 16 cm]{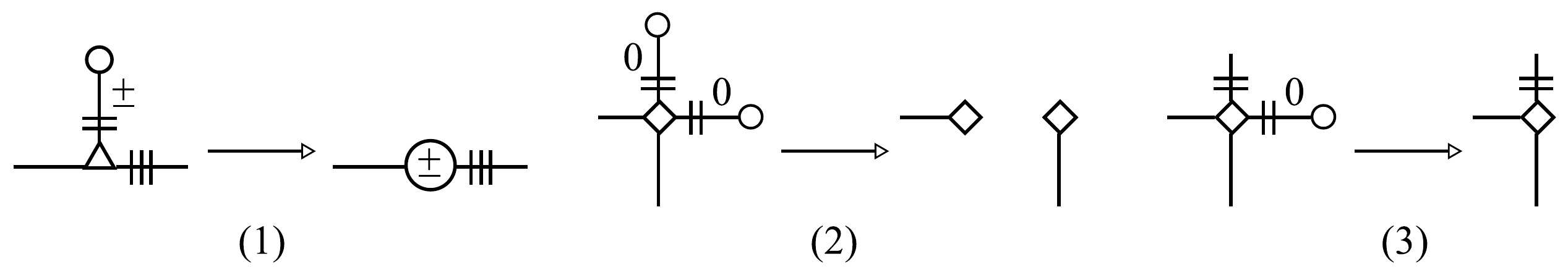}
\nota{How to transform $G$ into $G'$.}
\label{G':fig}
\end{center}\end{figure}

The graph $G'$ may be disconnected because of (2). Each connected component $G_i'$ determines a decomposition of $\#_h(S^2 \times S^1)$ along tori, where $h$ depends on $i$. The new vertices created in Figure \ref{G':fig} represent portions homeomorphic to $(A,3)$, solid tori, and $P\times S^1$, as prescribed by Figures \ref{boundary_contribution:fig} and \ref{Z00_boundary:fig}. We now study this decomposition in detail.

Summing up, the graph $G'$ contains vertices of these types:
$$\includegraphics[width = 1 cm]{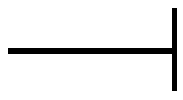} \quad 
\includegraphics[width = 1 cm]{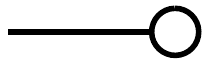}\quad 
\includegraphics[width = 1 cm]{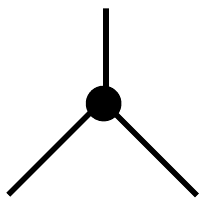}
\quad \includegraphics[width = 1 cm]{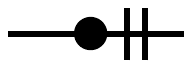}
\quad \includegraphics[width = 1.7 cm]{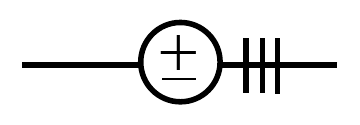}
\quad \includegraphics[width = 1 cm]{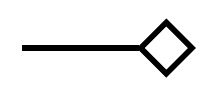}
\quad \includegraphics[width = 1.3 cm]{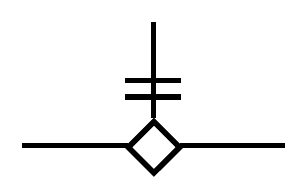}
$$

Each 1-valent vertex contributes to the decomposition of $\#_h(S^2 \times S^1)$ with a solid torus; the 2-valent vertices contribute with the Seifert manifolds $(A,2)$ and $(A,3)$; each 3-valent vertex contributes with $P\times S^1$.

A \emph{flat} vertex is a vertex $v$ if type \includegraphics[width = 0.6 cm]{0.pdf}.

\begin{prop} \label{no:flat:Z0:prop}
We may require that every flat vertex is adjacent to a vertex \includegraphics[width = 0.6 cm]{4.pdf}.
\end{prop}
\begin{proof}
If the flat vertex is adjacent to some other vertex $v$, it furnishes a vertical compressing disc. The arguments already used in the proofs of Proposition \ref{suppose:10:prop} and \ref{suppose:11:prop} show that $v$ can be either discarded or simplified in all cases except \includegraphics[width = 0.6 cm]{4.pdf}.
\end{proof}

The proof now pursues as follows: we have a decomposition of $\#_h(S^2 \times S^1)$ into Seifert manifolds of 4 types $D\times S^1$, $P \times S^1$, $(A,2)$, and $(A,3)$. Such a decomposition must simplify somewhere (in a sense that we need to state precisely) and in all cases this simplification translates into a simplification move on $X$, so we conclude by iteration. 

To apply rigorously this idea we unfortunately need to face some technicalities. The first is to equip $G$ with the structure of a \emph{tree with levels}, as in \cite{Ma:zero}. This structure will allow us to identify a simplification of the 3-dimensional boundary. The second technical part is the translation of this 3-dimensional simplification into a simplification of the shadow $X$. Unfortunately, there will be many cases to consider. All the moves in Figures \ref{X11moves:fig}, \ref{X10moves:fig}, \ref{X10e11moves:fig}, and \ref{X10e11moves2:fig} will be needed.

\subsection{Decorated trees with levels}
We now need to recall a more technical notion from \cite{Ma:zero}. 
Recall that $G'$ has some connected components $G_i'$. A \emph{level function} on $G'_i$ is a function that assigns a non-negative integer $l(v)$ (the level) to every vertex $v$ of $G'_i$, such that the following holds:

\begin{enumerate}
\item There are $k\geqslant 2$ vertices with level zero, which form a path $v_1,\ldots, v_k$ called \emph{root}. 
\item Every vertex $v$ of valence 2 or 3   
is adjacent to precisely one vertex $v'$ of strictly higher level  $l(v')>l(v)$;
\item For every $L\geq 0$ the portion of vertices $v$ with bounded level $l(v)\leq L$ form a (connected) subtree. 
\end{enumerate}

\begin{figure}
\begin{center}
\includegraphics[width = 7 cm]{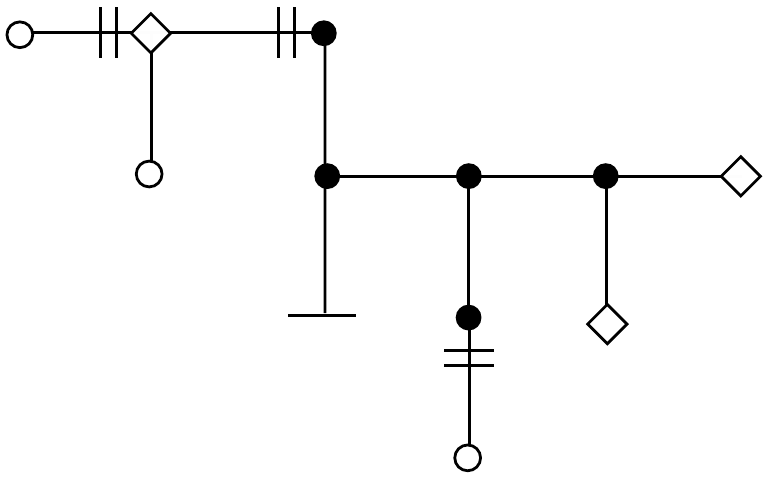}
\nota{A tree with levels. The level function may be deduced from the picture, with the convention that vertices at the same height have the same level, and vertices lying below have higher level than those lying above. There are three vertices at level 0 (the root), five vertices at level 1, three at level 2, and one at level 3.}
\label{leveled:fig}
\end{center}
\end{figure}

In particular, the graph $G_i'$ is a tree. One example is shown in Figure \ref{leveled:fig}. There is also a fourth condition related to the induced decomposition of $\#_h(S^2\times S^1)$ that we state soon. For every vertex $v$, let $S_v$ be the set of all vertices $v'$ such that there is a path
$$v = v_1, \ldots, v_k = v'$$
with $k\geq 2$ and $l(v_2)>l(v_1)$. 

The set $S_v$ is non-empty if and only if $v$ has valence 2 or 3. 
When non-empty, the set $S_v$ contains precisely one vertex adjacent to $v$, and possibly more; the set $S_v$ forms a subtree of $G_i'$. The vertex $v$ is the only one in $G_i'\setminus S_v$ which is adjacent to some vertex in $S_v$. We say that $S_v$ is the \emph{branch} that starts from $v$. The vertex $v$ is the \emph{base} of the branch.

\begin{rem}
We note that this terminology differs from that used in \cite{Ma:zero}, where we distinguished between \emph{branches}, \emph{leaves}, and \emph{fruits}. The presence of additional kinds of vertices here would complicate too much the terminology, so we prefer to use the term branch in all the possible cases.
\end{rem}

Every vertex $v$ in $G_i'$ determines a submanifold $M_v \subset \#_h(S^2 \times S^1)$ bounded by tori. If $S$ is a set of vertices of $G_i'$, we write $M_S = \cup_{v\in S} M_v$. 

For every $v\in G_i'$ 
of valence 2 or 3, the submanifold $M_{S_v}$ is connected and has only one boundary torus, attached to one boundary torus of $M_v$. The piece $M_v$ is a Seifert manifold, homeomorphic to either $P^2\times S^1$, $(A,2)$, or $(A,3)$. In all cases, the Seifert fibration is unique up to isotopy and induces a fibration on the boundary tori. We can now state the fourth and last requirement for our level function $l$.

\begin{enumerate}
\addtocounter{enumi}{3}
\item For every $v$
of valence 2 or 3, the manifold $M_{S_v}$ should be a solid torus, whose meridian is attached to a section of the fibration of $M_v$.
\end{enumerate}

\subsection{Reduction to decorated trees with levels}
We now adapt \cite[Theorem 8.1]{Ma:zero} to our context. 

\begin{prop} \label{leveled:prop}
We may suppose that each graph $G_i'$ has a structure of decorated tree with levels.
\end{prop}
\begin{proof}
The claim in the proof of \cite[Theorem 8.1]{Ma:zero} holds also here with the same proof, and it says that either a piece $P\times S^1$ has a fibre-parallel compressing disc, or $G_i'$ has a structure of decorated tree with levels. In the latter case we are done, so we investigate the former. The piece $P\times S^1$ is determined by some 3-valent vertex $v$ of $G'$.

If $v$ is \includegraphics[width = 0.6 cm]{4.pdf} the move in Figure \ref{scoppia:fig} applies and simplifies $X$. If $v$ is \includegraphics[width = 0.8 cm]{8.pdf}, we
can attach a horizontal compressing disc and apply Figure \ref{hv:fig}-(2) to transform the vertex into one as in Figure \ref{Z00_boundary:fig}-(left). After the move $G'$ contains one 3-valent vertex less and we conclude by induction on their number.
\end{proof}

\begin{figure}
\begin{center}
\includegraphics[width = 5 cm]{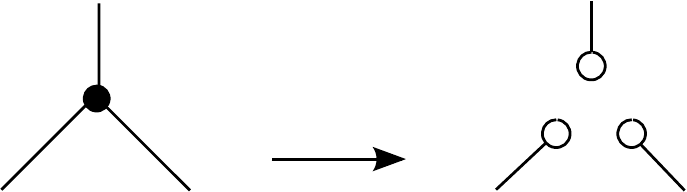}
\nota{This simplifying move applies when the fiber of the $P^2\times S^1$ lying above the vertex bounds a compressing disc (which is horizontal with respect to the 3 incident regions in the shadow).}
\label{scoppia:fig}
\end{center}
\end{figure}

We will henceforth suppose that each decorated graph $G_i'$ has a structure of decorated tree with levels. A decorated tree with levels describes a 3-manifold that is either $S^3$ or $S^2 \times S^1$, see \cite[Proposition 8.6]{Ma:zero}.

\subsection{Symmetries}
We will henceforth consider only trees with vertices representing the pieces $B, D, Y_{111}, Y_{12},\ X_{10}$, or $X_{11}$. We note that each of these pieces has a symmetry that fixes each boundary component and reverses its orientation. Therefore, as anticipated in Section \ref{encoding:graph:subsection}, there will be no need to write explicitly how two adjacent pieces are glued.

\subsection{Nice flat vertices}
A flat vertex $v$ in $G_i'$ is \emph{nice} if it is adjacent to a \includegraphics[width = 0.6 cm]{4.pdf} of strictly lower level, and moreover $v$ should
not be contained in a portion as in Figure \ref{not_nice:fig}. 

\begin{figure}
\begin{center}
\includegraphics[width = 2.5 cm]{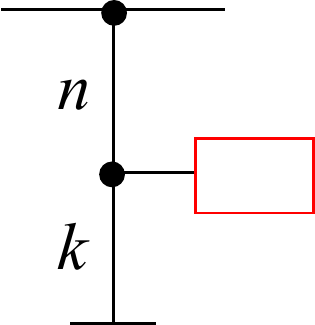}
\nota{A flat vertex that is not nice.}
\label{not_nice:fig}
\end{center}
\end{figure}

\begin{prop} \label{nice:prop}
We may suppose that every flat leaf in $G_i'$ is nice.
\end{prop}
\begin{proof}
As in \cite[Proposition 9.6]{Ma:zero}, this is done using the moves in \cite[Figure 53]{Ma:zero}. 
\end{proof}

\begin{figure}
\begin{center}
\includegraphics[width = 11 cm]{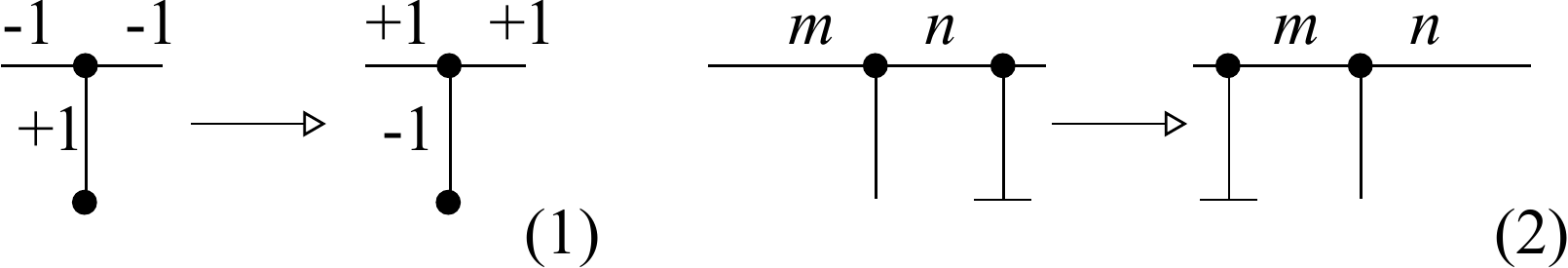}
\nota{These moves related different decorated tree with levels of the same block $M$.}
\label{move_leaf:fig}
\end{center}
\end{figure}

We will henceforth suppose that every flat leaf in $G_i'$ is nice.
The moves in Figure \ref{move_leaf:fig} modify a shadow into another shadow of the same block.

\subsection{Torsion of branches}
We are particularly interested in the 3 types of branches shown in Figure \ref{branch_torsion:fig}. Each such branch has a \emph{torsion} $q \in \matZ$, defined in \cite[Section 9.6]{Ma:zero} as follows.

Let $v$ be the base of the branch. It defines a block $M_v$ diffeomorphic to the Seifert manifold $P\times S^1$, $(A,2)$, or $(A,3)$. The branch defines a solid torus $M_{S_v}$ attached to a boundary torus $T$ of $M_v$. The torus $T$ has a preferred homology basis $(\mu, \lambda)$. The meridian $\mu$ is the vertical fiber $\pi^{-1} (x)$ of a point in $X$ along the natural projection $\pi\colon \partial N(X) \to X$. The longitude $\lambda$ is the fiber of the (unique) Seifert fibration of $M_v$, that is indeed horizontal here. The block $M_v$ is oriented, so $T$ also is, and we orient the pair $(\mu, \lambda)$ positively.

The meridian of the solid torus $M_{S_v}$ is attached to a section of the fibration, that is to a curve $\mu + q\lambda$ for some $q\in \matZ$. The integer $q$ is the \emph{torsion} of the branch. 

\begin{figure}
\begin{center}
\includegraphics[width = 8 cm]{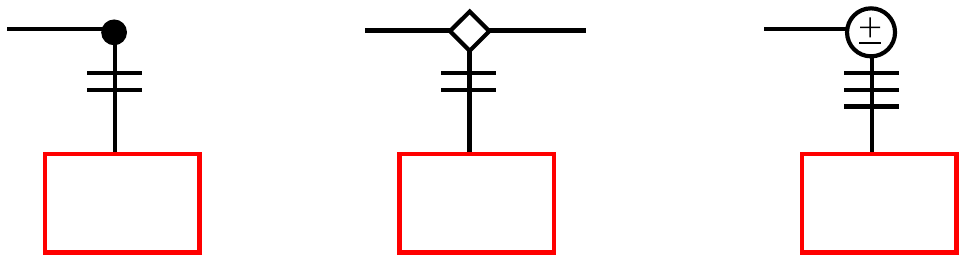}
\nota{Three particular types of branches.}
\label{branch_torsion:fig}
\end{center}
\end{figure}

\begin{figure}
\begin{center}
\includegraphics[width = 7 cm]{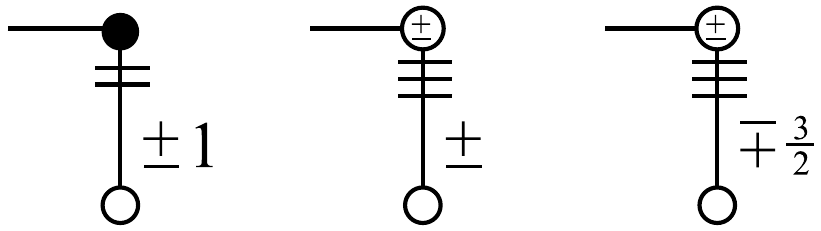}
\nota{Three particular types of branches with torsion $|q|=1$. We have $q=1$ if and only if the sign of the gleam $\pm 1, \pm, \pm \frac 32$ is positive (regardless of the sign of the small $\pm$ contained in the white vertex).}
\label{torsion:fig}
\end{center}
\end{figure}

\begin{prop} \label{branch:prop}
At every branch as in Figure \ref{branch_torsion:fig} we may suppose that:
\begin{itemize}
\item If $v$ is \includegraphics[width = 1 cm]{8.pdf}, then $|q| \geq 2$.
\item If $v$ is \includegraphics[width = 0.8 cm]{5.pdf}
or \includegraphics[width = 1.2 cm]{9.pdf}, either $|q| \geq 2$, or $|q|=1$ and the branch is as in Figure \ref{torsion:fig}.
\end{itemize}
\end{prop}
\begin{proof}
If $q=0$, there is a vertical compressing disc, so we conclude by Figure \ref{hv:fig} and Proposition \ref{no:flat:Z0:prop}. 

If $q= \pm 1$ there is a horizontal compressing disc, so after applying Figure \ref{hv:fig} the branch $S_v$ becomes a single vertex \includegraphics[width = 0.6 cm]{1.pdf} with the appropriate gleam that depends on the sign of $q$. 
If $v$ is \includegraphics[width = 0.6 cm]{8.pdf}, we get a portion as in Figure \ref{X11moves:fig}-(2) that may be simplified. In the other cases we get a branch as in Figure \ref{torsion:fig}.
\end{proof}

We will henceforth suppose that the conclusion of Proposition \ref{branch:prop} holds. If the torsion is $q = \pm 2$, there is yet something that we can do.

\begin{figure}
\begin{center}
\includegraphics[width = 12 cm]{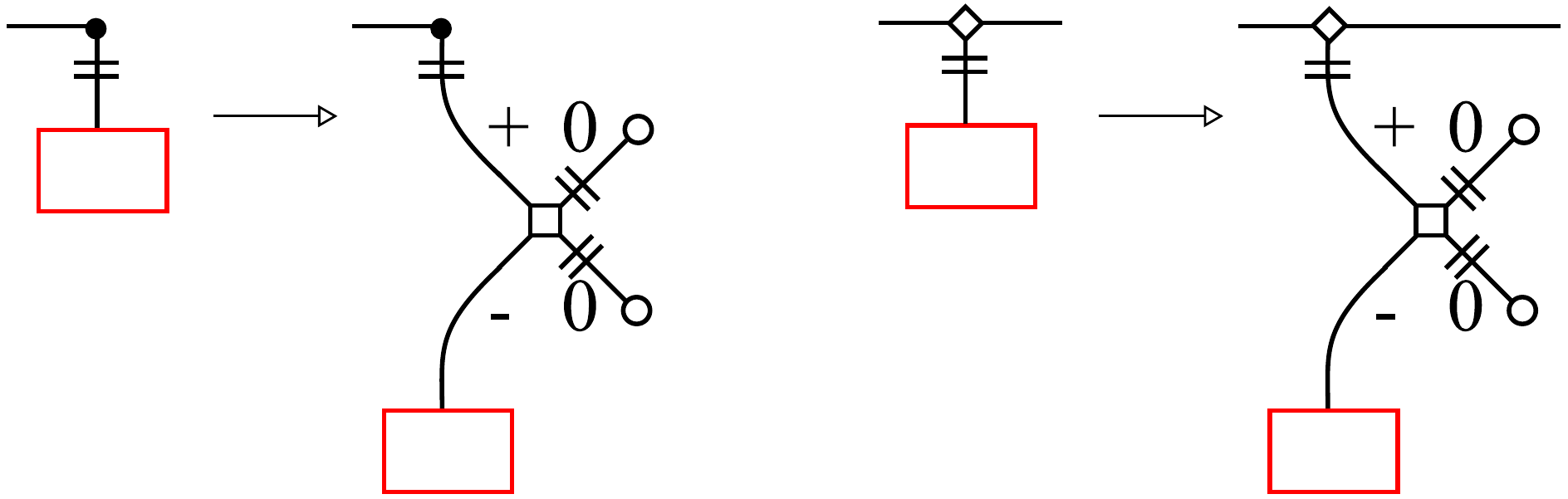}
\nota{If $q=2$ this move transforms $X$ into another shadow $X'$ of the same block.}
\label{q2:fig}
\end{center}
\end{figure}

\begin{prop}
If $q=2$ the move in Figure \ref{q2:fig} transforms the shadow $X$ into another shadow $X'$ of the same block $M$.
\end{prop}
\begin{proof}
The compressing disc winds twice and we add it as in Figure \ref{add_disc:fig}.
\end{proof}

\subsection{Plumbing lines}
The core of our arguments is a strengthened version of a theorem of Neumann and Weintraub \cite{Neu}, which deals with plumbings of spheres.

Recall that a plumbing line of spheres is determined by a sequence $(e_1,\ldots, e_n)$ of Euler numbers $e_i\in \matZ$. The following lemma is proved in \cite[Lemma 10.2]{Ma:zero}.

\begin{lemma} \label{plumbing:lemma}
Let $(e_1,\ldots,e_n)$ be a plumbing line, whose boundary is homeomorphic to either $S^3$ or $S^2\times S^1$. Up to reversing the sequence and/or changing all signs, one of the following holds.
\begin{itemize}
\item $e_1 = 0$,
\item $e_1=1$ and $n=1$,
\item $e_1=1$ and $e_2\in \{0,1,2,3\}$,
\item $e_i=0$ for some $i\not\in \{1,n\}$ and $e_{i-1}e_{i+1}\leqslant 0$,
\item $e_i= 1$ for some $i\not\in \{1,n\}$ and $e_{i-1} \in\{0,1,2,3\}, e_{i+1}\geqslant 0$.
\end{itemize}
\end{lemma}

The following moves modify the plumbing line while preserving its boundary:
\begin{eqnarray*}
(\ldots,e_{i-1},0,e_{i+1},\ldots) & \longrightarrow & (\ldots,e_{i-1}+e_{i+1},\ldots), \\
(0,e_2,e_3,\ldots ) & \longrightarrow & (e_3,\ldots), \\
(\ldots,e_{i-1},1,e_{i+1},\ldots) & \longrightarrow & (\ldots,e_{i-1}-1,e_{i+1}-1,\ldots), \\
(1,e_2,\ldots) & \longrightarrow & (e_2-1,\ldots). 
\end{eqnarray*}
Analogous moves hold with all the signs reversed.

\subsection{A particular portion}
We are approaching the conclusion of the proof of Theorem \ref{main:block:teo}. By the results proved in the last pages, we may suppose that each $G_i'$ is a decorated tree with levels, that every flat leaf there is nice, and also that Proposition \ref{branch:prop} holds.

Our goal is to show that $X$ simplifies somewhere, and to this purpose we now prove the existence of a particular portion $H$ of $G_i'$ wherein this simplification will be guaranteed to take place.

\begin{figure}
\begin{center}
\includegraphics[width =16 cm]{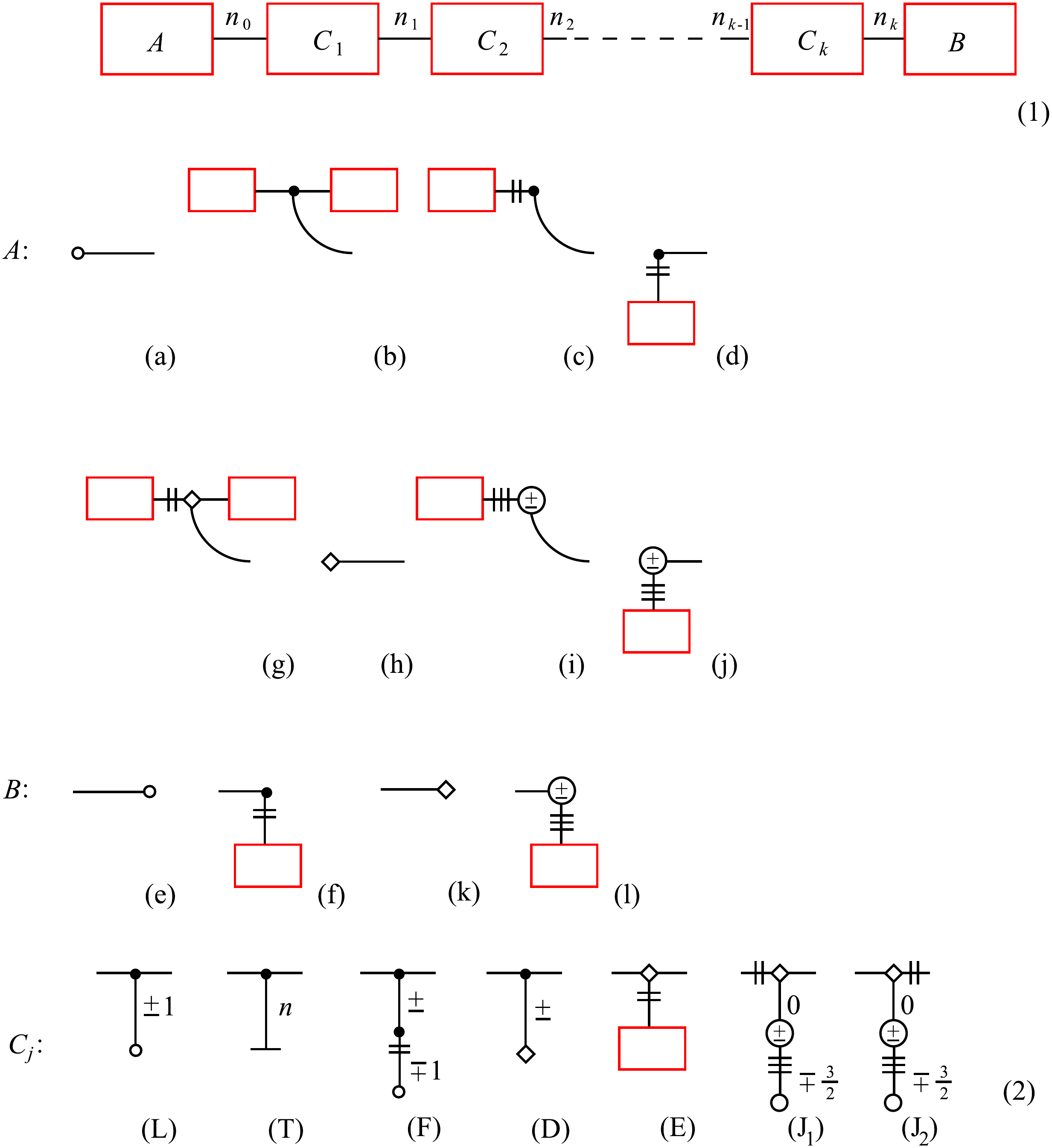}
\nota{The tree $G_i'$ contains a portion $H$ as in (1) for some $k\geq 0$, where $A, B$, and each $C_i$ is of one of the types shown in (2). (If A is of type (a), (d), or (h) the portion $H$ is actually the whole tree $H=G_i'$.) The (half-)integers $n_i$ decorate the horizontal edges and are arbitrary. There are two possible instances $J_1$ and $J_2$ of the same portion $J$ depending on left-right orientation.}
\label{branch:fig}
\end{center}
\end{figure}

\begin{prop}
The graph $G_i'$ contains a portion $H$ as in Figure \ref{branch:fig}-(1) for some $k\geq 0$, where the various possibilities for the subportions $A$, $B$, $C_i$ are shown in Figure \ref{branch:fig}-(2). Moreover, the portion $H$ is not equal to $(L), (T), (F), (D)$ nor $(J)$.
\end{prop}
\begin{proof}
Let $v$ be a vertex in $G_i'$ with the highest level among those such that the branch $S_v$ is non-empty, and such that:
\begin{enumerate}
\item the rooted branch $\{v\} \cup S_v$ is not as in Figure \ref{branch_torsion:fig}, and
\item the rooted branch $\{v\} \cup S_v$ is not equal to $(L), (T), (F), (D)$, nor $(J)$.
\end{enumerate}
If such a vertex $v$ exists, the rooted branch $\{v\} \cup S_v$ is our $H$ and we are done. If $v$ does not exist, the whole $H=G_i'$ is as required. 
\end{proof}

\begin{figure}
\begin{center}
\includegraphics[width = 8 cm]{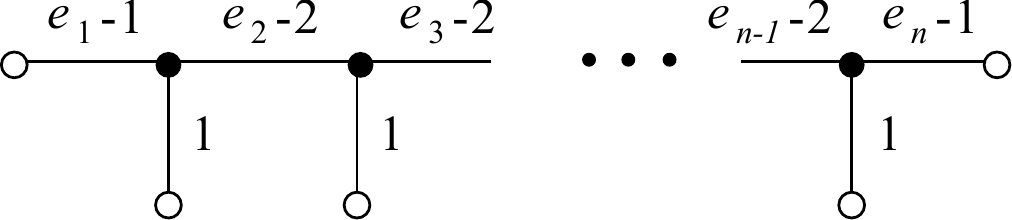}
\nota{A tree $H'$ with one level, that may be obtained as the perturbation of a plumbing line with Euler numbers $(e_1,\ldots, e_n)$.}
\label{plumbing_line2:fig}
\end{center}
\end{figure}

\begin{figure}
\begin{center}
\includegraphics[width = 16 cm]{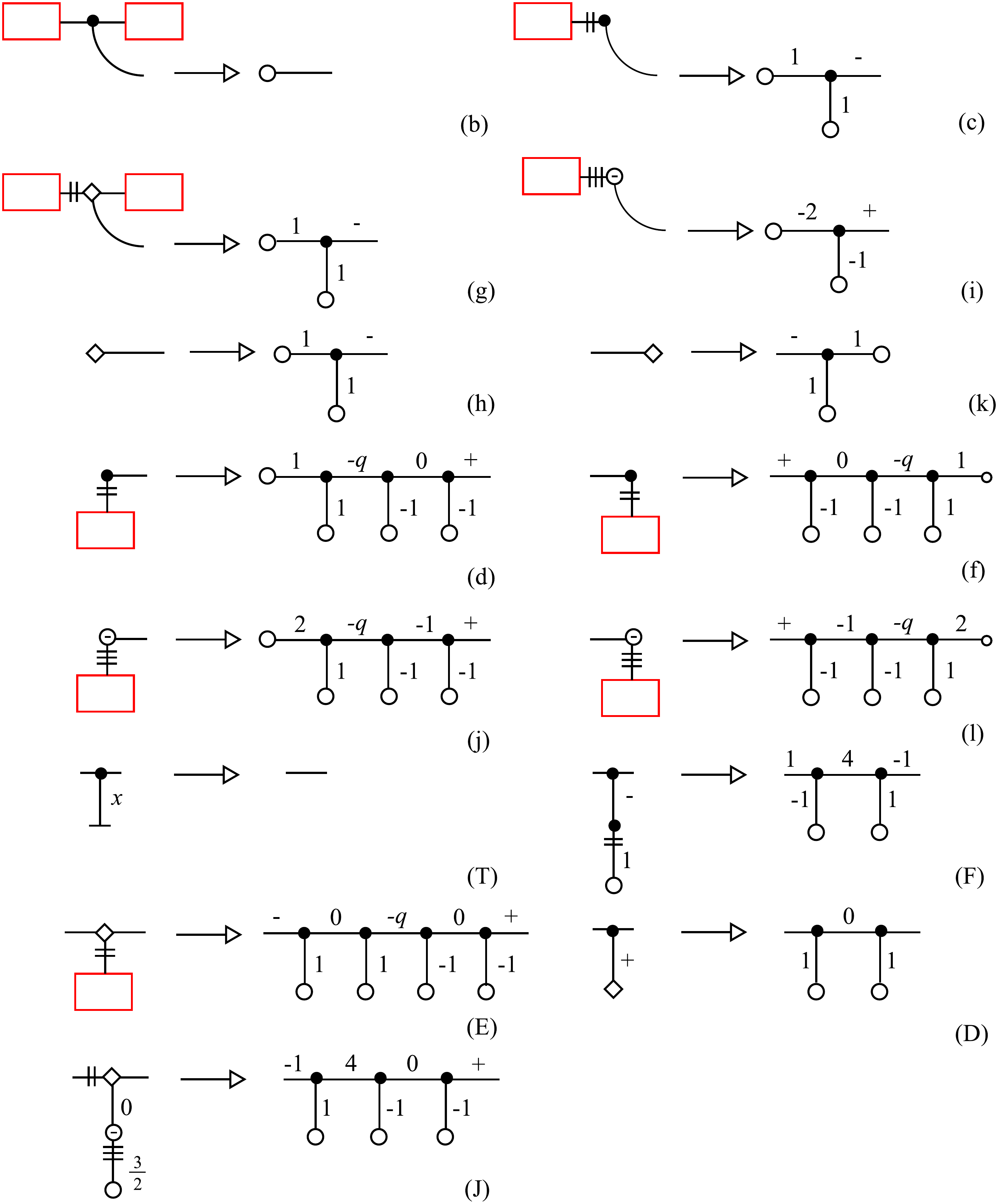}
\nota{These moves produce a graph for $\#_h(S^2 \times S^1)$. Here $q$ is the torsion of the branch. Everything holds also with all signs reversed.}
\label{ripeto:fig}
\end{center}
\end{figure}

We now construct from $H$ a tree $H'$ as in Figure \ref{plumbing_line2:fig}. We do this by substituting each piece $A$, $B$, and $C_i$ as prescribed by Figure \ref{ripeto:fig}. The resulting graph still describes a decomposition of $S^3$ or $S^2 \times S^1$ because of the following.

\begin{prop}
The moves in Figure \ref{ripeto:fig} transform $G_i'$ into another decorated tree with levels, that still encode a manifold $S^3$ or $S^2 \times S^1$. If one of the moves (b), (c), (g), or (i) is performed, the manifold is necessarily $S^3$.
\end{prop}
\begin{proof}
Concerning (b), (c), (d), (f), (T), and (F), this was already proved in \cite[Figure 67]{Ma:zero}. Moves (g) and (i) follow from Figure \ref{boundary_contribution:fig} and (b).
Moves (h), and (k) follow from Figure \ref{Z00_boundary:fig}. Moves (j), (l), (E), and (J) are proved like in \cite[Proposition 9.10]{Ma:zero}, using Figure \ref{boundary_contribution:fig}. Move (D) is obtained from Figure \ref{Z00_boundary:fig} by sliding an edge as in Figure \ref{thickening:fig}-(1).
\end{proof}

We note that $H'$ is just the perturbation of a plumbing line of spheres with labels $(e_1,\ldots, e_n)$. The encoded 3-manifold is either $S^3$ or $S^2\times S^1$ and Lemma \ref{plumbing:lemma} applies.

\subsection{Conclusion of the proof}
The proof ends as in \cite{Ma:zero}. Lemma \ref{plumbing:lemma} furnishes a list of possibilities, and we prove that each determines a portion in $G$ that may be simplified. Unfortunately, there are around 100 configurations to analyze by hand.

To preserve clarity, we first suppose that $Z$ does not contain any flat leaves. The portions $A$, $B$, $C_i$ contribute to the plumbing line $(e_1,\ldots, e_n)$ as shown in Table \ref{contributions:table}, following Figure \ref{ripeto:fig}.

\begin{table}
\begin{center}
\begin{tabular}{ccccccc}
\phantom{\Big|} (c) , (g), (h) & (k) & (d) & (f) & (j) & (l) \\
 \hline 
$(2,\ldots)$ & \!$(\ldots, 2)$\! & \!$(2,-q,-2, \ldots)$\! & \!$(\ldots, -2, -q, 2)$\! & \!$(3,-q,-3, \ldots)$\! & \!$(\ldots, -3, -q, 3)$\!
\phantom{\Big|} 
\end{tabular}

\begin{tabular}{ccccccc}
\phantom{\Big|} (i) & (F) & (D) & (E) & (${\rm J}_1$) & (${\rm J}_2$) \\
 \hline 
\!$(-3, \ldots)$\! & \!$(\ldots, 4, \ldots)$\! & \!$(\ldots, 2, \ldots)$\! & \!$(\ldots, 2,-q,-2,\ldots)$\! & \!$(\ldots,4,-2,\ldots)$\! & \!$(\ldots,-2,4,\ldots)$\!
\phantom{\Big|} 
\end{tabular}

\caption{Contribution of each piece to the plumbing line $(e_1,\ldots, e_n)$.}
\label{contributions:table}
\end{center}
\end{table}

We recall from Proposition \ref{branch:prop} that in (d), (f), (j), and (l) we have $|q|\geq 1$ and in (E) we have $|q|\geq 2$. Moreover, if $|q|=1$ then the portion is as in Figure \ref{torsion:fig}.

\subsection{The case $k=0$}
We consider first the case $k=0$, \emph{i.e.}~there is no $C_i$. The portion $Z$ thus consists of the pieces $A$ and $B$ glued together. The various possibilities, considered only up to reversing all signs, are shown in Figure \ref{AB:fig}. The cases in the grey box were already covered in \cite{Ma:zero}: in each case either the shadow $X$ can be simplified, or we can do some move that decreases the level on some vertex in $G_i'$ and we conclude by induction on the sum of the levels of the vertices, see the end of \cite[Theorem 11.3]{Ma:zero}. 

\begin{figure}
\begin{center}
\includegraphics[width = 16 cm]{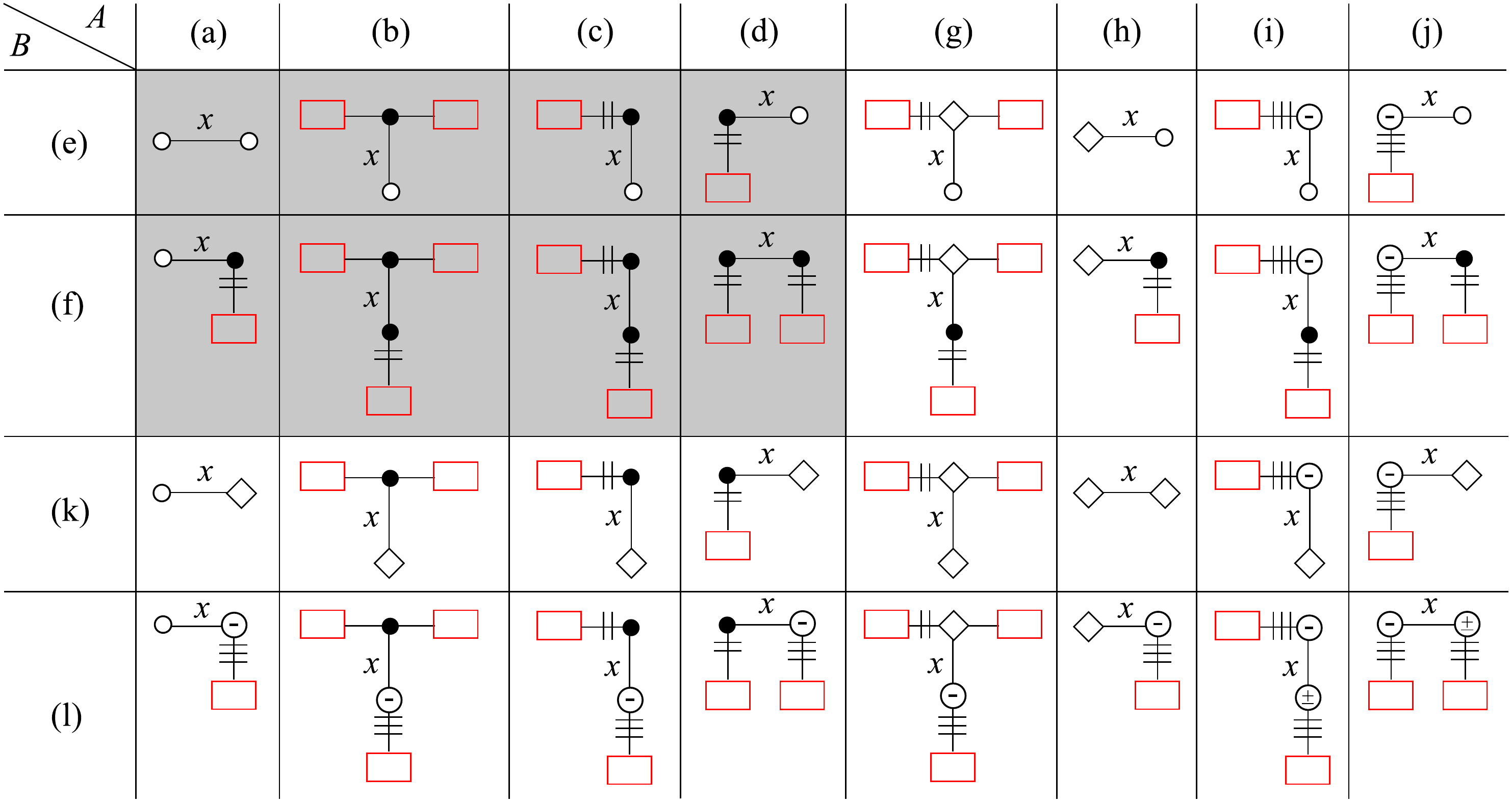}
\nota{When $k=0$, the portion $Z$ consists of $A$ and $B$ glued together and thus looks like one of the pictures listed here (up to changing simultaneously all signs).}
\label{AB:fig}
\end{center}
\end{figure}

We analyze each of the remaining cases separately. We apply Lemma \ref{plumbing:lemma} throughout the discussion.

\begin{itemize}
\item[(eg)] The sequence is $(2,x+\frac 12)$, so $x = \pm \frac 12$. The graph $G$ hence contains a portion as in Figure \ref{X11moves:fig}-(1) and can be simplified.
\item[(eh)] Same as above.
\item[(ei)] The sequence is $(-3, x-\frac 12)$, so $x=\frac 12$. The graph contains a portion as in Figure \ref{X10moves:fig}-(1) and can be simplified.
\item[(ej)] The sequence is $(3,-q,-3,x-\frac 12)$ with $|q| \geq 1$, and it never gives $S^3$ nor $S^2 \times S^1$. 
\item[(fg)] The sequence is $(2, x, -2, -q, 2)$ with $|q|\geq 1$, and it never gives $S^3$ nor $S^2 \times S^1$.
\item[(fh)] Same as above.
\item[(fi)] The sequence is $(-3,x-1,-2,-q,2)$ with $|q|\geq 1$. This gives $x=0$ and $q\in \{1,2\}$. If $q=1$ we get a portion as in Figure \ref{torsion:fig} and everything simplifies using Figure \ref{X10moves:fig}-(4). If $q=2$ we perform the move in Figure \ref{q2:fig}, then the inverse of Figure \ref{X10moves:fig}-(3) and finally Figure \ref{X10e11moves:fig}-(2). As a result we have substituted a vertex $X_{10}$ with a vertex $X_{11}$ and we conclude by induction on the number of vertices $X_{10}$.
\item[(fj)] The sequence is $(3, -q, -3, x-1, -2, -q', 2)$ with $|q|,|q'|\geq 1$. We get $x=0$ and the sequence is equivalent to $(3, -q+1, -q'+2, 2)$. We get either $q=1$, or $q'=1$, or $q'=2$. In the first case we simplify the shadow using Figure \ref{X10moves:fig}-(5). In the two subsequent cases we conclude as above. 
\item[(ka)] = (eh).
\item[(kb)] The sequence is $(x+\frac 12,2)$, so $x = \pm \frac 12$
However, this is excluded since $H \neq (D)$.
\item[(kc)] The sequence is $(2, x+1, 2)$ and it never gives $S^3$. 
\item[(kd)] = (fh).
\item[(kg)] We conclude as in (kc).
\item[(kh)] The sequence is $(2,x+1,2)$ and hence $x=0$. If the two vertices represent the same $X_{11}$ then $G$ is as in Figure \ref{examples_shadows:fig}-(right), so $M=\matRP^3\times S^1$ and we are done. If they represent two different $X_{11}$ then we simplify using Figure \ref{X11moves:fig}-(5).
\item[(ki)] The sequence is $(-3, x, 2)$. Therefore $x=0$ and Figure \ref{X10e11moves:fig}-(2) applies.
\item[(kj)] The sequence is $(3, -q, -3, x, 2)$. Therefore $x=0$ and we conclude as above.
\item[(la)] = (ej).
\item[(lb)] The sequence is $(x-\frac 12, -3, -q, 3)$ with $|q| \geq 1$, and it never gives $S^3$ nor $S^2 \times S^1$. 
\item[(lc)] The sequence is $(2, x, -3, -q, 3)$, so $x=0$ and $q=1$. We use Figure \ref{X10moves:fig}-(5).
\item[(ld)] = (fj).
\item[(lg)] The sequence is $(2, x, -3, -q, 3)$, so $x=0$ and $q=1$. This is excluded since $H \neq (J)$.
\item[(lh)] = (ej)
\item[(li)] There are two cases corresponding to the sign $\pm$. If positive, we get the sequence $(-3, x, 3, q, -3)$ that never gives $S^3$. If negative, we get $(-3, x-1, -3, -q, 3)$ that never gives $S^3$ neither.
\item[(lj)] There are two cases corresponding to the sign $\pm$. If positive, we get the sequence $(3, -q, -3, x, 3, q', -3)$, so $x=0$ and Figure \ref{X10moves:fig}-(7) applies. If negative, we get $(3, -q, -3, x-1, -3, q', 3)$ that never gives $S^3$ nor $S^2 \times S^1$.
\end{itemize}

\begin{figure}
\begin{center}
\includegraphics[width = 16 cm]{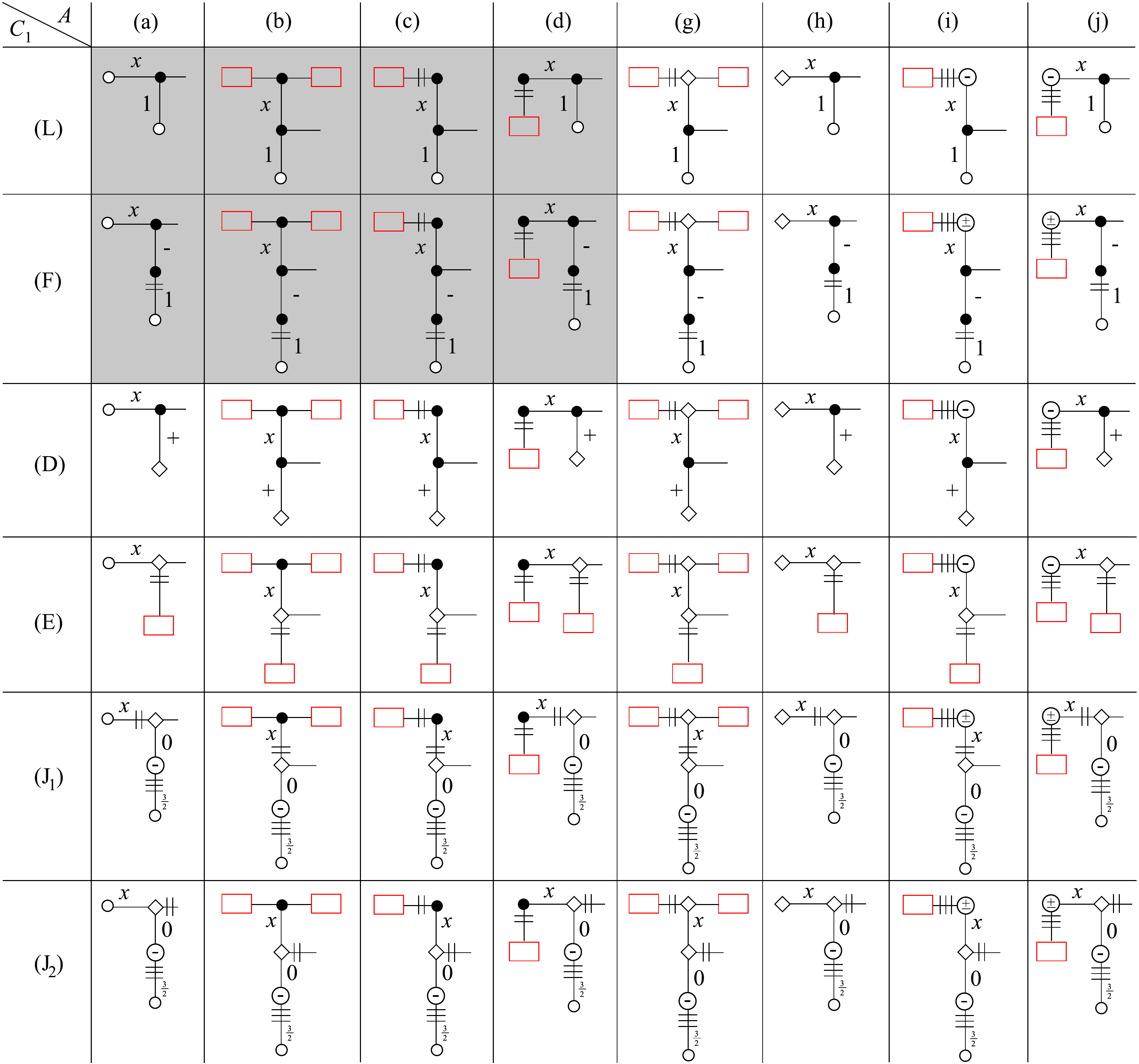}
\nota{Portions obtained as the union of $A$ and $C_1$ (up to changing simultaneously all signs).}
\label{AC:fig}
\end{center}
\end{figure}

\subsection{The case $k\geq 1$}
We now turn to the case $k\geq 1$. We first consider a portion formed by A and $C_1$ as in Figure \ref{AC:fig}. We use implicitly Figure \ref{move_leaf:fig}-(1) at various points.

The cases in the grey box were already covered in \cite{Ma:zero}, and we analyze each of the remaining cases separately. 
\begin{itemize}
\item[(Lg)] The sequence starts as $(2,x +\frac 32, \ldots)$. If $x+\frac 32 \in \{0,1\}$ then $x \in \big\{-\frac 32, - \frac 12 \big\}$ and the shadow simplifies as in Figure \ref{X11moves:fig}-(3).
\item[(Lh)] Same as above.
\item[(Li)] The sequence starts as $(-3, x+\frac 12, \ldots)$. If $x+\frac 12=0$ then $x=-\frac 12$ and we use Figure \ref{X10moves:fig}-(2). If $x+\frac 12 = -1$ the sequence must simplify also somewhere else.
\item[(Lj)] The sequence starts as $(3, -q, -3, x+\frac 12, \ldots)$ with $|q|\geq 1$ and we conclude as above.
\item[(Fg)] The sequence starts as $(2, x+\frac 12, 4, \ldots)$. If $x+\frac 12=1$ then $x=\frac 12$ and the shadow simplifies as in Figure \ref{X11moves:fig}-(8).
\item[(Fh)] Same as above.
\item[(Fi)] There are two cases to consider depending on the sign $\pm$. In the positive case, the sequence starts as $(3, x+\frac 12, 4, \ldots)$. If $x+\frac 12=1$ the sequence must simplify somewhere else. In the negative case, it starts as $(-3, x-\frac 12, 4, \ldots)$. If $x-\frac 12=0$ then $x=\frac 12$ and we use Figure \ref{X10e11moves:fig}-(4).
\item[(Fj)] Similar as above.
\item[(Da)] The sequence starts as $(x+1,2, \ldots,)$. If $x+1 \in \{0,1\}$ then $x\in \{-1,0\}$. If $x=0$ we simplify as in Figure \ref{sum:fig}. If $x=-1$ we simplify as in Figure \ref{X11moves:fig}-(3).
\item[(Db)] As above we conclude that $x\in \{-1,0\}$. If $x=0$ we may slide the edge as in \cite[Figure 71-(5)]{Ma:zero} and we conclude by induction on the sum of the levels of all vertices. If $x=-1$ we use Figure \ref{X11moves:fig}-(4) to get $x=0$.
\item[(Dc)] The sequence starts as $(2, x+\frac 32, 2, \ldots)$. If $x+\frac 32=1$ then $x=-\frac 12$ and we simplify $X$ as in Figure \ref{X11moves:fig}-(10).
\item[(Dd)] The sequence starts as $(2,-q,-2,x+\frac 12, 2, \ldots)$. If $x+\frac 12 = 0$ we conclude as above.
\item[(Dg)] The sequence starts as $(2, x+\frac 32, 2, \ldots)$. If $x+\frac 32 = 1$ then $x=-\frac 12$ and we simplify $X$ as in Figure \ref{X11moves:fig}-(9).
\item[(Dh)] Same as above.
\item[(Di)] The sequence starts as $(-3, x+\frac 12, 2, \ldots)$. If $x+\frac 12 = 0$ then $x=-\frac 12$ and we simplify as in Figure \ref{X10e11moves:fig}-(1) after applying Figure \ref{X11moves:fig}-(4).
\item[(Dj)] The sequence starts as $(3, -q, -3, x+\frac 12, 2, \ldots)$ and we conclude as above.
\item[(Ea)] The sequence starts as $(x+\frac 12, 2, -q, -2, \ldots)$ with $|q|\geq 2$. If $x= \frac 12$ we use Figure \ref{X11moves:fig}-(1). 
\item[(Eb)] The sequence starts as $(x+\frac 12, 2, -q, -2, \ldots)$ with $|q| \geq 2$. If $x= \frac 12$, we may replace this initial sequence with $(-q-1,-2, \ldots)$. Now $-q-1 \neq -1,0$ and hence the sequence must simplify somewhere else by Lemma \ref{plumbing:lemma}.
\item[(Ec)] The sequence starts as $(2, x+1, 2, -q, -2, \ldots)$ with $|q|\geq 2$. If $x=0$, 
we use Figure \ref{X11moves:fig}-(6).
\item[(Ed)] The sequence starts as $(2,-q,-2,x, 2, -q', -2, \ldots)$ with $|q|\geq 1$ and $|q'| \geq 2$. If $x=0$ we simplify the graph as in Figure \ref{X11moves:fig}-(6).
\item[(Eg)] The sequence starts as $(2, x+1, 2, -q, -2, \ldots)$ with $|q|\geq 2$. If $x=0$
we use Figure \ref{X11moves:fig}-(5).
\item[(Eh)] We conclude as above.
\item[(Ei)] The sequence starts as $(-3, x, 2, -q, -2, \ldots)$ with $|q|\geq 2$. If $x=0$ this is eqivalent to $(-q+1,-2, \ldots)$. If $-q+1 = -1$ then $q=2$ and we can apply Figure \ref{q2:fig}. Now we use Figure \ref{X10e11moves:fig}-(5).
\item[(Ej)] The sequence starts as $(3, -q, -3, x, 2, -q', -2, \ldots)$ with $|q|\geq 1$ and $|q'|\geq 2$. If $x=0$ the sequence is equivalent to $(3,-q+1, -q'+1, -2, \ldots)$. There are two cases to consider:
\begin{itemize}
\item If $-q+1 = 0$ the sequence is equivalent to $(4-q',-2, \ldots)$. Then, if $q'=4$ or $q'=5$ we deduce that there is a horizontal disc as in Figure \ref{Ej:fig} and then we simplify using Figure \ref{hv:fig} and \ref{X11moves:fig}-(1). 
\item If $-q'+1=-1$ then $q'=2$ and we conclude as in (Ei).
\end{itemize}
\item[(${\rm J}_1$a)] The sequence starts as $(x, 4, -2, \ldots)$ and it must simplify somewhere else.
\item[(${\rm J}_1$b)] Same as above.
\item[(${\rm J}_1$c)] The sequence starts as $(2, x+\frac 12, 4, -2, \ldots)$ and it must simplify somewhere else.
\item[(${\rm J}_1$d)] The sequence starts as $(2, -q, -2, x-\frac 12, 4, -2, \ldots)$ with $|q| \geq 1$. If $x = \frac 12$ this is equivalent to $(2,-q,2,-2,\ldots)$. If $q=-1$ we use Figure \ref{X10e11moves:fig}-(8).
\item[(${\rm J}_1$g)] The sequence starts as $(2, x+\frac 12, 4, -2, \ldots)$ and it must simplify somewhere else.
\item[(${\rm J}_1$h)] Same as above.
\item[(${\rm J}_1$i)] There are two cases to consider depending on the sign $\pm$. If positive, the sequence starts as $(3,x+\frac 12, 4, -2, \ldots)$.
If negative, it starts as $(-3, x-\frac 12, 4, -2, \ldots)$. In both cases it must simplify somewhere else.
\item[(${\rm J}_1$j)] There are two cases to consider depending on the sign $\pm$. 
\begin{itemize}
\item If positive, the sequence starts as $(-3, -q, 3, x+\frac 12, 4, -2, \ldots)$ with $|q|\geq 1$ and it must simplify somewhere else.
\item If negative, it starts as $(3, -q, -3, x-\frac 12, 4, -2, \ldots)$ with $|q|\geq 1$. If $x-\frac 12 = 0$ it is equivalent to $(3, -q, 1, -2, \ldots)$ and hence $(3,-q-1, -3, \ldots)$. If $q=-1$ we get $(0,\ldots)$. We deduce that there is a vertical disc and after Figure \ref{hv:fig} we get a portion as in Figure \ref{J1j:fig}. We conclude as in Proposition \ref{no:flat:Z0:prop}.
\end{itemize}
\item[(${\rm J}_2$a)] The sequence starts as $(x-\frac 12, -2, 4, \ldots)$ and it must simplify somewhere else.
\item[(${\rm J}_2$b)] Same as above.
\item[(${\rm J}_2$c)] The sequence starts as $(2, x, -2, 4, \ldots)$. If $x=0$ we use Figure \ref{X11moves:fig}-(6).
\item[(${\rm J}_2$d)] Similar as above.
\item[(${\rm J}_2$g)] The sequence starts as $(2, x, -2, 4, \ldots)$. If $x=0$ we use Figure \ref{X11moves:fig}-(5).
\item[(${\rm J}_2$h)] Same as above.
\item[(${\rm J}_2$i)] There are two cases to consider depending on the sign $\pm$. If positive, the sequence starts as $(3,x, -2, 4, \ldots)$. If negative, it starts as $(-3, x-1, -2, 4, \ldots)$. In both cases it must simplify somewhere else.
\item[(${\rm J}_2$j)] There are two cases to consider depending on the sign $\pm$. 
\begin{itemize}
\item If positive, the sequence starts as $(-3, -q, 3, x, -2, 4, \ldots)$ with $|q|\geq 1$.
If $x=0$ we get $(-3,-q-1,3,\ldots)$. If $q=-1$ we get $(0,\ldots)$ and a vertical disc, so
we conclude as in (${\rm J}_1$j).
\item If negative, it starts as $(3, -q, -3, x- 1, -2, 4, \ldots)$ with $|q|\geq 1$. If $x=0$ we get $(3,-q,-2,-1,4,\ldots)$ and hence $(3,-q+1,6,\ldots)$ that simplifies somewhere else.
\end{itemize}
\end{itemize}

\begin{figure}
\begin{center}
\includegraphics[width = 3.5 cm]{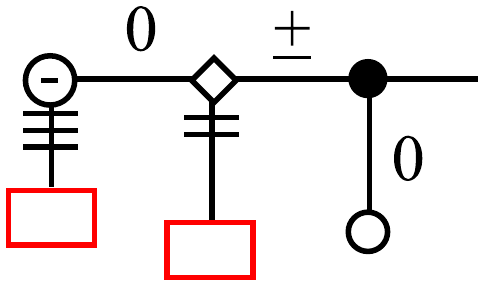}
\nota{In (Ej) if $(q,q') = (1,4)$ or $(1,5)$ we can attach a horizontal disc as shown here.}
\label{Ej:fig}
\end{center}
\end{figure}

\begin{figure}
\begin{center}
\includegraphics[width = 2.5 cm]{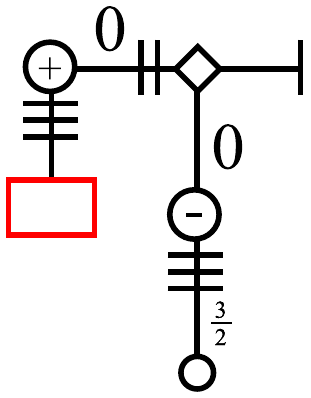}
\nota{In (Ej) if $q=-1$ there is a vertical disc.}
\label{J1j:fig}
\end{center}
\end{figure}

\begin{figure}
\begin{center}
\includegraphics[width = 16 cm]{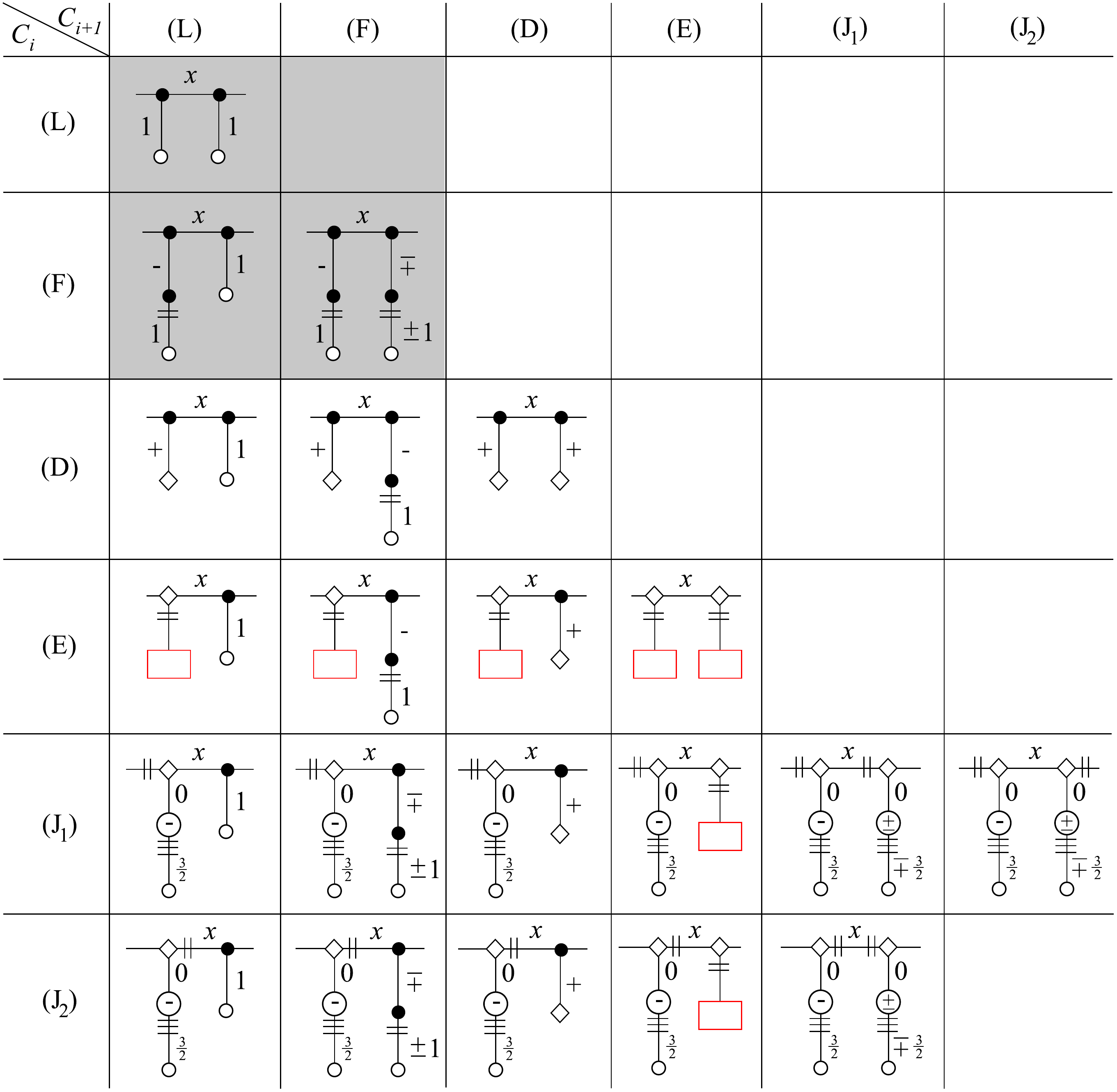}
\nota{Portions obtained as the union of $C_i$ and $C_{i+1}$ (up to changing simultaneously all signs).}
\label{CC:fig}
\end{center}
\end{figure}

The portion comprising $C_k$ and $B$ is treated analogously. We now investigate a portion involving $C_i$ and $C_{i+1}$ as in Figure \ref{CC:fig}. We use Figures \ref{move_leaf:fig}-(1) and \ref{X11moves:fig}-(4) at various points (sometimes implicitly) to modify the sign of the portions (L) and (D).
The cases in the grey box were already covered in \cite{Ma:zero}, and we analyse each of the remaining cases separately.

\begin{itemize}
\item[(DL)] The sequence contains $(\ldots, 2, x+2,\ldots)$. If $x+2 \in \{0,1\}$ then $x\in \{-2,-1\}$. In both cases we can use Figures \ref{move_leaf:fig}-(1) and \ref{X11moves:fig}-(4) to get a configuration where actually $x=0$, so that we can slide (L) above (D) and then simplify using Figure \ref{X11moves:fig}-(3).
\item[(DF)] The sequence contains $(\ldots, 2, x+1, 4, \ldots)$. If $x+1=1$, then $x=0$ and we can slide (D) above (F), and then simplify using Figure \ref{X11moves:fig}-(8). 
\item[(DD)] The sequence contains $(\ldots, 2,x+2,2, \ldots)$. If $x+2 = 1$ then $x=-1$ and after Figure \ref{X11moves:fig}-(4) we can slide one (D) above the other and apply Figure \ref{X11moves:fig}-(9).
\item[(EL)] The sequence contains $(\ldots, 2, -q, -2, x + \frac 12, \ldots)$ with $|q| \geq 2$. If $x+\frac 12 \in \{-1,0\}$, then $x \in \big\{ - \frac 32, - \frac 12 \big\}$ and we simplify using Figure \ref{X11moves:fig}-(3).
\item[(EF)] The sequence contains $(\ldots, 2, -q, -2, x - \frac 12, 4, \ldots )$ with $|q|\geq 2$. If $x - \frac 12 = 0$, then $x=\frac 12$ and we simplify using Figure \ref{X11moves:fig}-(8).
\item[(ED)] The sequence contains $(\ldots, 2, -q, -2, x+\frac 12, 2, \ldots)$ with $|q|\geq 2$. If $x+\frac 12=0$, then $x=-\frac 12$ and we simplify as in Figure \ref{X11moves:fig}-(9).
\item[(EE)] The sequence contains $(\ldots, 2, -q, -2, x-1, -2, -q', 2, \ldots)$ with $|q|, |q'| \geq 2$. If $x-1 = -1$, then $x=0$ and we simplify via Figure \ref{X11moves:fig}-(5).
\item[(${\rm J}_1$L)] The sequence contains $(\ldots, 4, -2, x+\frac 12,\ldots )$. If $x+\frac 12 \in \{-1,0\}$ then $x\in \{-\frac 32, -\frac 12\}$ and we simplify using Figure \ref{X11moves:fig}-(3).
\item[(${\rm J}_1$F)] There are two cases depending on the sign $\pm 1$. 
\begin{itemize}
\item If positive, the sequence contains $(\ldots, 4, -2, x-\frac 12, 4, \ldots)$. If $x-\frac 12 = 0$ we simplify using Figure \ref{X11moves:fig}-(8).
\item If negative, the sequence contains $(\ldots, 4, -2, x-\frac 12, -4, \ldots)$. If $x-\frac 12 = -1$ we simplify using Figure \ref{X11moves:fig}-(8) again.
\end{itemize}
\item[(${\rm J}_1$D)] The sequence contains $(\ldots, 4, -2, x+\frac 12, 2, \ldots)$. If $x+\frac 12 = 0$ we use Figure \ref{X11moves:fig}-(9).
\item[(${\rm J}_1$E)] The sequence contains $(\ldots, 4, -2, x, 2, -q, -2, \ldots)$ with $|q|\geq 2$. If $x=0$ we use Figure \ref{X11moves:fig}-(5).
\item[(${\rm J}_1{\rm J}_1$)] There are two cases depending on the sign $\pm$. 
If positive, the sequence contains $(\ldots, 4, -2, x-\frac 12, -4, 2, \ldots)$, and if negative it contains $(\ldots, 4, -2, x-\frac 12, 4, -2, \ldots)$. In both cases it simplifies somewhere else.
\item[(${\rm J}_1{\rm J}_2$)] There are two cases depending on the sign $\pm$. 
\begin{itemize}
\item If positive, the sequence contains $(\ldots, 4, -2, x, 2, -4, \ldots)$. If $x=0$ we simplify using Figure \ref{X11moves:fig}-(5). 
\item If negative, it contains $(\ldots, 4, -2, x- 1, -2, 4, \ldots)$. If $x-1 =-1$ it is equivalent to $(\ldots, 4, -1, -1, 4, \ldots)$ and hence to $(\ldots,9,\ldots)$. The sequence simplifies somewhere else.
\end{itemize}
\item[(${\rm J}_2$L)] The sequence contains $(\ldots, -2, 4, x+1, \ldots)$. If $x=-1$ we apply Figure \ref{X10e11moves:fig}-(7).
\item[(${\rm J}_2$F)] There are two cases depending on the sign of $\pm 1$. 
\begin{itemize}
\item If positive, the sequence contains $(\ldots, -2, 4, x, 4, \ldots)$ and the sequence simplifies somewhere else. 
\item If negative, the sequence contains $(\ldots, -2, 4, x, -4, \ldots)$. If $x=0$ we simplify using Figure \ref{X10e11moves:fig}-(8).
\end{itemize}
\item[(${\rm J}_2$D)] The sequence contains $(\ldots, -2, 4, x+1, 2,\ldots)$. If $x=0$ we simplify using Figure \ref{X10e11moves:fig}-(6).
\item[(${\rm J}_2$E)] The sequence contains $(\ldots, -2, 4, x+\frac 12, 2, -q, -2, \ldots)$ with $|q|\geq 2$. If $x=\frac 12$ this is equivalent to $(\ldots, -2, 3, 1, -q, -2, \ldots)$ and hence $(\ldots, -2, 2, -q-1, -2, \ldots)$. The sequence simplifies somewhere else. 
\item[(${\rm J}_2{\rm J}_1$)] There are two cases depending on the sign $\pm$. 
If positive, the sequence contains $(\ldots, -2, 4, x, -4, 2, \ldots)$, and if negative it contains $(\ldots, -2, 4, x, 4, -2, \ldots)$. In the positive case, if $x=0$ we simplify using Figure \ref{X10e11moves2:fig}. In the negative case it simplifies somewhere else. 
\end{itemize}

We are left to consider the presence of flat vertices. A flat vertex is just a drilling along some simple closed curve $\gamma$ and does not affect the sequence $(e_1,\ldots, e_n)$, so each simplification move $X\to X'$ as those described above holds, provided that we priorly remove the flat vertex. In all the cases one can verify that we can afterwards regenerate a flat vertex in $X'$ that drill a curve $\gamma'$ homotopic (and hence isotopic) to $\gamma$ and disjoint from $SX'$, so we are done. This holds for all the moves in Figures \ref{X11moves:fig}, \ref{X10moves:fig}, \ref{X10e11moves:fig}, and \ref{X10e11moves2:fig}.
The only exception is the case (kh) where with an additional flat vertex we get the shadow $X_{12}$ shown in Figure \ref{additional:fig}.

The proof of Theorem \ref{main:teo} is now complete.

\end{document}